%% file: 0_main.tex
\documentclass{cams-l} 

\newcommand{\dave}[1]{#1}\newcommand{\davebegin}{}\newcommand{\daveend}{}\newcommand{\george}[1]{#1}

\input{ex_shared}

\newtheorem{theorem}{Theorem}[section]
\newtheorem{lemma}[theorem]{Lemma}
\newtheorem{corollary}[theorem]{Corollary}
\newtheorem{proposition}[theorem]{Proposition}

\theoremstyle{definition}
\newtheorem{definition}[theorem]{Definition}
\newtheorem{example}[theorem]{Example}

\theoremstyle{remark}
\newtheorem{remark}[theorem]{Remark}

\numberwithin{equation}{section}

\begin{document}

\title[A Spectral Theory of Scalar Volterra Equations]{A Spectral Theory of Scalar Volterra Equations}


\author{David Darrow}
\address{Department of Mathematics, Massachusetts Institute of Technology, Cambridge, MA}
\curraddr{}
\email{ddarrow@mit.edu}
\thanks{}

\author{George Stepaniants}
\address{Department of Computing and Mathematical Sciences, California Institute of Technology, Pasadena, CA}
\curraddr{}
\email{gstepan@caltech.edu}
\thanks{}


\date{}

\dedicatory{}

\begin{abstract}
This work aims to bridge the gap between pure and applied research on scalar, linear Volterra equations by examining five major classes: integral and integro-differential equations with completely monotone kernels, \dave{such as linear} viscoelastic models; equations with positive definite kernels, \dave{such as} partially observed quantum systems; difference equations with discrete, positive definite kernels; a generalized class of delay differential equations; and a generalized class of fractional differential equations. We develop a general, spectral theory that provides a system of correspondences between these disparate domains. As a result, we see how `interconversion' (operator inversion) arises as a natural, continuous involution within each class, yielding a plethora of novel formulas for analytical solutions of such equations. \dave{This spectral theory unifies and extends existing results in viscoelasticity, signal processing, and analysis, and makes progress on an open question of Abel regarding the solution of integral equations of the first kind. Finally, it reduces simple Volterra equations of all classes to pen-and-paper calculation, and offers promising applications to the numerical solution of Volterra equations more broadly.}
\end{abstract}

\maketitle

\begin{figure}
    \centering
    \hspace*{-0.4in}
    \resizebox{1.2\textwidth}{!}{
    \begin{tikzpicture}
\node (img) {\includegraphics[width=\linewidth]{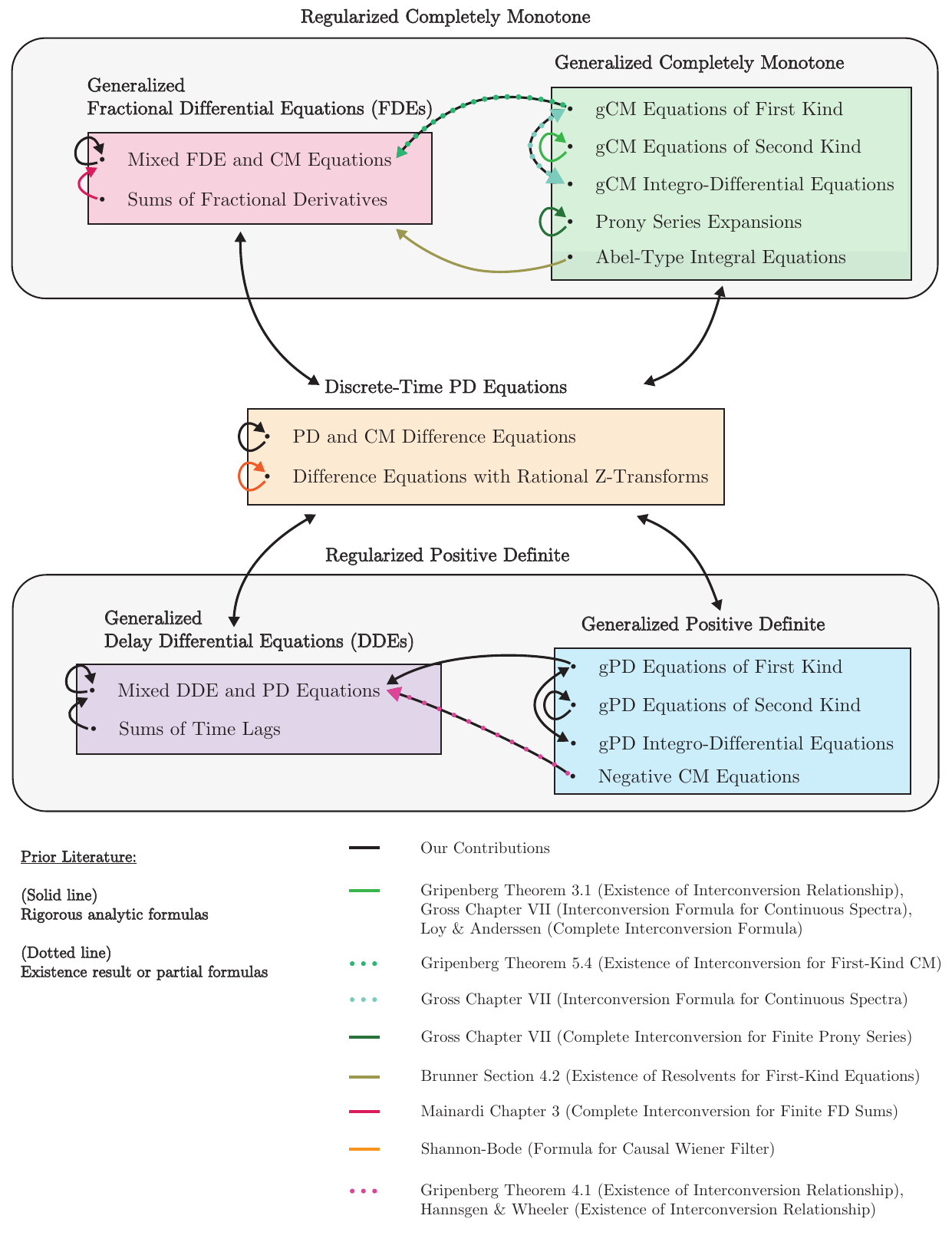}};
\node [below left,text width=3cm,align=center] at (1.67,7.25+0.22){\hyperref[prop:rCM]{$\cB_\mathrm{reg}$}}; 
\node [below left,text width=3cm,align=center] at (1.4,4.61+0.22){\hyperref[ex:abeltype]{$\cB_\mathrm{reg}$}}; 
\node [below left,text width=3cm,align=center] at (1.7,-0.15+0.22){\hyperref[prop:rPD]{$\cB_\mathrm{reg}$}}; 
\node [below left,text width=3cm,align=center] at (1.9,-1.7+0.22){\hyperref[ex:negativeCM]{$\cB_\mathrm{reg}$}}; 
\node [below left,text width=1cm,align=center] at (-5.1,-0.9+0.22){\hyperref[prop:rPD]{$\cB_\mathrm{reg}$}}; 
\node [below left,text width=1cm,align=center] at (-5,6.05+0.22){\hyperref[prop:rCM]{$\cB_\mathrm{reg}$}}; 
\node [below left,text width=3cm,align=center] at (2.05,6.4+0.22){\hyperref[prop:main_CM]{$\cB_\RR$}}; 
\node [below left,text width=3cm,align=center] at (2.19,5.35+0.22){\hyperref[cor:discrete_formula]{$\cB_\RR$}}; 
\node [below left,text width=3cm,align=center] at (2.15,-1.1+0.22){\hyperref[prop:main_PD]{$\cB_\RR$}}; 
\cref{ex:dPD}
\node [below left,text width=3cm,align=center] at (-1.7,1.95+0.22){\hyperref[prop:main_dPD]{$\cB$}}; 
\node [below left,text width=3cm,align=center] at (-1.7,2.47+0.22){\hyperref[prop:main_dPD]{$\cB$}}; 
\node [below left,text width=3cm,align=center] at (4.93,3.85+0.22){\hyperref[eq:Psi_map]{$\Psi$}}; 
\node [below left,text width=3cm,align=center] at (4.65,1.08+0.22){\hyperref[eq:Psi_map]{$\Psi$}}; 
\node [below left,text width=3cm,align=center] at (-1.6,3.85+0.22){\hyperref[eq:embedding_reg]{$\Psi_\mathrm{reg}$}}; 
\node [below left,text width=3cm,align=center] at (-1.62,1.08+0.22){\hyperref[eq:embedding_reg]{$\Psi_\mathrm{reg}$}}; 
\node [below left,text width=3cm,align=center] at (3.65,8.06+0.22){\footnotesize\ref{eq:integrodiff_rCM}};
\node [below left,text width=3cm,align=center] at (7.04,7.45+0.22){\footnotesize\ref{eq:integrodiff_CM}};
\node [below left,text width=3cm,align=center] at (3.34,3.14+0.22){\footnotesize\ref{eq:integrodiff_dPD}};
\node [below left,text width=3cm,align=center] at (3.34,0.90+0.22){\footnotesize\ref{eq:integrodiff_rPD}};
\node [below left,text width=3cm,align=center] at (6.76,-0.02+0.22){\footnotesize\ref{eq:integrodiff_PD}};
\node [below left,text width=3cm,align=center] at (0.92-0.35,-3.56+0.21){\tiny\cite{Gripenberg_Londen_Staffans_1990}};
\node [below left,text width=3cm,align=center] at (0.92-0.35,-3.82+0.21){\tiny\cite{gross1968mathematical}};
\node [below left,text width=3cm,align=center] at (0.92-0.35,-4.08+0.21){\tiny\cite{loy_anderssen}};
\node [below left,text width=3cm,align=center] at (0.92-0.35,-4.32){\tiny\cite{Gripenberg_Londen_Staffans_1990}};
\node [below left,text width=3cm,align=center] at (0.92-0.35,-4.6-0.2){\tiny\cite{gross1968mathematical}};
\node [below left,text width=3cm,align=center] at (0.92-0.35,-5.11-0.2){\tiny\cite{gross1968mathematical}};
\node [below left,text width=3cm,align=center] at (0.92-0.35,-5.62-0.2){\tiny\cite{BRUNNER199783}};
\node [below left,text width=3cm,align=center] at (0.92-0.35,-6.09-0.2){\tiny\cite{mainardi2022fractional}};
\node [below left,text width=3cm,align=center] at (0.92-0.35,-6.6-0.2){\tiny\cite{kailath1980linear}};
\node [below left,text width=3cm,align=center] at (0.92-0.35,-7.14-0.2){\tiny\cite{Gripenberg_Londen_Staffans_1990}};
\node [below left,text width=3cm,align=center] at (0.92-0.35,-7.4-0.2){\tiny\cite{doi:10.1137/0513067}};
\end{tikzpicture}
}
\vspace{-20pt}
    \caption[]{\dave{Our system of correspondences between the five classes of Volterra equations under consideration, with a summary of how it unifies and extends existing results. For detail on particular elements of this correspondence, click the relevant hyperlinks in the figure. For detail on the interconversion maps ($\cB$, $\cB_\RR$, and $\cB_\mathrm{reg}$) and embeddings ($\Psi$ and $\Psi_\mathrm{reg}$) that make up these correspondences, see \cref{sec:main}. For detail on existing literature, see \cref{sec:history}.}}
    \label{fig:map_of_paper}
\end{figure}

$ $
\pagebreak

\renewcommand{\contentsname}{Table of Contents}

\tableofcontents

\vspace{-0.5in}

\renewcommand{\listfigurename}{List of Numerical Experiments}

\let\oldnumberline\numberline
\renewcommand{\numberline}[1]{}
\SkipTocEntry
\listoffigures%
\let\numberline\oldnumberline

\input{1_intro}
\input{2_history}
\input{3_preliminaries}
\input{4_mainresults}
\input{5_laplace}
\input{6_topologies}
\input{7_circle}
\input{8_line}
\input{9_numerics}

\section{Perspectives and Future Directions}
\dave{Although our work covers a broad range of Volterra equations, there remain several interesting directions for future research. For one, \emph{matrix-valued} completely monotone kernels have been studied in some depth by previous authors~\cite{Gripenberg_Londen_Staffans_1990}, and we are optimistic that the perspective offered here might extend such results further. We would also like to develop a better understanding of how interconversion applies to measures with nonzero singular continuous components; such measures are covered by our general theory, but fall outside the scope of our analytical interconversion formulas. Notably, the case of second-kind CM equations has been understood to some degree by Loy \& Anderssen~\cite{loy_anderssen}, by leveraging the operator-theoretic techniques of Aronszajn and Donoghue~\cite{e48c169e-0374-3786-a618-5e3861c40257,https://doi.org/10.1002/cpa.3160180402}}.

Of course, the most restrictive of our hypotheses is that our integral kernels correspond to \emph{non-negative} measures in the spectral domain. Broadly, there are two reasons we need \dave{non-negativity}: to bound the variation norm of the interconverted measure in \cref{lem:firstone}, and to ensure that no poles exist when we take contour integrals in the proof of \cref{thm:main_formula}. If we had \emph{a priori} knowledge of either of these facts (or knowledge of any poles that \emph{do} arise), \dave{the hypothesis of non-negativity may be relaxed}.

\dave{On the applied side, we believe that the basic elements of our spectral theory can be leveraged to solve a broad class of numerical problems outside the present scope. We are presently working on a comprehensive software package, \texttt{Sieve} (\texttt{S}pectral \texttt{I}ntegral transforms, \texttt{E}xponential approximants, and \texttt{V}olterra \texttt{E}quations)~\cite{sieve}, to carry out this program. In short, by extending the tools introduced in \cref{sec:B_implem}, we recover fast, accurate, and noise robust algorithms for several problems of interest: computing continuous and discrete Fourier transforms for arbitrary discontinuous or singular data, competitive with the FFT for smooth data; approximating arbitrary integral kernels with exponential or poly-exponential series; and solving more general classes of Volterra equations. }


\section*{Acknowledgments}
DD would like to thank John Bush and Glenn Flierl (MIT) for allowing him to spend so much time on work with such tenuous connections to his doctoral research. He would also like to thank Glenn Flierl and Paolo Giani (MIT) for helpful discussions regarding the geophysical applications of scalar Volterra equations, \dave{and David Jerison (MIT) for helpful discussions on the broader mathematical context surrounding our work}. Finally, DD would like to acknowledge the support of an NDSEG Graduate Fellowship.

GS would like to thank Andrew Stuart and Kaushik Bhattacharya (Caltech) for helpful discussions on dynamics with memory in the contexts of the Mori--Zwanzig formalism and in applications to materials science, as well as to Lianghao Cao and Margaret Trautner (Caltech) for insightful discussions. He is particularly grateful to Andrew Stuart for encouraging mathematical exploration in this direction considering the well-established history of Volterra integral equations and viscoelastic material models. GS is supported by an NSF Mathematical Sciences Postdoctoral Research Fellowship (MSPRF) under award number 2402074.

\bibliographystyle{amsplain}
\bibliography{refs}

\end{document}

%% file: ex_shared.tex
\usepackage{lipsum}
\usepackage{subfigure}
\usepackage{amsfonts}
\usepackage{graphicx}
\usepackage{float}
\usepackage{hyperref}
\usepackage[dvipsnames]{xcolor}
\hypersetup{
    colorlinks=true,
    linkcolor=NavyBlue,
    citecolor=Blue
}
\usepackage{multirow}
\usepackage{epstopdf}
\usepackage{algorithmic}
\usepackage[capitalise]{cleveref}

\crefformat{equation}{(#2#1#3)}
\ifpdf
  \DeclareGraphicsExtensions{.eps,.pdf,.png,.jpg}
\else
  \DeclareGraphicsExtensions{.eps}
\fi

\usepackage{float,mathtools,amsfonts,amssymb,xcolor,empheq,fancybox,booktabs,subcaption,array,longtable}
\usepackage[most]{tcolorbox}
\usepackage{thmtools}
\usepackage{thm-restate}
\usepackage{tikz,tkz-graph,tkz-berge}
\usepackage[ruled,vlined]{algorithm2e}

\usepackage[nocompress]{cite}
\newcommand{\CC}{\mathbb C}

\newcommand{\RR}{\mathbb R}
\newcommand{\ZZ}{\mathbb Z}
\newcommand{\HH}{\mathbb H}
\newcommand{\DD}{\mathbb D}

\newcommand{\cH}{\mathcal H}
\newcommand{\cL}{\mathcal L}

\newcommand{\cM}{\mathcal M}
\newcommand{\cN}{\mathcal N}

\newcommand{\cB}{\mathcal B}

\newcommand{\cC}{\mathcal C}
\newcommand{\cE}{\mathcal E}
\newcommand{\cF}{\mathcal F}

\newcommand{\cW}{\mathcal W}
\newcommand{\cZ}{\mathcal Z}
\newcommand{\cD}{\mathcal D}

\newcommand{\argmax}{\mathrm{arg\,max}}
\newcommand{\argmin}{\mathrm{arg\,min}}

\newcommand{\eps}{\varepsilon}

\newcommand{\ovl}[1]{{\overline{#1}}}

\newcommand{\tsum}{\sum\nolimits}

\newcommand{\op}{\operatorname}
\newcommand{\supp}{\operatorname{supp}}

\newcommand{\pv}{\op{p.\!v.}}

\renewcommand{\Re}{\op{Re}}
\renewcommand{\Im}{\op{Im}}

\DeclareRobustCommand{\gobblefour}[5]{}
\newcommand*{\SkipTocEntry}{\addtocontents{toc}{\gobblefour}}

\usepackage{bbm}

\makeatletter
\renewcommand*\l@figure{\@dottedtocline{1}{1.2em}{2.3em}}
\newcommand\@dotsep{140}
\let\l@table\l@figure
\makeatother

\usetikzlibrary{shapes.geometric, arrows}

\tikzstyle{startstop} = [rectangle, rounded corners, 
minimum width=3cm, 
minimum height=1cm,
text centered, 
draw=black, 
fill=blue!10]

\tikzstyle{io} = [trapezium, 
trapezium stretches=true, 
trapezium left angle=70, 
trapezium right angle=110, 
minimum width=3cm, 
minimum height=1cm, text centered, 
draw=black, fill=blue!30]

\tikzstyle{fake} = [trapezium, 
trapezium stretches=true, 
trapezium left angle=70, 
trapezium right angle=110, 
minimum width=0cm, 
minimum height=0cm, text centered, 
draw=black, fill=blue!30]

\tikzstyle{land} = [rectangle, 
minimum width=3cm, 
minimum height=1cm, 
text centered, 
text width=3cm, 
draw=black, 
fill=green!30!black!20]

\tikzstyle{ocean} = [rectangle, 
minimum width=3cm, 
minimum height=1cm, 
text centered, 
text width=3cm, 
draw=black, 
fill=cyan!20!]

\tikzstyle{decision} = [diamond, 
minimum width=3cm, 
minimum height=1cm, 
text centered, 
draw=black, 
fill=black!10]
\tikzstyle{arrow} = [thick,->,>=stealth]


%% file: 1_intro.tex
\davebegin
\section{Introduction}\label{sec:intro}
We study five classes of convolution equations. The first is the class of \emph{generalized completely monotone} (gCM) integral and integro-differential equations:
\begin{equation}\tag*{\textbf{(gCM)}}\label{eq:integrodiff_CM}
    y(t) = c_1\dot{x}(t) - c_0x(t) - \int_{0}^t K(t-\tau)x(\tau)\,d\tau,\qquad x(0) = x_0\;\;\text{(if }c_1\neq 0\text{)},
\end{equation}
where $c_0,c_1\in\RR$ with $c_1\geq 0$, the source term $y$ is a sufficiently regular\footnote{Sufficient conditions to yield a classical solution $x$ will be made clear in the theory that follows.} function on $\RR_+\doteq [0,\infty)$, and $K$ is a gCM kernel:
\begin{definition}\label{def:CMkernel}
    A smooth, non-negative $F:\RR_+\to\RR_+$ is \emph{completely monotone} (CM) if $(-1)^jF^{(j)}(t)\geq 0$ for all $j\geq 0$ and all $t>0$. A function $K:\RR_+\to\RR_+$ is \emph{generalized CM} (gCM) if $K(t) = e^{\sigma t}F(t)$ for a CM kernel $F$ and a value $\sigma\in\RR$.
\end{definition}

The second is the class of \emph{generalized positive definite} (gPD) integral and integro-differential equations; notably, these encompass the class \ref{eq:integrodiff_CM} with $c_1<0$:
\begin{equation}\tag*{\textbf{(gPD)}}\label{eq:integrodiff_PD}
	y(t) = c_1\dot{x}(t) -ic_0x(t) + \int_0^t K(t-\tau)x(\tau)\,d\tau,\qquad x(0) = x_0\;\;\text{(if }c_1\neq 0\text{)},
\end{equation}
noting the imaginary factor and the difference in sign with \ref{eq:integrodiff_CM}. Here, $c_1\geq 0$ as before, but we allow any $c_0\in\CC$ with $\Im c_0\geq 0$, and $K$ can be any gPD kernel:
\begin{definition}\label{def:PDkernel}
    A function $F:\RR\to\CC$ is \emph{positive (semi)definite} (PD) if, for all $\{t_1,...,t_N\}\subset\RR$, the matrix $A_{jk} = F(t_j - t_k)$ is positive semi-definite. A function $K:\RR\to\CC$ is \emph{generalized PD} (gPD) if $K(t) = (1 - \frac{d^2}{dt^2})^{1/2}F(t)$ weakly for a PD kernel $F$. It follows from Bochner's Theorem (\cref{lem:bochner} below) that PD kernels are absolutely continuous, and thus that gPD kernels are classical functions.
\end{definition}

The third is the class of \emph{discrete-time positive definite} (dPD) equations:
\begin{equation}\tag*{\textbf{(dPD)}}\label{eq:integrodiff_dPD}
    y(n) = c_0x(n) + \sum_{j=0}^n K(n-j)x(j).
\end{equation}
Here, $c_0\in\CC$ satisfies $\Re c_0\geq -\frac{1}{2}K(0)$, and $K:\ZZ\to\CC$ is \emph{positive (semi)definite}, defined analogously to the continuous case.

The fourth is the class of \emph{regularized PD} (rPD) equations, which generalize the \ref{eq:integrodiff_PD} class to encompass a variety of delay differential equations:
\begin{equation}\tag*{\textbf{(rPD)}}\label{eq:integrodiff_rPD}
	y(t) = c_1\dot{x}(t) + \frac{1}{2}\int_{-t}^t K(\tau)x(|t-\tau|)\,d\tau,\qquad x(0) = x_0\;\;\text{(if }c_1\neq 0\text{)},
\end{equation}
where $c_1\geq 0$ and $K$ is a real\footnote{The extension to complex $K$ can be made with minimal changes to our results.} \emph{rPD} kernel, i.e., $K(t)=(1-\frac{d^2}{dt^2})\widetilde{K}(t)$ weakly for a real PD kernel $\widetilde{K}$.

The fifth and final class is that of \emph{regularized CM} (rCM) equations, which generalize the \ref{eq:integrodiff_CM} class to encompass a wide range of fractional differential equations:
\begin{equation}\tag*{\textbf{(rCM)}}\label{eq:integrodiff_rCM}
    \begin{gathered}
	y(t) = c_1\dot{x}(t) - c_0x(t) - \int_0^t K_1(t-\tau)x(\tau)\,d\tau + \frac{d}{dt}\int_{0}^t K_2(t-\tau)x(\tau)\,d\tau,\\
    x(0) = x_0\;\;\text{(if }c_1\neq 0\text{)},
    \end{gathered}
\end{equation}
where $c_1\geq 0$, $c_0\in\RR$, $K_1$ is a gCM kernel, and $K_2$ is a CM kernel.

In the present work, we propose a unified spectral theory for these five classes of equations, solving a number of seemingly-disparate problems. At the broadest level, we develop a system of correspondences between these classes (summarized in \cref{fig:map_of_paper}), allowing insights gained for one class to be transferred to the others. We thus see how `interconversion' (operator inversion) arises as a natural, continuous involution within each class of equations, and we recover rigorous, closed-form solutions for all five classes. Particularly in the most challenging limits---including first-kind integral equations and fractional and delay differential equations---we find a plethora of novel formulas for the analytic solutions of scalar Volterra equations, as well as substantial insight into their behavior. For instance, in the context of \ref{eq:integrodiff_CM}, we will see that $x(t)$ and $y(t)$ satisfy the interconverted gCM equation
\[-\pi^2 x(t) = \zeta_1\dot{y}(t) - \zeta_0y(t) - \int_0^tJ(t-\tau)y(\tau)\,d\tau,\qquad y(0) = y_0\;\;\text{(if }\zeta_1\neq 0\text{)},\]
where $J$, $\zeta_0$, and $\zeta_1$ are calculated analytically through \cref{thm:main_formula}. We highlight several examples in \cref{fig:main_inversion}, corresponding to \cref{ex:dPD,ex:gCM,ex:gPD,ex:gPD_complex,ex:rPD}.

\begin{figure}
    \centering
    \resizebox{\textwidth}{!}{
    \begin{tikzpicture}
\node (img) {\includegraphics[width=\linewidth]{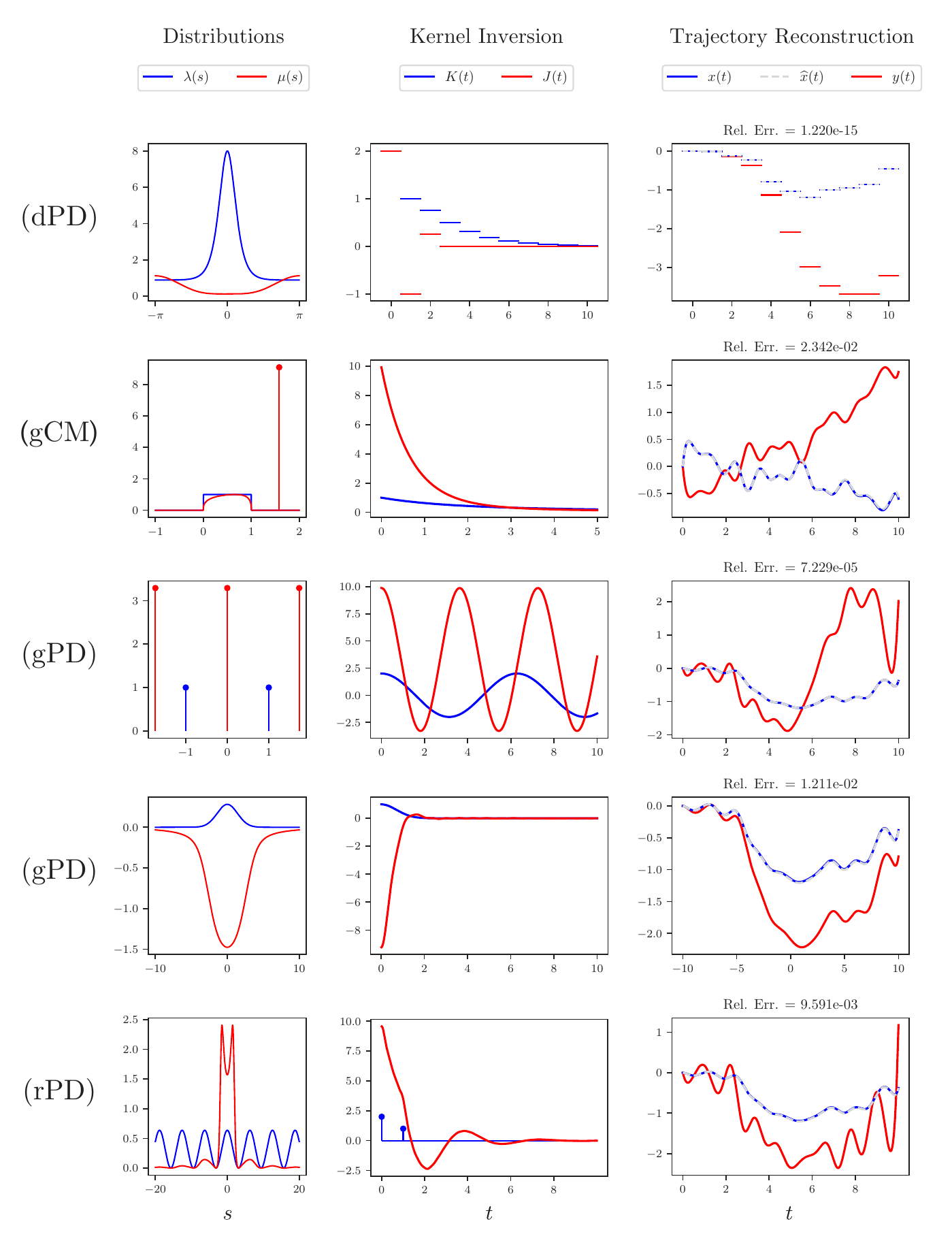}};
\node [below left,text width=3cm,align=center] at (-4,5.05){\small Ex.~\ref{ex:dPD}};
\node [below left,text width=3cm,align=center] at (-4,2.17){\small Ex.~\ref{ex:gCM}};
\node [below left,text width=3cm,align=center] at (-4,-0.75){\small Ex.~\ref{ex:gPD}};
\node [below left,text width=3cm,align=center] at (-4,-3.64){\small Ex.~\ref{ex:gPD_complex}};
\node [below left,text width=3cm,align=center] at (-4,-6.57){\small Ex.~\ref{ex:rPD}};
\end{tikzpicture}
}
    \vspace{-20pt}
    \caption[Simple, Analytic Volterra Interconversion]{Simple examples of four classes of Volterra equations studied in this paper: \ref{eq:integrodiff_dPD}, \ref{eq:integrodiff_CM}, \ref{eq:integrodiff_PD}, and \ref{eq:integrodiff_rPD}. \dave{The first column shows the measures $\lambda$ and $\mu$ that correspond to the spectrum of the original equation and its resolvent, respectively; these measures are defined on $S^1$ for \ref{eq:integrodiff_dPD} and on $\RR$ for the remaining examples. The second column depicts the Volterra integral kernels $K$ and $J$ associated with each spectrum; for instance, in the gCM context, we have $K=\cL[\lambda]$ and $J=\cL[\mu]$. In the third column, we confirm that the predictions of our theory in \cref{ex:dPD,ex:gCM,ex:gPD,ex:gPD_complex,ex:rPD} correctly solve the corresponding Volterra equations. Namely, we show that, given a Volterra equation with kernel $K$, input $x$, and output $y$, the interconverted Volterra equation with kernel $J$ accurately reconstructs the input $\widehat{x}\approx x$ from $y$.}}
    \label{fig:main_inversion}
\end{figure}

Our work unifies, extends, and recontextualizes several existing results from mathematics and applied science. In the context of \ref{eq:integrodiff_CM}, our theory places the classical interconversion formulas of Gross~\cite{gross1968mathematical} on rigorous ground, and it extends the results of Loy \& Anderssen~\cite{loy_anderssen} to yield a new duality between gCM integral equations of the first kind and gCM integro-differential equations. In the context of \ref{eq:integrodiff_rCM}, it greatly generalizes existing interconversion results relating fractional derivatives to \emph{Mittag--Leffler} integral kernels, revealing a general relationship between generalized fractional differential operators and gCM kernels. Moreover, it sheds new light on the work of Hannsgen \& Wheeler~\cite{doi:10.1137/0513067}, who found that integro-differential equations with CM kernels do not necessarily have CM resolvents. They (and other authors~\cite{Gripenberg_Londen_Staffans_1990}) considered only the `negative' CM equation
\begin{equation}\label{eq:negative_CM}
    y(t) = \dot{x}(t) + \int_0^t K(t-\tau) x(\tau)\,d\tau, \qquad x(0)=x_0,
\end{equation}
which interconverts within the class \ref{eq:integrodiff_PD} rather than the class \ref{eq:integrodiff_CM}; consequently, the resolvent is \emph{positive definite} rather than completely monotone. Our work also makes progress on an open question of Abel: given a kernel $K$ and an equation $y(t)=\int_0^tK(t-\tau)x(\tau)\,d\tau$, when does the solution take the form
\[x(t) = \frac{d}{dt}\int_0^t J(t-\tau)y(\tau)\,d\tau\]
for a kernel $J$ depending on $K$? Classical work of Abel offers a solution for the case $K(t) = t^{-\alpha}, \alpha\in(0,1)$, and work of Gripenberg shows that, if $K$ is completely monotone, then there exists a resolvent $J$ that is a completely monotone kernel plus an atom at $t=0$~\cite{Gripenberg1980}. Otherwise, the question remains generally open~\cite{BRUNNER199783}. We answer this question in the affirmative for both gCM and gPD kernels $K$, and offer an explicit formula for both the continuous and discrete parts of $J$.

Our work builds upon a deep and interdisciplinary literature on Volterra equations, with major contributions coming from pure and applied mathematics, science, and engineering. We attempt to summarize key elements of the literature in \cref{sec:history}, and indicate applications of the present work where appropriate. In \cref{sec:preliminaries}, we review the basic results of measure theory and complex analysis that are used in later sections. We describe the main results of our spectral theory in \cref{sec:main}. We handle \ref{eq:integrodiff_dPD} in \cref{sec:main_circ}, \ref{eq:integrodiff_CM} and \ref{eq:integrodiff_PD} in \cref{sec:main_line}, and \ref{eq:integrodiff_rPD} and \ref{eq:integrodiff_rCM} in \cref{sec:regularizedH}. We highlight analytical examples throughout \cref{sec:main}, demonstrating how our work brings simple Volterra equations from all classes within the realm of pen-and-paper calculation.

We prove the correspondences between our five classes of equations in \cref{sec:laplace}, introduce and study appropriate topologies for each class in \cref{sec:topologies}, and develop our spectral theory in \cref{sec:disc,sec:line}. We present the numerical side of our work in \cref{sec:numerics}. Namely, by connecting our theory with the \emph{AAA} rational approximation algorithm~\cite{nakatsukasa2018aaa}, we recover a promising approach to the numerical solution of scalar Volterra equations. We give numerical demonstrations involving a number of practical problems: fast interconversion of Volterra equations, interconversion from noisy time series data and/or sparsely-sampled integral kernels, analysis of quantum search algorithms, and others. We refine and generalize our numerical framework in the sequel~\cite{sieve}, to recover a high order of accuracy and to apply it to problems beyond the present scope.

Our codebase has been made available at the following GitHub link:
\begin{restatable}{equation*}{ourgithublink}
    \parbox{\dimexpr\linewidth-4em}{%
		\strut
		\texttt{\url{https://github.com/sgstepaniants/time-deconvolution}}
		\strut
	}
\end{restatable}

\daveend

\bigskip

\paragraph{\bfseries Note on infinite time horizons}\label{sec:note_initialconds}
In any of our integral or integro-differential equations, one might be interested in an \emph{infinite time horizon}, where we specify homogeneous conditions on the solution in the limit $t\to-\infty$. For instance, \ref{eq:integrodiff_CM} would then take the form
\[y(t) = c_1\dot{x}(t) - c_0x(t) - \int_{-\infty}^t K(t-\tau)x(\tau)\,d\tau.\]
Our results adapt straightforwardly to this setting; in the gCM case, for instance, all that is necessary is replacing the use of the Laplace transform in \cref{sec:laplace} with the \emph{bilateral Laplace transform}
\[\cL_b[y](s) = \int_{-\infty}^\infty y(t) e^{-ts}\,dt.\]
Similar modifications can be carried out for the other classes of equations under consideration.

\bigskip

\paragraph{\bfseries Note on higher-order integro-differential equations} As we \dave{discuss in the following section, CM integral equations of the second kind have been treated by existing literature~\cite{loy_anderssen}. One might wonder, then, could we solve CM integral equations of the first kind  or CM integro-differential equations by integrating or differentiating a second-kind equation appropriately? The answer turns out to be, \emph{sometimes, but not consistently}.}

Suppose we begin with a Volterra equation of the form \ref{eq:integrodiff_CM} with $c_1 = c_0 = 0$ and a CM kernel $K(t)$, and we differentiate (and negate) both sides:
\begin{equation*}
    -\dot{y}(t) = K(0)x(t) + \int_{0}^t \dot{K}(t-\tau)x(\tau)\,d\tau.
\end{equation*}
This equation is now of the second kind, and it follows from \cref{def:CMkernel} that $-\dot{K}(t)$ is CM, and thus that the equation is of the form \ref{eq:integrodiff_CM}. Of course, this procedure requires the additional hypothesis that $K(0)<\infty$, or equivalently, that $\dot{K}$ is locally integrable near $0$. This hypothesis is violated by important examples of CM kernels, such as those corresponding to fractional integrals~\cite{2014fcip.book.....H}.

It turns out that integro-differential CM equations are covered even less completely using this `integration by parts' strategy. Here, we start with a CM equation of the second kind, i.e., with $c_1=0$, $c_0\neq 0$, and a strictly CM kernel $K$, and we differentiate to find
\begin{align*}
    -\dot{y}(t) &= -c_0\dot{x}(t) + K(0)x(t) + \int_{0}^t \dot{K}(t-\tau)x(\tau)\,d\tau \\
    &= \widetilde{c}_1\dot{x}(t) - \widetilde{c}_0x(t) - \int_0^t\widetilde{K}(t-\tau)x(\tau)\,d\tau.
\end{align*}
Once again, it follows from \cref{def:CMkernel} that $\widetilde{K}(t) \doteq -\dot{K}(t)$ is CM. However, from the hypothesis that $K(t)$ itself is CM, we require that
\[0\leq \lim_{t\to\infty} K(t) = K(0) + \int_0^\infty \dot{K}(t)\,dt = -\widetilde{c}_0 - \int_0^\infty \widetilde{K}(t)\,dt,\]
which places substantial requirements on $\widetilde{c}_0$ and $\widetilde{K}$. Roughly, this requires that the contribution from the integral term is dominated by that of the $\widetilde{c}_0$ term.

%% file: 2_history.tex
\davebegin
\section{Prior Work}\label{sec:history}
Scalar Volterra equations have been studied from several perspectives, and much is already known about the solution of such equations. In the present section, we attempt to give a broad perspective of the literature surrounding Volterra equations, and indicate how the present work fits into this larger context. 

First and foremost, we note that the Volterra equations under consideration contain several particular subclasses of importance, which have historically been treated semi-independently. The classes \ref{eq:integrodiff_CM} and \ref{eq:integrodiff_PD} split naturally into three subclasses: integral equations \emph{of the first kind}, when $c_0=c_1=0$; integral equations \emph{of the second kind}, when $c_1=0$ but $c_0\neq 0$; and \emph{integro-differential} equations, when $c_1 \neq 0$. These three subclasses share many of the same physical applications; for instance, as we discuss shortly, CM equations corresponding to linear viscoelastic models can fall into any of these three subclasses. The \ref{eq:integrodiff_rCM} class also contains a wide range of \emph{delay differential equations}, of the form
\[y(t) = c_1\dot{x}(t) + \sum\nolimits_i b_i x(t-t_i).\]
Delay differential equations have found broad applications in biology, such as in the study of gene networks and neuron models~\cite{rihan2021delay}. Likewise, the \ref{eq:integrodiff_rPD} class encompasses many \emph{fractional differential equations}; given $\alpha\in(0,1)$, the \emph{Riemann--Liouville fractional derivative} is the integral operator~\cite{2014fcip.book.....H}
\begin{equation}\label{eq:frac_deriv}
    D^\alpha:f\mapsto \frac{1}{\Gamma(1-\alpha)}\,\frac{d}{dt}\int_0^t\frac{f(\tau)}{(t-\tau)^\alpha}\,d\tau.
\end{equation}
Fractional differential equations have been used to model anomalous diffusion processes~\cite{2000PhR...339....1M}, complex media~\cite{2011ASAJ..130.2195H}, and ladder models in materials science~\cite{gross1956ladder,gross1956ladder2}.

Now, classical formulas exist to solve various limits of these equations. Integral equations involving finite sums of exponentials (or `Prony series') have been independently solved in classical analysis~\cite{whittaker1918numerical, polyanin2017integral}, in signal processing~\cite[Sec.~7.5]{kailath1980linear}, and in the theory of viscoelastic materials~\cite{gross1968mathematical, anderssen2006interconversion, loy2015interconversion, hristov2018linear, bhattacharya2023learning}. Classical formulas also exist for equations involving finite sums of fractional derivatives~\cite{mainardi2022fractional}, used in the context of \emph{ladder models} in materials science~\cite{gross1956ladder,gross1956ladder2}. More recently, the work of Loy \& Anderssen~\cite{loy_anderssen} formalized a classical interconversion formula of Gross~\cite{gross1968mathematical} for CM integral equations of the second kind, which our work recovers in the appropriate limit. In turn, Loy \& Anderssen leveraged the operator-theoretic results of Aronszajn~\cite{e48c169e-0374-3786-a618-5e3861c40257} and Donoghue~\cite{https://doi.org/10.1002/cpa.3160180402}, which we describe in \cref{sec:operatortheory}.

Linear Volterra equations are also covered by broader existence and uniqueness results, both classical and recent~\cite{Gripenberg_Londen_Staffans_1990}. In the case of CM integral equations of the second kind, it has long been known that the resolvent of the equation is another CM kernel, and in the integro-differential case (but with the sign $c_1<0$), the results of Hannsgen and Wheeler~\cite{doi:10.1137/0513067} show that the resolvent differs only from a CM kernel by an exponentially-decaying function. 

Some other aspects of our theory also have strong precedents in the literature. It is well-known that certain classes of discrete- and continuous-time convolution equations can be brought into correspondence~\cite{Nikolski_2020}, for instance, though we carry the program further for Volterra equations in the present work. Moreover, the regularized Hilbert transform we use to do so has previously been constructed in the context of Calder\'on--Zygmund theory~\cite{10.1007/BF02392130} and in the context of rank-one perturbations of linear operators~\cite{Albeverio2004}.

The literature surrounding Volterra equations is spread across several areas of mathematics, science, and engineering, each of which has contributed to our current understanding of the subject. We attempt to outline these various threads in greater depth in the following subsections.

\subsection{Linear Time-Invariant Systems}\label{sec:LTI}
One application of completely monotone kernels is provided by partially-observed \emph{linear time-invariant} (LTI) systems. Suppose a vector quantity $\mathbf{q}=(q_0,q_1,...,q_N)$ evolves according to the system
\begin{equation}\label{eq:LTI}
    \dot{\mathbf{q}}(t) = -\mathbf{M}\mathbf{q}(t) + \mathbf{f}(t),
\end{equation}
where $\mathbf{M}$ is a positive semi-definite matrix\footnote{This hypothesis is not strictly necessary. One example relevant to carbon reservoir models (\cref{fig:co2}) is, if $M_{0k} = 0$ for $k\geq 2$, it turns out to be sufficient that the submatrix $[M_{ij}]_{1\leq i,j\leq N}$ is a product of a positive semi-definite matrix and a positive definite diagonal matrix.
} and $\mathbf{f}(t)$ is a time-dependent forcing term. In many applications, one is only able to observe the value of one element of $\mathbf{q}$, say, $q_0$. Formally solving \cref{eq:LTI} in terms of $q_0$, we can rewrite
\begin{align*}
    \dot{q}_0(t) = -\lambda q_0(t) + \int_0^t K(t-\tau)q_0(\tau)\,d\tau + g(t),
\end{align*}
where $K(t)$ is an integral kernel dependent only on $\mathbf{M}$ and $g$ is a modified forcing term dependent on $\mathbf{f}$ and on the initial values of $\mathbf{q}$. This program---an example of the more-general Mori--Zwanzig formalism~\cite{zwanzig2001nonequilibrium, givon2004extracting}---reduces the system \cref{eq:LTI} to model the self-interaction of $q_0$ as mediated by the other elements of $\mathbf{q}$. So long as $\mathbf{M}$ is positive semi-definite, this equation is of the form \ref{eq:integrodiff_CM}. As we show later, the LTI example is highly general; \emph{any} CM equation can be approximated to arbitrary precision by finite-dimensional LTI models of this form (see \cref{sec:topologies}).

In \cref{fig:co2}, we highlight how LTI systems are used to construct \emph{reservoir models} for carbon transport~\cite{Keeling1973}. In this example, $q_0$ would represent the carbon budget of the atmosphere, $q_{k\neq 0}$ would represent the carbon budget of other reservoirs, and $\mathbf{f}$ would represent the rate of emission from each reservoir into the atmosphere. 

\begin{figure}
    \centering
    \input{tikzfig_1}
    \caption[]{A five-box \emph{reservoir model} for the global carbon cycle~\cite{Keeling1973}. In such models, large environmental reservoirs of CO$_2$ are hypothesized to be well-mixed, such that the transport of CO$_2$ between them is determined by the total quantity in each. Such models have been in use since the 1950s~\cite{https://doi.org/10.1111/j.2153-3490.1957.tb01848.x,Keeling1973}, with subsequent developments introducing more reservoirs~\cite{https://doi.org/10.1111/j.2153-3490.1967.tb01509.x,Bjorkstrom1986}, refined diffusion effects~\cite{https://doi.org/10.1111/j.2153-3490.1975.tb01671.x}, and more. They have seen extensive use in understanding anthropogenic effects on the global carbon cycle; for instance, they have been used recently to estimate historical carbon budgets~\cite{KAMIUTO1994825} and the impacts of radiative forcing~\cite{CHOI2022424} and of burning biomass~\cite{CHOI2020106942} on global temperatures.}
    \label{fig:co2}
\end{figure}

Following a similar argument, one can see that the class \ref{eq:integrodiff_PD} corresponds to partially-observed quantum systems. Namely, fix a Hilbert space $\cH$, a (self-adjoint) Hamiltonian $\hat{H}$, and a state $|0\rangle\in\cH$, and decompose our wavefunction $\psi\in\cH$ as
$\psi = \phi_0|0\rangle + \phi_1$. If we write $\hat{P} = 1-|0\rangle\langle 0|$ and
\[H_0 = \langle 0|\hat{H}|0\rangle,\qquad \hat{H}_1 = \hat{P}\hat{H}\hat{P},\]
then it is straightforward to show that $\phi_0$ satisfies the Nakajima--Zwanzig equation~\cite{nakajima1958quantum, zwanzig1960ensemble}
\begin{equation}\label{eq:quantum_intro}
    \dot{\phi}_0(t) + iH_0\phi_0(t) + \int_0^t \langle 0|\hat{H}e^{-i\hat{H}_1(t-\tau)}\hat{H}|0\rangle\phi_0(\tau)\,d\tau = -i\langle 0|\hat{H}e^{-i\hat{H}_1t}|\phi_1(0)\rangle.
\end{equation}
Since $\hat{H}_1$ is self-adjoint, this equation falls into the class \ref{eq:integrodiff_PD} with
\[c_1=1,\qquad c_0 = -H_0,\qquad K(t) = \langle 0|\hat{H}e^{-i\hat{H}_1t}\hat{H}|0\rangle,\]
and forcing determined by the initial conditions of $\hat{P}\psi$. 

We investigate a practical example of a partially-observed quantum system in \cref{sec:quantumwalk}, where we show how our interconversion results allow one to maximize the probability of success of a basic quantum search algorithm.

\begin{figure}
    \centering
    \includegraphics[width=\linewidth]{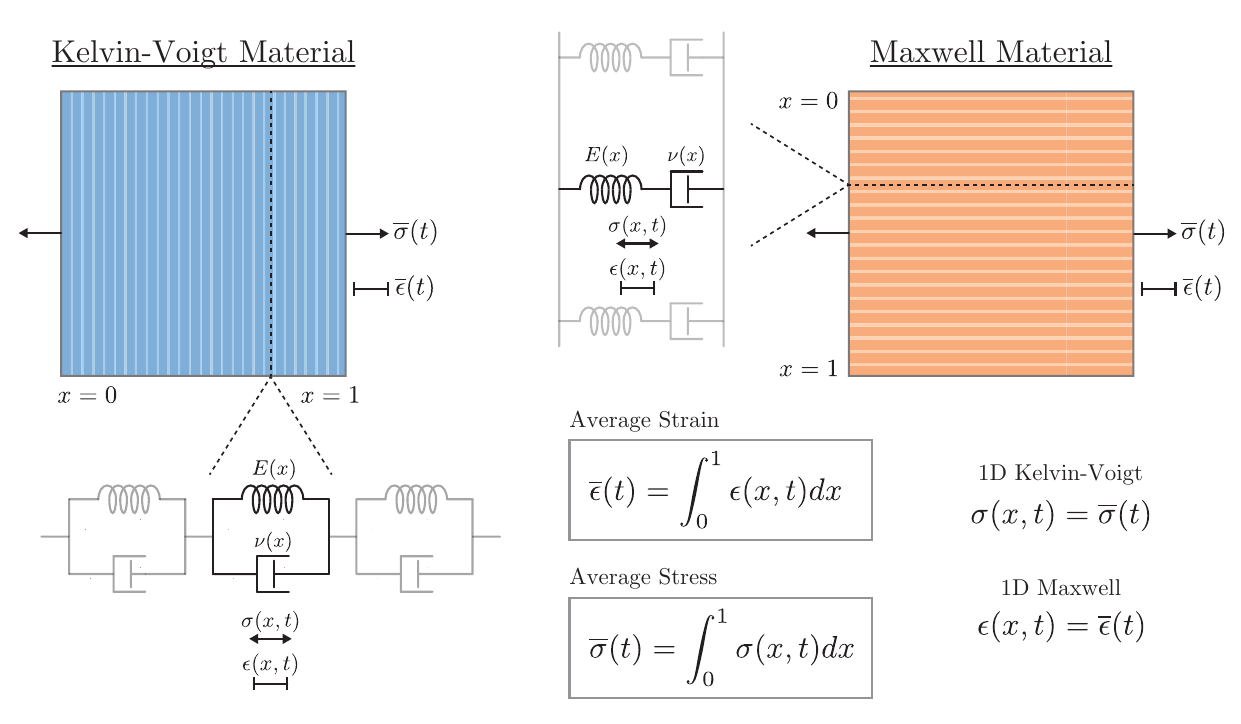}
    \caption[]{The Kelvin--Voigt and Maxwell models of viscoelasticity describe materials as (potentially-infinite) collections of springs and \dave{dashpots}, connected in series or in parallel, respectively. The spring-\dave{dashpot} elements in each model can be indexed by a position variable $x$, giving rise to a position-dependent strain (displacement gradient) $\epsilon(x,t)$ and stress (force gradient) $\sigma(x,t)$. The map from average stress $\ovl{\epsilon}(t)$ to average strain $\ovl{\sigma}(t)$ in a Kelvin--Voigt material is a CM integral equation of either the first or second kind, while for Maxwell materials, the map from average strain to average stress is either a CM integral equation of the second kind or a CM integro-differential equation.}
    \label{fig:materials_fig}
\end{figure}

\daveend

\subsection{Materials Science}\label{sec:materials}
Materials are defined by their response to applied stresses. Elastic solids deform (e.g., strain or shear) under force, but return to their original configuration as soon as the force is removed; viscous fluids resist deformation, but also resist returning from a deformed state. Naturally, then, \emph{viscoelastic} materials give rise to a rich family of stress-strain relations. 

The simplest examples of viscoelastic materials are constructed from springs and dashpots. When a single spring-dashpot pair is connected in parallel, we recover the Kelvin--Voigt model:
\[\sigma(t) = E\epsilon(t) + \nu\dot{\epsilon}(t),\qquad \frac{1}{E+s\nu}\cL[\sigma](s) = \cL[\epsilon](s),\]
written in both the time and Laplace domains. Here, $\sigma$ is the applied stress, $\epsilon$ is the resulting strain, $E$ is the material's elastic modulus, and $\nu$ is its viscosity. 

The Kelvin--Voigt model can be extended straightforwardly to model inhomogeneous media. If we connect Kelvin--Voigt spring-dashpot pairs in series over a continuous interval $x\in[0,1]$, we find
\begin{equation}\label{eq:material_KV}
    \cL[\sigma](x, s) = (E(x) + s\nu(x))\cL[\epsilon](x, s).
\end{equation}
Since our elements are connected in a one-dimensional chain, the applied stress must be constant throughout the material:
\[\sigma(x,t) = \overline{\sigma}(t) \doteq \int_0^1\sigma(x,t)\,dx,\]
so (similarly writing $\overline{\epsilon}(t) = \int_0^1\epsilon(x,t)\,dx$) we can integrate \cref{eq:material_KV} to recover
\begin{equation}\label{eq:KV_laplace}
    \Big(\int_0^1\frac{dx}{E(x) + s\nu(x)}\Big)\cL[\overline{\sigma}](s) = \cL[\overline{\epsilon}](s).
\end{equation}
Note that if $\nu(x)$ is strictly positive, then back in the time domain, the Kelvin-Voigt model corresponds to a {Volterra equation of the first kind} relating average stress to average strain:
\[\ovl{\epsilon}(t) = \int_0^t J(t-\tau)\dot{\ovl{\sigma}}(\tau)\,d\tau = (J*\dot{\ovl{\sigma}})(t),\]
where the \emph{creep compliance} function $J$ is given by
\[J(t) = \int_0^1\frac{1}{E(x)}\left(1 - e^{-\frac{E(x)}{\nu(x)}t}\right)\,dx.\]
The Kelvin--Voigt model is illustrated on the left-hand side of \cref{fig:materials_fig}. Notably, if $\nu$ is allowed to vanish anywhere in the domain, the relationship between $\ovl{\epsilon}$ and $\ovl{\sigma}$ becomes a {Volterra equation of the second kind}.

In a different direction, we could start with a single spring-dashpot pair connected in series, recovering the Maxwell model:
\[\frac{\dot{\sigma}(t)}{E} + \frac{\sigma(t)}{\nu} = \dot{\epsilon}(t),\qquad \left(\frac{s}{E}+\frac{1}{\nu}\right)\cL[\sigma](s) = s\cL[\epsilon](s).\]
The Maxwell model is illustrated on the right-hand side \cref{fig:materials_fig}. By connecting these pairs in parallel over a continuous interval $x\in[0,1]$, now orthogonal to the direction of stress, we obtain
\[\left(\frac{s}{E(x)} + \frac{1}{\nu(x)}\right)\cL[\sigma](x, s) = s\cL[\epsilon](x, s).\]
Now it is the \emph{strain} that must be constant throughout the material, so a similar analysis as above shows that
\[\cL[\overline{\sigma}](s) = \Big(\int_0^1\frac{dx}{\frac{s}{E(x)} + \frac{1}{\nu(x)}}\Big)s\cL[\overline{\epsilon}](s).\]
Back in the time domain, this corresponds to a {Volterra equation of the second kind} relating average strain and stress, so long as $E(x)$ is everywhere finite:
\[\ovl{\sigma}(t) = \int_0^t G(t-\tau)\dot{\ovl{\epsilon}}(\tau)\,d\tau = (G*\dot{\ovl{\epsilon}})(t),\]
with the \emph{relaxation modulus} $G$ defined by
\begin{equation}\label{eq:relaxation}
    G(t) = \int_0^1 E(x)e^{-\frac{E(x)}{\nu(x)}t}\,dx.
\end{equation}
Notably, if we allow $E(x) \to \infty$ anywhere in the domain, this is replaced by a Volterra integro-differential equation relating average strain and stress.

Although $E$ and $\nu$ generally differ between the Kelvin--Voigt and Maxwell models, any linear viscoelastic material should have well-defined, physical values of $\ovl{\sigma}$, $\ovl{\epsilon}$, $J$, and $G$. Moreover, the kernels $J$ and $G$ always satisfy the resolvent (or \emph{interconversion}) formula
\[(G*J)(t) \doteq \int_0^t G(t-\tau)J(\tau)\,d\tau = t,\]
which can be used to uniquely determine one from the other~\cite{ferry1980viscoelastic}. This program was first carried out by Gross to derive analytical formulas relating $J$ and $G$~\cite{gross1952inversion,gross1968mathematical}. Gross' interconversion formulas became a cornerstone of viscoelastic theory~\cite{ferry1980viscoelastic,mainardi2022fractional}, though a formal proof was given only recently by Loy \& Anderssen~\cite{loy_anderssen}, and only for a certain class of materials.

Indeed, as mentioned above, if $E(x)\to\infty$ for any $x$ in the Maxwell model, the expression \cref{eq:relaxation} must be replaced with an integro-differential equation relating $\ovl{\epsilon}$ to $\ovl{\sigma}$. The operator-theoretic techniques leveraged by Loy \& Anderssen (which we return to shortly) do not apply in this case, putting this class of materials outside the scope they studied. Physically, these materials correspond to a Maxwell model where some spring-\dave{dashpot} elements have no springs. Mapping to a Kelvin--Voigt model, this corresponds to a system where $\nu(x)$ vanishes for any $x$, or physically, where some spring-\dave{dashpot} elements have no \dave{dashpots}.

\dave{Mathematically, the work of Loy \& Anderssen allows one to solve CM integral equations of the second kind. Among other applications, the present work extends their results to cover CM integral equations of the first kind and CM integro-differential equations, allowing us to study more general viscoelastic materials.}

Materials science has also inspired a host of other solution methods for particular classes of Volterra (and related) equations. For one, classical results in the field allow for analytic interconversion of finite Prony series~\cite{tschoegl2012phenomenological,bhattacharya2023learning} and finite sums of fractional derivatives~\cite{mainardi2022fractional}, which are key to the ladder models employed by Gross~\cite{gross1956ladder,gross1956ladder2}. We will see that these results, along with the work of Loy \& Anderssen discussed above, form special cases of \dave{the present theory}.

\subsection{Electrical Networks}
Electrical networks are typically built from \dave{three kinds of elements:} resistors (R), which resist the flow of electric current; inductors (L), which oppose changes in current by exchanging energy with a magnetic field; and capacitors (C), which manipulate the flow of current by exchanging energy with an electric field. Mathematically, these elements relate the current $I$ to the voltage $V$ in a \dave{simple circuit} by the equations
\begin{equation}
    V(t) = RI(t), \qquad V(t) = L\dot{I}(t), \qquad V(t) = \frac{1}{C}\int_0^t I(s)ds.
\end{equation}
Arranging these elements in different network configurations allows \dave{one} to achieve a broad array of \emph{transfer functions} that map between current and voltage~\cite{darlington1984history}, and these networks are used in a wide array of applications including signal filtering, audio processing, and communication systems. Writing in the Laplace domain,
\begin{equation}
    \cL[V](s) = R\cL[I](s), \qquad \cL[V](t) = sL\cL[I](s), \qquad \cL[V](s) = \frac{1}{sC}\cL[I](s),
\end{equation}
we see that composing these elements in series or in parallel generally leads to transfer functions with complex poles. This is a fundamental difference from transfer functions in linear viscoelasticity (discussed above), which can only exhibit real poles. As a simple example, RLC circuits are able to form \dave{general \emph{biquadratic filters}---corresponding to rational transfer functions of degree two---which do not generically have real poles~\cite{morelli2019passive}. Consequently, RLC circuits can exhibit oscillatory dynamics, allowing for behaviors such as resonance and phase shifting.}

In practical applications, the use of RLC circuits may be unnecessary if modulation of complex frequencies is not needed, and RC or RL networks built with two of the three components may still offer important lowpass or highpass signal filtering functionality. The transfer functions of RC and RL networks are once again rational functions with real poles, and are therefore identical to the viscoelastic transfer functions described \dave{above}~\cite[Ch.~4]{morelli2019passive}. In fact, even the Kelvin--Voigt and Maxwell models discussed in the previous section have natural analogues in \emph{Foster synthesis}~\cite{bakshi2020electrical}. As such, \dave{our results} are applicable to RC and RL circuits in much the same way as they are to viscoelastic materials.

\subsection{Operator Theory}\label{sec:operatortheory}
With particular choices of parameters---corresponding to the case studied by Loy \& Anderssen~\cite{loy_anderssen}---our problem can be recast in the language of operator theory. Namely, fix a Hilbert space $\cH$, and suppose $A$ is a self-adjoint operator on $\cH$ with simple spectrum $\sigma(A)\subset\RR$. \dave{The spectral theorem guarantees that, for some (non-unique) non-negative measure $\lambda$ on $\sigma(A)$, the operator $A$ is unitarily equivalent to the multiplication operator $M_{s,\lambda}:g(s)\mapsto sg(s)$ on $L^2(\sigma(A),\lambda)$, as $A = U^\dagger M_{s,\lambda} U$. In this context, there is a unique \emph{Borel functional calculus} associated to $A$; for any real-valued Borel function $f$ on $\RR$, there is a unique (generally unbounded) operator 
\[f(A)= U^\dagger M_{f(s),\lambda}U,\]
independent of $\lambda$, with $M_{f(s),\lambda}:g(s)\mapsto f(s)g(s)$ on $L^2(\sigma(A),\lambda)$.
Since the spectrum is simple, we can fix a cyclic vector $v\in\cH$, i.e., such that the subspace
\[\{f(A)v\;|\;f\;\text{bounded and continuous}\}\subset\cH
\]
is dense in $\cH$. The measure $\lambda$ can then be uniquely chosen such that
\[\langle v\;|\;f(A)v\rangle = \int f(s)\,d\lambda(s).\]
We say that $\lambda$ is the \emph{spectral measure} of $A$ corresponding to $v$. }

\dave{Next, we say that $v\in\cH_{-1}(A)$ if 
\[\langle v\;|\;(1 + A^2)^{-1/2}v\rangle = \int (1+s^2)^{-1/2}\,d\lambda < \infty,\]
and in this case, we define the \emph{Borel transform}
\[F:t\mapsto \langle v\;|\;(A-t)^{-1}v\rangle = \int\frac{d\lambda(s)}{s-t}.\]}
The Borel transform provides a natural setting in which to study the spectrum of $A$. In particular, consider the rank-one perturbation
\[A_\alpha\doteq A + \alpha v\langle v\;|\;\cdot\;\rangle\]
for $\alpha\in\RR$, and let $\lambda_\alpha$ be the spectral measure of $A_\alpha$ corresponding to $v$. The Borel transform $F_\alpha$ of $A_\alpha$ is related to $F$ using the Aronszajn--Krein formula:
\[F_\alpha = \frac{F}{1+\alpha F},\]
from which key spectral properties of $A_\alpha$ can be derived. For instance, work of Aronszajn~\cite{e48c169e-0374-3786-a618-5e3861c40257} and Donoghue~\cite{https://doi.org/10.1002/cpa.3160180402} leverages this formula to recover explicit formulas for $\lambda_\alpha$ in terms of $\lambda$, corresponding to our \cref{thm:main_formula} in the case $c_0=-\alpha^{-1}$, $c_1=0$. As one consequence, for $\alpha_1\neq\alpha_2$, they deduce that the point spectra of $A_{\alpha_1}$ and $A_{\alpha_2}$ are disjoint.

The Aronszajn--Donoghue theory has since been extended in a number of directions. Simon and Wolff derived a necessary and sufficient criterion for the perturbations $A_\alpha$ to have pure point spectrum for almost all $\alpha$~\cite{Simon1986SingularCS}, and they showed that the ``almost all'' qualifier cannot be dropped in general. Gordon~\cite{AYaGordon_1992,Gordon1994} and del Rio et al.~\cite{delrio1994singularcontinuousspectrumgeneric,Rio97} (independently) took this analysis one step further, showing that for a wide class of operators $A$, there \dave{are} an uncountable number of $\alpha$ for which $A_\alpha$ has pure singular continuous spectrum. All three groups applied their results to random Hamiltonians, where spectral results can be related to questions of Anderson localization; see the review by Simon~\cite{Simon97} for more details. 

Separately, Gesztezy and Simon explored the strong-coupling limit $\alpha\to\infty$, showing that the (weighted) spectral measures of $A_\alpha$ converge weakly to a measure $\rho_\infty$ on $\RR$, and Albeverio and Koshmanenko~\cite{Albeverio1999} related this limit to the Friedrichs extension of $A$. More recently, Albeverio, Konstantinov, and Koshmanenko~\cite{Albeverio2004} have extended the Aronszajn--Krein relation to the case $v\in\cH_{-2}(A)$, i.e., when it is only known that
\[\langle v\;|\;(1 + A^2)^{-1}v\rangle = \int (1+s^2)^{-1}\,d\lambda < \infty.\]
\dave{Notably, they make use of a \emph{regularized Borel transform} that connects closely to the regularized Hilbert transform of Calder\'on and Zygmund~\cite{10.1007/BF02392130}.} Frymark and Liaw~\cite{Frymark2019SpectralAO} have separately applied Aronszajn--Donoghue-type techniques to explore infinite iterations of rank-one perturbations.

\dave{As a result of our theory, we will see that several results of the Aronszajn--Donoghue theory can be connected closely to the solution of Volterra equations. In particular, we believe that our theory may offer an alternate perspective on the extended Aronszajn--Krein relation for $v\in\cH_{-2}(A)$~\cite{Albeverio2004}.}

\subsection{Signal Processing}\label{sec:signals}
The field of signal processing focuses on the analysis, modification, and synthesis of time-dependent signals, which may be relayed, for instance, as physical waves or electronic signals~\cite{bhattacharyya2018handbook}. A classical problem in the signal processing literature is that of signal deconvolution~\cite{riad1986deconvolution}, which we present here in the discrete-time setting. Given a known filter $K(n)$ and output signal $y(n)$, we aim to determine the input signal $x(n)$ that satisfies the convolution equation
\begin{equation}\label{eq:disc_conv}
    y(n) = \sum_{i=-\infty}^n K(n-i) x(i).
\end{equation}
We can map this problem to the spectral domain by taking a bilateral Z-transform,
\begin{align*}
    \mathcal{Z}_b[x](z) \doteq \sum_{n=-\infty}^\infty x(n)z^n,
\end{align*}
interpreted as a formal power series in $z$. We can likewise define $Y = \mathcal{Z}_b[y]$ and $H = \cZ_b[K]$, defining $K(n) = 0$ for $n<0$. The function $H$ is called the \emph{transfer function} \dave{of the system; assuming $K$ does not grow with time, $H$} is a holomorphic function on the unit disc $\DD$. \dave{In the spectral domain, \cref{eq:disc_conv} becomes
\begin{align*}
    Y(z) = H(z)X(z),
\end{align*}
noting that convolution transforms into pointwise multiplication. For continuous-time deconvolution, a similar equation results from applying the Laplace transform}.

\dave{Continuing in the discrete-time setting, the classical solution to the deconvolution problem is given by the \emph{frequency domain deconvolution formula},
\begin{equation}\label{eq:freq_deconv}
    X(z) = \frac{Y(z)}{H(z)} \doteq G(z)Y(z)
\end{equation}
where $G(z) = 1/H(z)$ is defined wherever $H(z)\neq 0$. If $H$ is analytic and nonzero in $\DD$, then of course, its reciprocal $G$ is analytic and nonzero in $\DD$ as well. The inverse transform $J\doteq \cZ_b^{-1}[G]$ is thus a causal kernel (i.e., with $J(n)=0$ for $n<0$), and
\begin{equation}\label{eq:J_disc}
    x(n) = \sum_{i=-\infty}^n J(n-i)y(i).
\end{equation}
One objective of the present work in the discrete-time setting is to recover a rigorous, closed-form formula for $J$ even in cases where $H$ vanishes on the boundary of $\DD$, which appear in several problems of interest.} 

\dave{Indeed, when $H(z)$ vanishes at or near the boundary of $\DD$, the deconvolution problem becomes \emph{ill-posed}~\cite{riad1986deconvolution, brown2012introduction}, i.e., small errors in $y$ are magnified to become large errors in $x$. As such, instead of studying the exact deconvolution problem discussed above, several regularized variants have been proposed:
\begin{equation}\label{eq:approx_freq_deconv}
    X(z) = \frac{Y(z)}{H(z) + \varepsilon}, \quad X(z) = \frac{\chi_{\{|H(z)| > \varepsilon\}}(z)}{H(z)}Y(z), \quad X(z) = \frac{\overline{H}(z)}{|H(z)|^2 + \varepsilon}Y(z)
\end{equation}
where $\overline{H}$ denotes the complex conjugate and $\eps>0$ is small.} These methods all give rise to different regularized filters $G_\varepsilon(z)$, each of which approximately solves the inverse problem as $X(z) \approx G_\varepsilon(z)Y(z)$. The first two filters listed in \cref{eq:approx_freq_deconv} are pseudoinverse filters and the third is a form of \emph{Tikhonov regularization}, sometimes called the \textit{Wiener deconvolution filter}~\cite{brown2012introduction} if $\varepsilon$ is chosen to scale with the level of noise in $y$. Only the first of the filters in~\cref{eq:approx_freq_deconv} is holomorphic in the unit disc, and hence it is the only filter for which $J_\varepsilon \doteq \cZ_b^{-1}[G_\varepsilon]$ is causal. \dave{For the latter two, an alternate Shannon-Bode construction must be used to enforce causality~\cite{kailath1980linear}}.

A fundamental problem with frequency domain deconvolution is that the regularization in $G_\varepsilon$ biases the estimation of the true inverse filter $G(z)$, potentially leading to large errors in the reconstruction of $x$. Furthermore, the spectral reconstruction $X(z) = G_\varepsilon(z)Y(z)$ is typically evaluated at $N$ equispaced points $z_k = e^{2\pi ik/N}$ on the unit circle and then inverted by the FFT to estimate $x(n)$. This approach is efficient and performs inversion in near-linear time, but \dave{enforces that the reconstruction of $x$ is $N$-periodic}. Furthermore, \dave{standard} FFT inversion performs poorly when $G(z)$ is not a smooth function on the unit circle. \dave{More sophisticated FFT-based algorithms relax this smoothness assumption on $G$, at the cost of several FFT applications and considerable implementation complexity~\cite{commenges1984fast,chandrasekaran2008superfast, LIN2004511}}. In this paper, we develop analytical formulas for the non-regularized inverse transfer function $G(z)$ \dave{even when it is discontinuous or singular, thus mitigating the bias introduced by frequency-domain deconvolution and removing a primary source of error in this ill-posed problem (see \cref{sec:discrete_num}).}

An alternative approach to deconvolution is \textit{time-domain deconvolution}, which directly solves~\eqref{eq:disc_conv} by forming the Toeplitz triangular system
\begin{align*}
    \mathbf{y} = \mathbf{T}\mathbf{x}, \qquad \mathbf{x} = \begin{pmatrix}x(0)\\ \vdots\\ x(n)\end{pmatrix}, \ \mathbf{y} = \begin{pmatrix}y(0)\\ \vdots\\ y(n)\end{pmatrix}
\end{align*}
where $\mathbf{T} \in \RR^{(n+1) \times (n+1)}$ with $T_{ij} = \chi_{\{i \geq j\}}K(i-j)$, and where we have assumed that $x(n) = 0$ for all $n < 0$. Assuming $K(0) > 0$, this system can be solved by computing an inverse (or pseudoinverse) of $\mathbf{T}$, a method referred to as Finite Impulse Response (FIR) Wiener filtering. Performing deconvolution in the time domain alleviates the need to compute spectral properties of noisy signals $y$, as would be required by frequency deconvolution. However, inverting a triangular Toeplitz matrix is most easily done with forward substitution or with Levinson recursion~\cite{wiener1964wiener}, each of which has computational complexity $O(n^2)$. We show in \cref{sec:discrete_num} that \dave{a numerical implementation of our analytical spectral theory yields similar accuracy as time-domain deconvolution, but time complexity competitive with frequency-domain deconvolution}.

\subsection{Numerical Analysis}
Numerical solutions of Volterra equations have been developed for both linear and nonlinear equations. Linear equations can be solved in either the spectral or time domain, and methods for solving such linear equations largely follow the approaches summarized in the signal processing section above.

For linear Volterra equations of the first and second kind, discretization through Newton--Cotes quadrature leads to a triangular Toeplitz system, much like those discussed in the preceding section. Such systems can be inverted through forward substitution, Levinson recursion, or more involved \textit{superfast} methods based on repeated applications of the fast Fourier transforms~\cite{commenges1984fast, chandrasekaran2008superfast, LIN2004511}. For linear integro-differential equations, the convolution kernel can be discretized with Newton--Cotes, Gaussian, or other quadrature schemes, and the resulting delay differential equation can be integrated numerically~\cite{ansmann2018efficiently}. We investigate these methods for solving linear Volterra equations in \cref{sec:numerics}, and we show that analytic interconversion using our general theory is able to match the accuracy of these approaches. In the sequel, we rework our algorithm to achieve high-order, \emph{spectral} accuracy, improving upon the polynomial rate of convergence seen here~\cite{sieve}.

Although not explored in this work, there exist a variety of methods for obtaining numerical solutions of nonlinear Volterra equations. Important classes of algorithms consist of iterative methods based on Picard iteration, series solutions such as the Taylor or Adomian decompositions, analytic conversion into initial value or boundary value problems, direct numerical quadrature for integral equations and time stepping for integro-differential equations, or a combination of these approaches~\cite{linz1985analytical, wazwaz2011linear}.

%% file: tikzfig_1.tex
\begin{tikzpicture}[node distance=2cm]

\node (start) [startstop] {Atmosphere};
\node (emissions) [decision, left of=start, xshift=-2cm,yshift=-2cm] {Emissions};
\node (ocean1) [ocean, below of=start,yshift=0.5cm] {Surface Ocean};
\node (bio) [land, right of=ocean1,xshift=2cm,yshift=0.0cm] {Biosphere};

\node (ocean2) [ocean, below of=ocean1,yshift=0.5cm] {Deep Ocean};
\node (soil) [land, below of=bio,yshift=0.5cm] {Soil};

\node (fake) [left of=bio] {};
\node (fake2) [left of=bio,yshift=0.3cm] {};

\draw [arrow] (emissions) |- (start);

\draw [arrow,<->] (start) -- (ocean1);
\draw [arrow,<->] (ocean1) -- (ocean2);
\draw [arrow,<->] (start) -| (bio);
\draw [arrow] (bio) -- (soil);
\draw [arrow,-] (start) -| (fake);
\draw [arrow,->] (fake2) |- (soil);

\end{tikzpicture}

%% file: 3_preliminaries.tex
\section{Preliminaries}\label{sec:preliminaries}
We largely study our Volterra equations under the action of various integral transforms, where they can be related to questions of measure theory. As a starting point, we introduce the following notation:
\begin{definition}[Sets of Measures]\label{def:measures}
    Let $\cM_\mathrm{loc}(\RR)$ and $\cM(S^1)$ be the spaces of signed Borel measures of locally bounded variation on $\RR$ and on $S^1$, respectively. We define the following subsets of each:
    \begin{enumerate}
        \item Let $\cM(\RR)\subset\cM_\mathrm{loc}(\RR)$ be the subspace of finite measures.
        \item Let $\cM_{+,\mathrm{loc}}(\RR)\subset\cM_\mathrm{loc}(\RR)$, $\cM_+(\RR)\subset\cM(\RR)$, and $\cM_+(S^1)\subset\cM(S^1)$ be the subsets of non-negative measures, excluding the zero measure.
        \item Let $\cM_c(\RR)\subset\cM_+(\RR)$ be the subset of non-negative, compactly supported measures. 
        \item For any $n\in\RR$, let $\cM^{(n)}(\RR)\subset\cM_\mathrm{loc}(\RR)$ be the subspace of measures $\lambda$ on $\RR$ such that $\int (1+s^2)^{-n/2}\,|d\lambda(s)|<\infty$. In particular, $\cM^{(0)}(\RR)=\cM(\RR)$.
        \item Let $\cM_+^{(n)}(\RR)=\cM_{+,\mathrm{loc}}(\RR)\cap \cM^{(n)}(\RR)$. In particular, $\cM_+^{(0)}(\RR)=\cM_+(\RR)$.
        \item Let $\cM_\mathrm{exp}^{(n)}(\RR)$ be the set of measures $\lambda\in\cM_+^{(n)}(\RR)$ with $\inf\supp\lambda>-\infty$.
    \end{enumerate}
\end{definition}
The notation $\cM_\mathrm{exp}^{(n)}(\RR)$ is inspired by the fact that, for any $\lambda\in\cM_\mathrm{exp}^{(n)}(\RR)$, the bilateral Laplace transform 
\[\cL_b[\lambda](t)\doteq \int e^{-\sigma t}\,d\lambda(\sigma)\]
is \dave{at most exponentially growing as} $t\to\infty$. Also note that $\cM_+^{(n)}(\RR)\subset\cM_+^{(m)}(\RR)$ and $\cM_\mathrm{exp}^{(n)}(\RR)\subset\cM_\mathrm{exp}^{(m)}(\RR)$ for $n\leq m$. To make contact between the theory on the circle and the theory on the real line, we make use of the embedding $\psi:\cM^{(2)}(\RR)\to \cM(S^1)$ given by
\begin{equation}\label{eq:embedding}
    d\lambda(s) \doteq \pi(1+s^2)\,\phi_*\psi[d\lambda](s),
\end{equation}
where
\begin{equation}\label{eq:cayleymap}
    \phi:z\mapsto i\,\frac{1-z}{1+z}, \qquad \phi^{-1}:w \mapsto \frac{i - w}{i + w},
\end{equation}
are Cayley maps between the unit disc and the upper half-plane. \dave{In particular, if $\lambda = f(s)\,ds$ is absolutely continuous with respect to the Lebesgue measure $ds$, then
\[\psi[f(s)\,ds] = \frac{1}{2\pi}(f\circ\phi)(\theta)\,d\theta,\]
where $d\theta$ is the Lebesgue measure on $S^1$.}

Such measures provide a helpful dual language for all three of \ref{eq:integrodiff_CM}, \ref{eq:integrodiff_PD}, and \ref{eq:integrodiff_dPD}, albeit, in slightly different ways; we return to the more-involved classes \ref{eq:integrodiff_rPD} and \ref{eq:integrodiff_rCM} in \cref{sec:regularizedH}.

\begin{lemma}[Bernstein~\cite{Widder1931NecessaryAS,Gripenberg_Londen_Staffans_1990}]\label{lem:bernstein}
    A kernel \dave{$K:\RR_+\to\RR_+$} is generalized-completely-monotone if and only if
    \[
        K(t) = \int e^{-\sigma t}\,d\lambda(\sigma)
    \]
    for a non-negative Borel measure $\lambda$ with $\inf\supp\lambda >-\infty$. We write $\lambda = \cL_b^{-1}[K]$ and $\cL_b[\lambda]=K$ for the (bilateral) Laplace transform in this context.
\end{lemma}
\begin{lemma}[Bochner~\cite{reed1975ii}]\label{lem:bochner}
    A kernel $K:\RR\to\CC$ is positive definite if and only if it is the Fourier transform of a measure $\lambda\in\cM_+(\RR)$, and generalized-positive-definite if and only if it is the Fourier transform of a measure $\lambda\in\cM_+^{(1)}(\RR)$. In the positive definite case (for which $K(0)=\|\lambda\|<\infty$), $K$ takes the familiar form
    \[K(t) = \cF[\lambda](t) \doteq \int e^{-i\omega t}\,d\lambda(\omega).\]
    We write $\lambda = \cF^{-1}[K]$ for the (inverse) Fourier transform.

    Likewise, a kernel $K:\ZZ\to\CC$ is positive definite if and only if it is the (discrete) Fourier transform of a measure $\lambda\in\cM_+(S^1)$:
    \[K(n) = \int_0^{2\pi} e^{-in\theta}\,d\lambda(\theta).\]
    We apply the same notation for the Fourier transform in this context.
\end{lemma}

As such, we can reduce all three classes of equations to the common language of non-negative measures. In turn, we largely study these measures by extending them to holomorphic functions:
\begin{definition}[Integral Transforms on $S^1$]
    For any $\lambda\in\cM(S^1)$, we define the \emph{Cauchy transform}
\begin{equation}\label{eq:cauchycirc}
    Q[\lambda](z)\doteq \int_{0}^{2\pi}\frac{1+e^{-i\theta}z}{1-e^{-i\theta}z}\,d\lambda(\theta),
\end{equation}
viewed as a holomorphic map on the open unit disc $\DD\subset\CC$. The real part of $Q[\lambda](z)$ is known as the \emph{Poisson integral},
\begin{equation}\label{eq:poissoncirc}
    P[\lambda](re^{i\theta}) \doteq \Re Q[\lambda](re^{i\theta}) = \int_0^{2\pi} \frac{1-r^2}{1-2r\cos(\theta-\theta')+r^2}\,d\lambda(\theta'),
\end{equation}
and the imaginary part is the \emph{conjugate Poisson integral}:
\begin{equation}\label{eq:poissoncirc_conj}
    \Im Q[\lambda](re^{i\theta}) = \int_0^{2\pi} \frac{2r\sin(\theta-\theta')}{1-2r\cos(\theta-\theta')+r^2}\,d\lambda(\theta').
\end{equation}
\end{definition}
These integral transforms can be seen to be isometries of appropriate function spaces. To see this, we define the following harmonic Hardy spaces on the disc:
\begin{definition}
    Suppose $h$ is a real harmonic function on $\DD$, and fix $0<p\leq\infty$. We say that $h\in h^p(\DD)$ if, for all $r<1$, the circular traces $e^{i\theta}\mapsto h(re^{i\theta})$ are uniformly bounded in $L^p(S^1)$. We define $\|h\|_{h^p}=\sup_r \left(\frac{1}{2\pi}\int |h(re^{i\theta})|^p\,d\theta\right)^{1/p}$.
\end{definition}
In this language, we have the following classical results:
\begin{proposition}\label{prop:classic}
    The following classical results are established, for instance\footnote{Respectively, these correspond to Thm.~6.13a, Thm.~6.13b, Thm.~6.9, and Cor.~6.44.}, in Axler et al.~\cite{axler2013harmonic}:
    \begin{enumerate}
        \item The Poisson kernel $P:\lambda\to P[\lambda]$ is an isometry from $\cM(S^1)$ (with variation norm) to $h^1(\DD)$ (Herglotz--Riesz)~.\label{prop:classic1}
        \item If $1<p\leq \infty$, the map $P:f\mapsto P[(2\pi)^{-1}f(e^{i\theta})\,d\theta]$ is an isometry from $L^p(S^1)$ to $h^p(\DD)$.\label{prop:classic2}
        \item If $\lambda\in \cM(S^1)$, the measures $\lambda_r\doteq (2\pi)^{-1}P[\lambda](re^{i\theta})\,d\theta$ converge weakly to $\lambda$ as $r\to 1$.\label{prop:classic3}
        \item If $\lambda\in \cM(S^1)$, and $\lambda_c\in L^1(S^1)$ is the density of its continuous part with respect to the normalized Lebesgue measure $(2\pi)^{-1}d\theta$, as furnished by the Lebesgue decomposition~\cite{rudin1974real}, then $P[\lambda]$ has non-tangential limit $\lambda_c$ almost everywhere in $S^1$.\label{prop:classic4}
    \end{enumerate}
\end{proposition}
In particular, we make use of the following corollary:
\begin{corollary}[Shifted Cauchy and Hilbert Transforms]\label{cor:uniqueness}
    For any $\lambda\in\cM(S^1)$ and $\sigma\in\RR$, there is a unique holomorphic function $Q_\sigma[\lambda](z)$ on $\DD$ such that $\Im Q_\sigma[\lambda](0) = \sigma$ and such that the measures $\lambda_r\doteq (2\pi)^{-1}\Re Q_\sigma[\lambda](re^{i\theta})\,d\theta$ \dave{have uniformly bounded variation $\|\lambda_r\|$ and} converge weakly to $\lambda$ as $r\to 1$. This function is given by
    \begin{equation}\label{eq:cauchygeneral}
        Q_\sigma[\lambda](z) = Q[\lambda](z) + i\sigma = \int_{0}^{2\pi}\frac{1+e^{-i\theta}z}{1-e^{-i\theta}z}\,d\lambda(\theta) + i\sigma.
    \end{equation}

    We say that $Q_\sigma[\lambda]$ is the $\sigma$-Cauchy transform of $\lambda$, and we define the $\sigma$-Hilbert transform to be its imaginary trace along $S^1$:
    \begin{equation}\label{eq:H_sigma}
        H_\sigma[\lambda](e^{i\theta}) = \lim_{r\to 1}\Im Q[\lambda](re^{i\theta}),
    \end{equation}
    well-defined almost everywhere in $S^1$~\cite{axler2013harmonic}. We set $H[\lambda] = H_0[\lambda]$.
\end{corollary}
\begin{remark}
    \dave{We make particular use of the case $\lambda\in\cM_+(S^1)$, in which case it is known \emph{a priori} that $\|\lambda_r\|=\|\lambda\|=\Re Q_\sigma[\lambda](0)$ for all $r<1$, and it follows from the maximum principle that $0<\Re Q_\sigma[\lambda](z)<\infty$ for all $z\in\DD$. The $\sigma$-Cauchy and $\sigma$-Hilbert transforms are illustrated in \cref{fig:H_Q_sigma}.}
\end{remark}
\begin{proof}
    \dave{\Cref{prop:classic} makes clear that the real part of $P[\lambda]=\Re Q_\sigma[\lambda]$ is uniquely defined; the lemma follows by noting that the harmonic conjugate $\Im Q_\sigma[\lambda]$ of $P[\lambda]$ is unique up to a constant term~\cite{conway1978functions}.}
\end{proof}

\begin{figure}
    \centering
    \includegraphics[width=\linewidth]{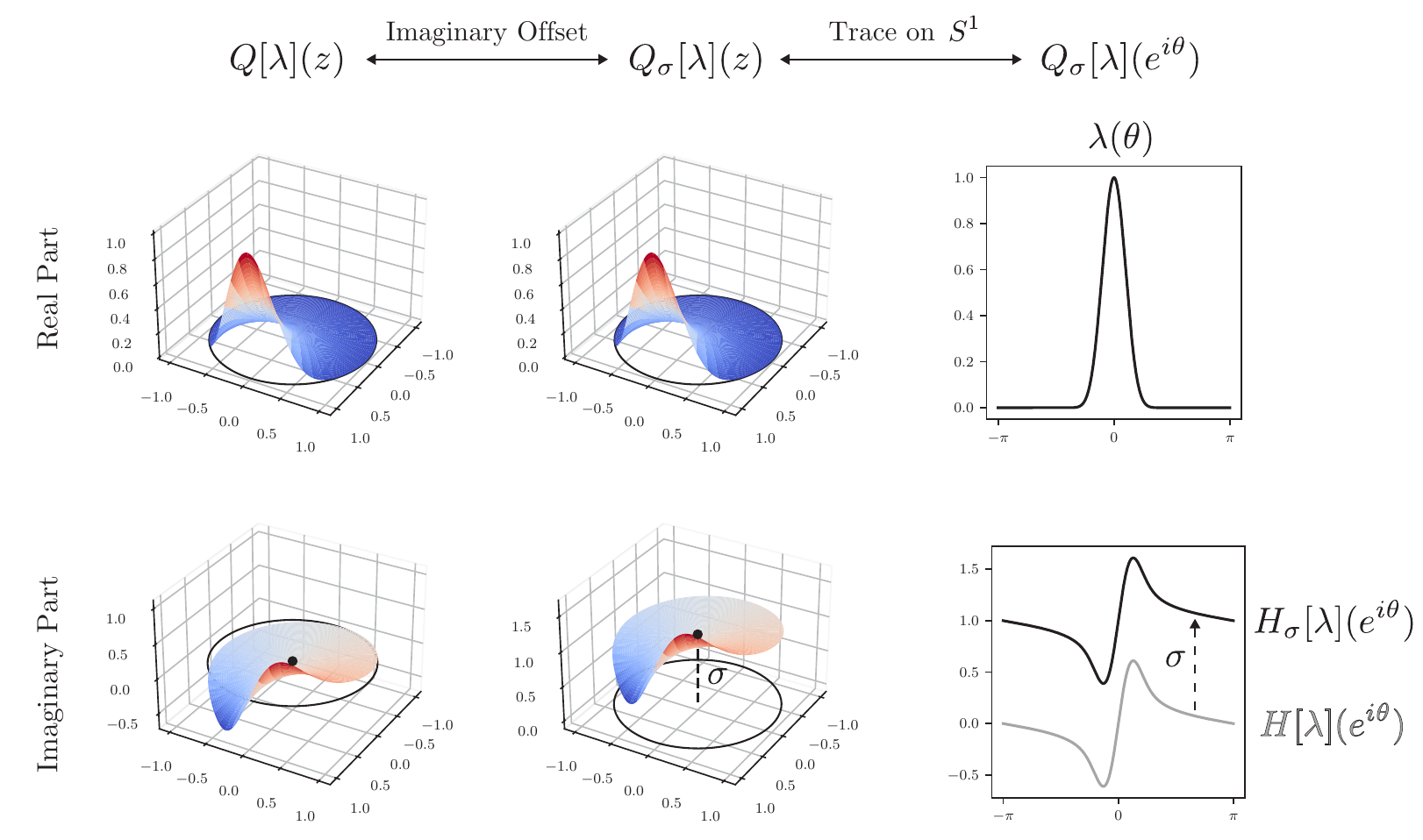}
    \caption[Visualization of the $\sigma$-Cauchy Transform]{Visualization of the Cauchy transform $Q$ given by \cref{eq:cauchycirc}. By adding an imaginary component to the Cauchy transform, we recover the one-parameter family of $\sigma$-Cauchy transforms $Q_\sigma$, given by \cref{eq:cauchygeneral}; these transforms allow us to capture the Cauchy transforms on the circle and real line (and in fact, any smooth Jordan curve) using the same theory. By taking the imaginary trace of $Q$ and $Q_\sigma$ along the unit circle, we recover the Hilbert and $\sigma$-Hilbert transforms, respectively.}
    \label{fig:H_Q_sigma}
\end{figure}

The utility of this one-parameter family of integral transforms is best seen by mapping our setting to the real line. The standard definition of the \emph{Cauchy transform} on the real line is as follows; for any $\lambda\in\cM^{(1)}(\RR)$, we set
\[Q_\RR[\lambda](z) \doteq \int\frac{i\,d\lambda(s)}{\pi (z-s)}.\]
This is a holomorphic function on the open half-plane $\HH = \{z\in\CC\;|\;\Im z > 0\}$. \dave{Its imaginary trace along the real line} is known as the \emph{Hilbert transform}:
\begin{equation}\label{eq:hilberttransform}
    H_\RR[\lambda](t) \doteq \lim_{\eps\to 0}\Im Q_\RR[\lambda](t+i\eps) = \pv\int\frac{d\lambda(s)}{\pi(t-s)},
\end{equation}
defined almost everywhere in $\RR$. Here, the (Cauchy) `principal value' of the integral is taken~\cite{kanwal1996linear}---in other words, the contour of integration is understood to travel above \dave{(i.e., in the $+i$ direction)} any singularities of $Q_\RR[\lambda]$.

Then we note that, for any $\lambda\in\cM_+^{(1)}(\RR)$, the embedding $\psi$ given by \cref{eq:embedding} yields the identity
\begin{equation}\label{eq:cauchytransform}
    Q[\psi[\lambda]](\phi^{-1}(z)) = \int\frac{i\,d\lambda(s)}{\pi (z-s)} + \int\frac{is\,d\lambda(s)}{\pi(1+s^2)} \doteq Q_\RR[\lambda](z) - i\sigma_\RR(\lambda), 
\end{equation}
where \[\sigma_\RR:\cM_+^{(1)}(\RR)\to\RR,\qquad \lambda\mapsto -\int\frac{s\,d\lambda(s)}{\pi(1+s^2)} \]
measures the imaginary part of $Q_\RR[\lambda]$ at $z=+i$. Intuitively, the imaginary offset of $Q_\RR[\lambda](z)$ is fixed by the requirement that $Q_\RR[\lambda](z)\to 0$ as $z\to\infty$, but the imaginary offset of $Q[\psi[\lambda]](\phi^{-1}(z))$ is fixed by the requirement that $\Re Q[\psi[\lambda]](0)\in\RR$. Since $0=\phi^{-1}(i)$, it is exactly the functional $\sigma_\RR$ that quantifies this difference. 

Combining the identity \cref{eq:cauchytransform} with the result of \cref{cor:uniqueness}, we deduce that $Q_\RR$ is defined (up to the addition of an imaginary constant) by the property\footnote{\dave{This convergence corresponds to the $W_{-2}$ topology that we will introduce in \cref{def:winf}.}} that
\[(1+t^2)^{-1}\Re Q_\RR[\lambda](t+i\eps)\,dt \rightharpoonup (1+t^2)^{-1}\,d\lambda(t)\]
weakly as $\eps\to 0$. In particular, we see that $\Re Q_\RR[\lambda](t+i\eps)\,dt \to \lambda$ \emph{locally} weakly\footnote{In other words, the restriction of the measures $\Re Q_\RR[\lambda](t+i\eps)\,dt$ to any compact set converges weakly to the same restriction of $\lambda$. This notion is sometimes known as \emph{vague} convergence.}. 

Critically, this insight implies that the Cauchy transform on the real line (and similarly for any smooth Jordan curve) can be seen as a special case of the one-parameter family of transforms furnished by \cref{cor:uniqueness}, with $\sigma = \sigma_\RR(\lambda)$. Of course, a similar statement for the Hilbert transform follows:
\begin{equation}\label{eq:H_real2circ}
    H[\psi[\lambda]](\phi^{-1}(t)) = H_\RR[\lambda](t) - \sigma_\RR(\lambda).
\end{equation}

Before proceeding \dave{to our main results}, we develop a quick result characterizing $H_\sigma[\lambda]$ outside the support of $\lambda$:
\begin{lemma}\label{lem:smooth}
    Fix $\sigma\in\RR$. If $\lambda\in\cM_+(S^1)$, then $H_\sigma[\lambda]$ is smooth and strictly decreasing (in the counterclockwise direction) on each component of $S^1\setminus\supp f$. As a consequence, if $\psi^{-1}[\lambda]\in \cM_+(\RR)$ is compactly supported, then $t\mapsto H_\RR[\psi^{-1}[\lambda]](-1/t)$ is smooth and strictly decreasing in an interval of $t=0$.
\end{lemma}
\begin{proof}
    Fix a component $I\subset S^1\setminus\supp \lambda$, so that $\lambda\equiv 0$ uniformly on this interval. For any $\theta_0\in I$, it follows from \cref{eq:poissoncirc_conj} that
    \[
        H_\sigma[\lambda](\theta_0) = \sigma+\pv\int_{S^1}\cot\left(\tfrac{\theta_0-\theta}{2}\right)\,d\lambda(\theta) = \sigma+\int_{S^1\setminus I}\cot\left(\tfrac{\theta_0-\theta}{2}\right)\,d\lambda(\theta).
    \]
    Since the integrand has a partial derivative (with respect to $\theta$) defined almost everywhere, a strong version of the Leibniz rule~\cite{folland2013real} shows that
    \[
    \frac{d}{d\theta}H_\sigma[\lambda](\theta_0) = -\frac{1}{2}\int_{S^1\setminus I} \csc^2\left(\tfrac{\theta_0-\theta}{2}\right)\,d\lambda(\theta)\leq 0,
\]
    with equality if and only if $\lambda\equiv 0$. Since the cotangent is smooth with each derivative uniformly bounded in $S^1\setminus I$, we can deduce similarly that $H_\sigma[\lambda]$ is smooth in $I$. The final claim follows from choosing $I\ni -1$ and applying the Cayley transform.
\end{proof}

%% file: 4_mainresults.tex
\section{Main Results}\label{sec:main}
As discussed above, much of our analysis is performed `two steps removed' from the topic of Volterra equations, in the setting of holomorphic functions on the disc. \Cref{sec:main_circ} is dedicated to understanding positive measures on the circle, which are related to holomorphic functions on the disc through \cref{cor:uniqueness}. We introduce a natural involution $\cB$ on the set $\cM_+(S^1)\times\RR$, we show it to be weakly continuous, and we develop a practical closed-form expression for $\cB$. We show how the map $\cB$ corresponds to the solution (or \emph{interconversion}) of the discrete-time equation \ref{eq:integrodiff_dPD}, giving a flavor of our subsequent results for continuous-time equations.

\dave{\Cref{sec:main_line} pulls the involution $\cB$ back to the real line, yielding a map $\cB_\RR$ well-defined on a large subset of $\cM_+^{(1)}(\RR)\times\RR\times\RR_+$. Before exploring how widely $\cB_\RR$ can be defined, we show how it corresponds to the interconversion of \emph{both} \ref{eq:integrodiff_CM} and \ref{eq:integrodiff_PD}. We then develop a closed-form expression for $\cB_\RR$ under mild hypotheses on the measure $\lambda\in\cM_+^{(1)}(\RR)$, along with two more-specialized results in this direction. First, we see how $\cB_\RR$ reduces to known interconversion formulas for Prony series~\cite{tschoegl2012phenomenological,bhattacharya2023learning} when $\lambda$ is a finite sum of atoms; second, we see how it can be modified to handle \ref{eq:integrodiff_PD} in the case $\Im c_0>0$. Finally, we show that $\cB_\RR$ is well-defined over several larger subsets of $\cM_+^{(1)}(\RR)$, and we show that $\cB_\RR$ is continuous on these subsets with respect to natural variants of the weak topology.}

Finally, \cref{sec:regularizedH} extends our theory on the real line by constructing a regularized Hilbert transform $H_\mathrm{reg}$ \dave{on} $\cM^{(2)}(\RR)$, \dave{an object first discovered in the context of Calder\'on--Zygmund theory~\cite{10.1007/BF02392130}. Corresponding to $H_\mathrm{reg}$ is a new involution $\cB_\mathrm{reg}$, which extends the involution $\cB_\RR$ to all of $\cM_+^{(2)}(\RR)\times\RR\times\RR_+$.} After proving similar closed-form expressions and continuity properties for $\cB_\mathrm{reg}$, we show how it yields interconversion formulas for both \ref{eq:integrodiff_rPD} and \ref{eq:integrodiff_rCM} in different limits.

\subsection{Measures on the Circle and Discrete-Time Volterra Equations}\label{sec:main_circ}
On the circle, our primary object of study is the following involution:
\begin{theorem}[Definition of $\cB$]\label{thm:main_circle}
    For any $c_0\in\RR$ and $\lambda\in\cM_+(S^1)$, there are unique $\zeta_0\in\RR$ and $\mu\in\cM_+(S^1)$ such that
    \[Q_{c_0}[\lambda](z)Q_{\zeta_0}[\mu](z)\equiv 1.\]
    In this context, we write
    \begin{equation}\label{eq:involution_s1}
        \cB[\lambda,c_0] = (\mu,\zeta_0).
    \end{equation}
    The map $\cB$ is an involution of $\cM_+(S^1)\times\RR$, continuous with respect to the product of the weak topology on $\cM_+(S^1)$ and the standard topology on $\RR$. By evaluating at the origin, we find
    \[(\|\lambda\| + ic_0)(\|\mu\|+i\zeta_0) = 1,\]
    with $\|\cdot\|$ the total variation norm on $\cM_+(S^1)$.
\end{theorem}
\Cref{thm:main_circle} is proved in \cref{sec:disc}, \dave{and the involution $\cB$ is illustrated in \cref{fig:circle_inversion}.} In fact, we prove a significant generalization of the theorem, encompassing a wide class of nonlinear functions applied to $Q[\lambda]$. We are interested in calculating $\cB$ explicitly, for which we introduce the following notation:
\begin{definition}\label{def:zeroset_circ}
    Suppose $\lambda\in\cM_+(S^1)$. Define the zero set of $\lambda$ as
    \[N_{0}(\lambda) = \bigcap_{\eps>0}\op{clos}\Big\{e^{i\theta}\in S^1\;\Big|\;\limsup_{\delta\to 0}\lambda(\exp i[\theta-\delta,\theta+\delta])/2\delta < \eps\Big\}.\]
\end{definition}

If $\lambda$ is a continuous measure with continuous density, for instance, the set $N_0(\lambda)$ corresponds exactly to the zeroes of this density. So long as $N_0(\lambda)$ is not too badly behaved, we can compute $\cB$ exactly: 
\begin{theorem}[Closed form of $\cB$]\label{thm:main_circ_formula}
    Let $\lambda\in\cM_+(S^1)$, and write $\supp\lambda\subset S^1$ for its closed, essential support. Fix $c_0\in\RR$ and suppose that
    \[Z'\doteq \left(N_0(\lambda)\cap\supp\lambda\right)\cup\{z\notin\supp\lambda\;|\;H[\lambda](z) + c_0 = 0\}\]
    is discrete\footnote{This definition allows $Z'$ to be infinite, so long as the limit points of $Z'$ do not themselves belong to $Z'$.}, i.e., if $z\in Z'$, there is an $\eps>0$ such that $|z-z'|>\eps$ for any $z'\neq z$ in $Z'$. Write $\lambda_c\in L^1(S^1)$ for the density of the continuous component of $\lambda$ with respect to the normalized Lebesgue measure $(2\pi)^{-1}\,d\theta$, and write $\cB[\lambda,c_0] = (\mu,\zeta_0)$. Then we find
    \begin{equation}\label{eq:mu_circle}
        d\mu(\theta) = (2\pi)^{-1}\mu_c(e^{i\theta})\,d\theta + \sum_{\alpha_i\in Z}\beta_i\delta(\theta-\theta_i)\,d\theta,
    \end{equation}
    where the continuous part is given by
    \begin{equation}\label{eq:mu_c_circle}
        \mu_c(e^{i\theta}) = \frac{\lambda_c(e^{i\theta})}{\lambda_c(e^{i\theta})^2 + \big(H[\lambda](e^{i\theta}) + c_0\big)^2} \in L^1( S^1),
    \end{equation}
    and the discrete part has weights
    \begin{equation}\label{eq:beta_i_circle}
        \beta_i = \left(\int \frac{d\lambda(\theta)}{\sin^2[(\theta-\theta_i)/2]}\,\right)^{-1},
    \end{equation}
    for all $e^{i\theta_i}\in Z$ in the discrete set
    \begin{equation}\label{eq:Z_circle}
        Z = N_0(\lambda)\cap\{z\in S^1\;|\;H[\lambda](z)+ c_0 = 0\} \dave{\subset Z'}.
    \end{equation}
    Finally, we have that
    \begin{equation}\label{eq:zeta_0_circle}
        \zeta_0 = \Im\big[(\|\lambda\| + ic_0)^{-1}\big].
    \end{equation}
\end{theorem}

\begin{figure}
    \centering
    \includegraphics[width=\linewidth]{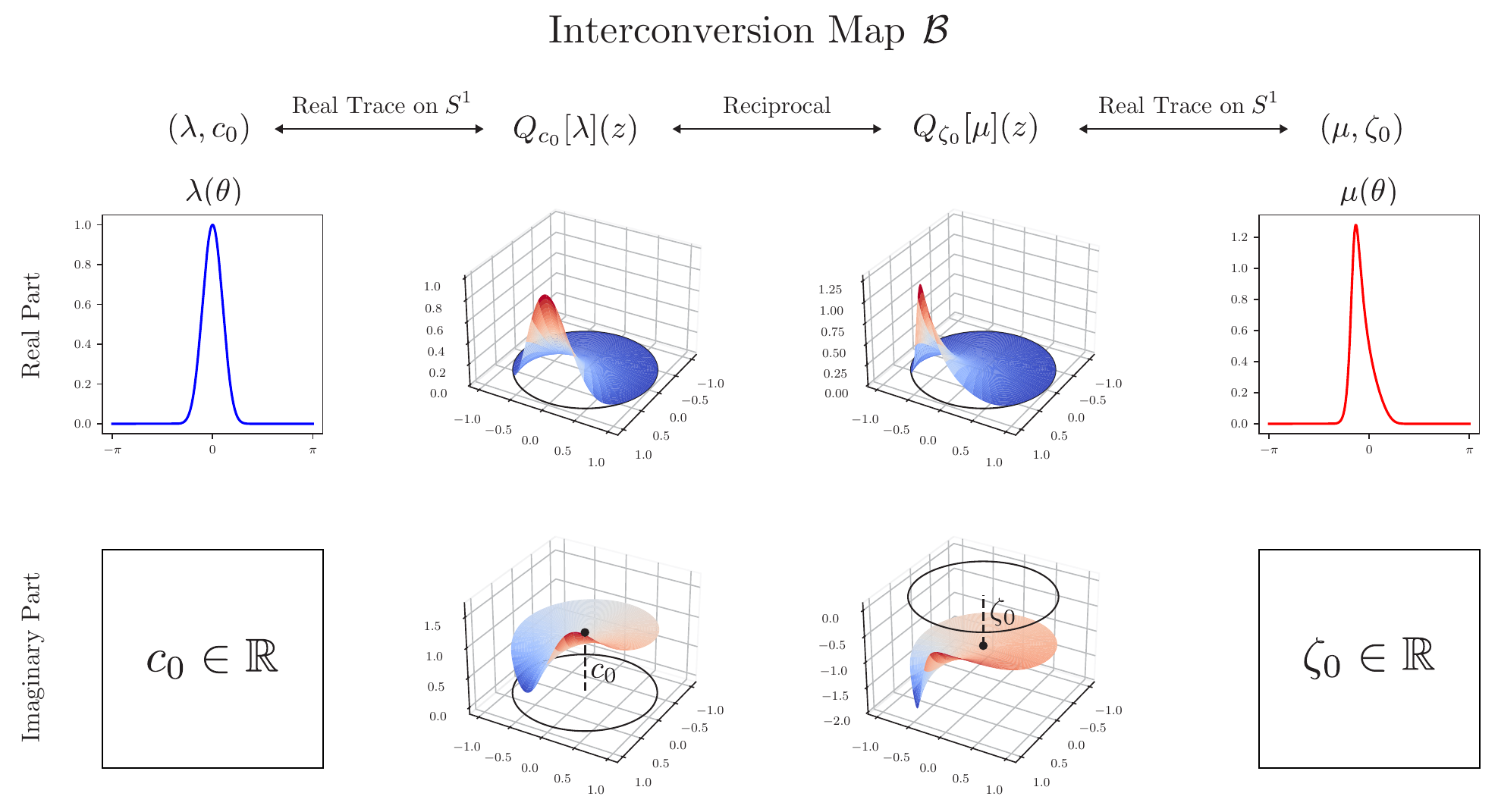}
    \caption[Visualization of the Interconversion Map $\cB$]{Visualization of the interconversion map $\cB$ of \cref{thm:main_circle}. This map directly allows for the interconversion of discrete-time Volterra equations of the form \ref{eq:integrodiff_dPD}, but can also be leveraged to solve integral, integro-differential, delay differential, and fractional differential equations.}
    \label{fig:circle_inversion}
\end{figure}

We prove the (substantially harder) case of the real line below, as \cref{thm:main_formula}; our proof can be adapted straightforwardly to the case of $S^1$.
Although our ultimate aim is to pull $\cB$ back to the real line to understand continuous-time Volterra equations, it is also directly useful for solving discrete-time Volterra equations. We prove the following proposition in \cref{sec:laplace}:
\begin{restatable}[Solution of \ref{eq:integrodiff_dPD}]{proposition}{propmaindPD}\label{prop:main_dPD} Consider the setting of \ref{eq:integrodiff_dPD}, and recall that $\Re c_0 \geq -\tfrac{1}{2}K(0)$ by hypothesis. Write
\[c'_0 = c_0 -2\Re c_0 - K(0),\qquad K'(n) = K(n) + \delta(n)\left(2\Re c_0 + K(0)\right),\]
where $\delta(n)$ is a discrete delta function. It is easy to verify that $K'(n)$ is positive definite, and that the pair $(c_0',K')$ give rise to the same discrete-time Volterra equation as $(c_0,K)$ but now satisfying the equality $\Re c_0' = -\tfrac{1}{2}K'(0)$. Write $\lambda\doteq\cF^{-1}[K']\in\cM_{+}(S^1)$, and define
\[(\mu,\zeta_0')=\cB[\lambda, 2\Im c_0'],\qquad J = 4\cF[\mu].\]
Setting $\zeta_0 = 2i\zeta_0' - \frac{1}{2}J(0)$, the equation \ref{eq:integrodiff_dPD} is satisfied by
    \[x(n) = \zeta_0y(n) + \sum_{j=0}^n J(n-j)y(j).\]
\end{restatable}
We illustrate this result with a simple, analytical example, for which the above theorem reduces to classical power series techniques:
\begin{example}\label{ex:dPD}
    Fix $-1<a<1$, and consider the equation
    \[y(n) = \sum_{j=0}^n (j+1)a^jx(n-j).\]
    Following \cref{prop:main_PD}, we make the choice $c'_0=-1$, $K'(n) = (|n|+1)a^{|n|}+\delta(|n|)$, which corresponds to the measure
    \[d\lambda(\theta) = \Re\frac{2}{(1-ae^{i\theta})^2}\,\frac{d\theta}{2\pi},\qquad Q[\lambda](z) = \frac{2}{(1-az)^2}.\]
    By comparing against the statement of \cref{thm:main_circle}, we see that $\zeta_0'=0$ and $d\mu(\theta) = \Re[(1-ae^{i\theta})^2]\,d\theta/4\pi$, and thus that
    \[J = 2\delta(n) - 2a\delta(n-1) + a^2\delta(n-2),\qquad \zeta_0 = -1.\]
    Putting these ingredients together, we find
    \[x(n) = y(n) - 2ay(n-1) + a^2y(n-2).\]
    This inversion is shown in \cref{fig:main_inversion}.
\end{example}

\subsection{Measures on the Line and Continuous-Time Volterra Equations}\label{sec:main_line}
In treating integral and integro-differential equations, we are primarily interested in the pullback of the involution $\cB$ to $\RR$. Now, the embedding $\psi:\cM_+^{(2)}(\RR)\to\cM_+(S^1)$ defined by \cref{eq:embedding} nearly covers its entire codomain, with the only element in the cokernel being the Dirac measure $\delta_{-1}\in\cM_+(S^1)$ at $-1=\phi^{-1}(\infty)$. To understand how the latter `should' behave under our map, we calculate 
\[Q[\delta_{-1}](\phi^{-1}(z))=-iz.\]
We can combine this expression with that of \cref{eq:cauchytransform} to develop a slight extension of our embedding $\psi$, to account for both constant contributions to $\lambda$ as well as possible `poles at infinity'. In short, if $\psi\in\cM_+^{(1)}(\RR)$, $c_0\in\RR$, and $c_1\geq 0$, we know that there is a value $c'_0 = \pi(\sigma_\RR(\lambda) - c_0)\in\RR$ such that
\[Q[\psi[\lambda] + \pi c_1\delta_{-1}](\phi^{-1}(z)) + ic'_0 = Q_\RR[\lambda](z) - i\pi^{-1}(c_0 + c_1z),\]
with $\pi$ scalings chosen for later convenience. To codify this relationship, we write
\begin{equation}\label{eq:Psi_map}
\begin{gathered}
    \Psi:\cM_+^{(1)}(\RR)\times\RR\times\RR_+\to \cM_+(S^1)\times\RR,\\\Psi[\lambda,c_0,c_1] = \left(\psi[\lambda] + \pi^{-1} c_1\delta_{-1}, \sigma_\RR(\lambda) - \pi^{-1} c_0\right).
\end{gathered}
\end{equation}

\begin{figure}
    \centering
    \includegraphics[width=\linewidth]{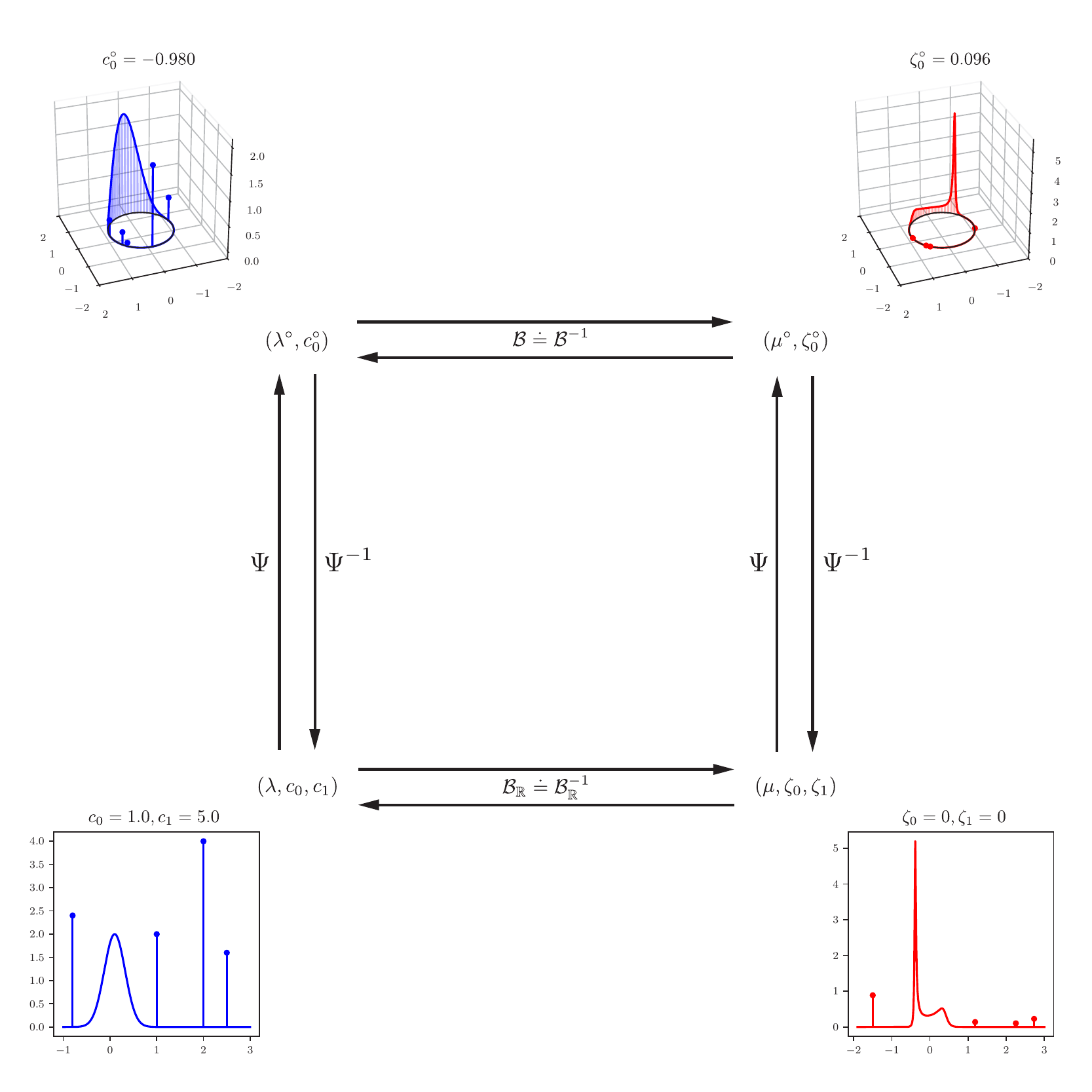}
    \vspace{-20pt}
    \caption[Commutativity Relation of Continuous- and Discrete-Time Interconversion]{Commutative diagram showing how Volterra integral and integro-differential equations, corresponding to triples $(\lambda,c_0,c_1)\in\cM_+^{(1)}(\RR)\times\RR\times\RR_+$, can be lifted to the circle by the map $\Psi$ defined in \cref{eq:Psi_map}. The interconversion maps $\cB$ and $\cB_\RR$, corresponding to discrete-time equations and integral (or integro-differential) equations, respectively, \dave{are related to each other by} the embedding $\Psi$.}
    \label{fig:commutative_diagram}
\end{figure}

The behavior of $\Psi$ is shown in \cref{fig:commutative_diagram}. In particular, we see that it allows us to pull the involution $\cB$ back to the line in a natural way, at the cost of introducing a second real parameter. More rigorously, \cref{thm:main_circle} implies that\footnote{We will formalize this \dave{particular} claim in \cref{thm:main_rPD}, below.}, for any measure $\lambda\in\cM_+^{(2)}(\RR)$ and parameters $c'_0\in\RR$ and $c_1\geq 0$, there is a unique measure $\mu\in\cM_+^{(2)}(\RR)$ and parameters $\zeta'_0\in\RR$ and $\zeta_1\geq 0$ such that
\[\left(Q[\psi[\lambda]](\phi^{-1}(z)) - i\pi^{-1}c'_0 - i\pi^{-1}c_1z\right)\left(Q[\psi[\mu]](\phi^{-1}(z)) - i\pi^{-1}\zeta'_0 - i\pi^{-1}\zeta_1z\right)\equiv 1.\]

\dave{For now, we are interested in the case that both $\lambda$ and $\mu$ are known to live in $\cM_+^{(1)}(\RR)$, corresponding to \emph{local integrability} of the kernel $K(t)$ in \cref{lem:bernstein,lem:bochner}.} If $\lambda,\mu\in\cM_+^{(1)}(\RR)$, then the values $\sigma_\RR(\lambda),\sigma_\RR(\mu)\in\RR$ are well-defined by \cref{eq:cauchytransform}, and we see that
\[\left(Q_\RR[\lambda](z) - i\pi^{-1}c_0 - i\pi^{-1}c_1z\right)\left(Q_\RR[\mu](z) - i\pi^{-1}\zeta_0 - i\pi^{-1}\zeta_1z\right)\equiv 1,\]
where $c_0 = \sigma_\RR(\lambda) - \pi^{-1}c'_0$ and $\zeta_0 = \sigma_\RR(\mu) - \pi^{-1}\zeta'_0$ are both real.
In parallel with \cref{thm:main_circle}, we write
\begin{equation}\label{eq:define_B}
	\cB_\RR[\lambda,c_0,c_1] = (\mu,\zeta_0,\zeta_1),
\end{equation}
\dave{though we note that $\cB_\RR$ is not well-defined for all $\lambda\in\cM_+^{(1)}(\RR)$; we discuss sufficient conditions for $\cB_\RR$ to be well-defined in \cref{thm:main_formula,thm:main_continuity} below. The utility of $\cB_\RR$ is highlighted by the following results, which we prove in \cref{sec:laplace}:}
\begin{restatable}[Solution of \ref{eq:integrodiff_CM}]{proposition}{propmainCM}\label{prop:main_CM}
    Suppose $K:\RR\to\CC$ is a gCM kernel for which $\lambda\doteq\cL_b^{-1}[K]\in\cM_\mathrm{exp}^{(1)}(\RR)$. In the setting of \ref{eq:integrodiff_CM}, suppose  $(\mu,\zeta_0,\zeta_1)=\cB_\RR[\lambda,c_0,c_1]$ is well-defined with $\mu\in\cM_\mathrm{exp}^{(1)}(\RR)$, and write $J = \cL_b[\mu]$. Then \ref{eq:integrodiff_CM} is satisfied by
    \[-\pi^{2}x(t) = \zeta_1\dot{y}(t) - \zeta_0y(t) - \int_0^t J(t-\tau)y(\tau)\,d\tau - c_1x_0 J(t).\]
\end{restatable}
\begin{restatable}[Solution of \ref{eq:integrodiff_PD}]{proposition}{propmainPD}\label{prop:main_PD} Suppose $K:\RR\to\CC$ is a gPD kernel for which $\lambda\doteq\cF^{-1}[K]\in\cM_{+}^{(1)}(\RR)$. In the setting of \ref{eq:integrodiff_PD}, suppose $(\mu,\zeta_0,\zeta_1)=\cB_\RR[\lambda,c_0,c_1]$ is well-defined, and write $J = \cF[\mu]$. If $\mu\in\cM_+^{(1)}(\RR)$, then \ref{eq:integrodiff_PD} is satisfied by
    \[\pi^{2}x(t) = \zeta_1\dot{y}(t) - i\zeta_0y(t) + \int_0^t J(t-\tau)y(\tau)\,ds + c_1x_0J(t).\]
\end{restatable}
\begin{remark}
    As discussed in \cref{sec:note_initialconds}, both \cref{prop:main_CM} and \cref{prop:main_PD} can be adapted to homogeneous initial data with an infinite time horizon---i.e., $x\to 0$ as $t\to-\infty$. For this, we need only to change the lower bound of each integral above, from $0$ to $-\infty$, and discard the term depending on $x_0$.
\end{remark}

\dave{We now aim to develop} a practical formula for $\cB_\RR$, for which we need the following analogue of \cref{def:zeroset_circ}:
\begin{definition}\label{def:zeroset}
    Suppose $\lambda\in\cM_{+,\mathrm{loc}}(\RR)$. Then we define the zero set
    \[N_{0}(\lambda) \doteq \bigcap_{\eps>0}\op{clos}\Big\{s\in\RR\;\Big|\;\limsup_{\delta\to 0}\lambda([s-\delta,s+\delta])/2\delta < \eps\Big\}.\]
\end{definition}
Equivalently, we could define the zero set as the pullback of the zero set of \cref{def:zeroset_circ} to $\RR$:
    \[N_0(\lambda)= \phi\big(N_0(\psi[\lambda])\setminus\{-1\}\big)\subset\RR.\]
Next, given a non-negative function $f\in L^1(\RR)$, we say that $f\in L^*(\RR)$ if
\[(1+s^2)^{1/2}f(s)\in L^1(\RR),\qquad \cF\left[s\mapsto s^2f(s)\right]\in L^1(\RR),\]
denoting by $\cF$ the Fourier transform on $L^1(\RR)$. These conditions ensure that $f$ is sufficiently smooth and decaying sufficiently quickly for our analysis to go through. For instance, it is sufficient that $f\in\cC^2(\RR)$ is second-differentiable with $(1+t^2)\,\frac{d^2}{dt^2}f(t)$ bounded. The following theorem is proved in \cref{sec:line}:
\begin{restatable}[Closed form of $\cB_\RR$]{theorem}{thmmainformula}\label{thm:main_formula}
    Let $\lambda\in L^*(\RR) +\cM_c(\RR)\subset\cM_+^{(-1)}(\RR)$, in the sense that $\lambda=\lambda_1+\lambda_2$ for a non-negative function $\lambda_1\in L^*(\RR)$ and measure $\lambda_2\in\cM_c(\RR)$. Fix $c_0\in\RR$ and $c_1\geq 0$, and suppose that
    \[Z'\doteq \left(N_0(\lambda)\cap\supp\lambda\right)\cup\{s\notin\supp\lambda\;|\;H[\lambda](s) - \pi^{-1}(c_1s + c_0) = 0\}\]
    is discrete (i.e., it does not contain any of its limit points). Write $\lambda_c$ for the density of the continuous component of $\lambda$. Then $\cB_\RR[\lambda,c_0,c_1] = (\mu,\zeta_0,\zeta_1)$ is well-defined, and we find
    \begin{equation}\label{eq:dmu}
        d\mu(s) = \mu_c(s)\,ds + \sum_{\alpha_i\in Z} \beta_i\delta(s-\alpha_i)\,ds,
    \end{equation}
    where the continuous part is given by
    \begin{equation}\label{eq:mu_c}
        \mu_c(s) = \frac{\lambda_c(s)}{\lambda_c(s)^2 + \big(H_\RR[\lambda](s) - \pi^{-1}(c_1s + c_0)\big)^2} \in L^1(\RR),
    \end{equation}
    and the discrete part has weights
    \begin{equation}\label{eq:betaeq}
        \beta_i = \pi^2\left(c_1 + \int\frac{d\lambda(\tau)}{(\tau-\alpha_i)^2}\right)^{-1},
    \end{equation}
    for all $\alpha_i \in Z$ in the discrete set
    \begin{equation}\label{eq:zero_set}
        Z = N_0(\lambda)\cap\big\{s\in\RR\;\big|\;H_\mathbb{R}[\lambda](s) - \pi^{-1}(c_1s + c_0) = 0\big\}.
    \end{equation}
    
    If $c_1\neq 0$, then we have $\zeta_0 = \zeta_1 = 0$. If $c_1 = 0$ but $c_0\neq 0$, then $\zeta_1 = 0$ and $\zeta_0 = -\pi^2/c_0$. Finally, if $c_0=c_1=0$, then we have
    \begin{equation}\label{eq:zetas}
        \zeta_0 = -\frac{\pi^2}{\|\lambda\|^2}\int\tau\,d\lambda(\tau),\qquad \zeta_1 = \frac{\pi^2}{\|\lambda\|},
    \end{equation}
    writing $\|\lambda\| = \int d\lambda$ for the variation norm of $\lambda$.
\end{restatable}

\begin{example}\label{ex:gCM}
    \dave{Consider the equation
    \[y(t) = x(t) + \int_0^t (1-e^{-\tau})\,x(t-\tau)\,\frac{d\tau}{\tau},\qquad x(0) = 0.\]
    This is an integral equation of the form \ref{eq:integrodiff_CM}, with $c_1=0$, $c_0=1$, and integral kernel
    \[K(t) = \frac{1}{t}(1-e^{-t}) = \cL[\lambda](t),\]
    where $d\lambda(s) = \chi_{[0,1]}(s)\,ds$ is the restriction of the Lebesgue measure to the unit interval. Since $c_1=0$ but $c_0\neq 0$, \cref{thm:main_formula} yields $\zeta_1=0$ and $\zeta_0=-\pi^2/c_0$. Next, we find
    \[H_\RR[\lambda](t) = \frac{1}{\pi}\int_0^1\frac{ds}{t-s} = -\frac{1}{\pi}\log\left|1-t^{-1}\right|.\]
    The set $Z$ has one element, $\alpha_1=(1-e^{-1})^{-1}$, with corresponding weight
    \[\beta_1 = \frac{\pi^2}{e + e^{-1} - 2}.\]
    In all, we find
    \[\mu(s) = \beta_1\delta(s-\alpha_1)\,ds + \frac{\chi_{[0, 1]}(s)\,ds}{1 + (1 + \log|1-s^{-1}|)^2/\pi^2},\]
    so we have
    \begin{gather*}
        -\pi^2x(t) = (\pi^2/c_0)y(t) - \int_0^t \beta_1e^{-\alpha_1(t-\tau)}y(\tau)\,d\tau - \int_0^t J_c(t-\tau)y(\tau)\,d\tau,\\
        J_c(t) \doteq \int_0^1e^{-st}\left(1 + (1 + \log|1-s^{-1}|)^2/\pi^2\right)^{-1}\,ds.
    \end{gather*}
    This example is depicted in \cref{fig:main_inversion}.}
\end{example}

In numerical applications, a key case of interest is that of a \emph{discrete} $\lambda$ with a finite number of atoms. This case is already well-understood in the context of Prony series~\cite{gross1968mathematical}, but it is instructive to see how \cref{thm:main_formula} reduces in this limit:
\begin{corollary}[$\cB_\RR$ on discrete measures]\label{cor:discrete_formula}
    Let $\lambda \in \cM_{c}(\RR)$ be a discrete measure
    \begin{equation}
        \lambda(s) = \sum_{i=1}^N b_i\delta(s - a_i)
    \end{equation}
    where $a_i \in \RR$ are distinct and $b_i > 0$. Fix values $c_0\in\RR$ and $c_1\geq 0$, and write $\cB_\RR[\lambda,c_0,c_1] = (\mu,\zeta_0,\zeta_1)$. Then we have that
    \begin{equation}
        \mu(s) = \sum_{i=1}^M \beta_i\delta(s - \alpha_i), \quad M = \begin{cases}N+1 & c_1 \neq 0\\N & c_0 \neq 0, c_1 = 0\\N-1 & c_0=c_1=0\end{cases}
    \end{equation}
    where the positions of the atoms $\alpha_i$ are the $M$ roots of $H_\RR[\lambda](s) - \pi^{-1}(c_1s + c_0) = 0$. These values interleave with the $a_i$ such that exactly one $\alpha_i$ lies in each interval $(a_i, a_{i+1})$. If $c_0 < 0$ or $c_1 > 0$, then one root will also lie in $(-\infty, a_1)$, and if $c_0 > 0$ or $c_1 > 0$, then one root will lie in $(a_N, \infty)$. As before, the weights are given by
    \begin{equation}
        \beta_i = \pi^2\Big(c_1 + \sum_{j=1}^N\frac{b_j}{(a_j - \alpha_i)^2}\Big)^{-1}
    \end{equation}
    and the formulas for the constants $\zeta_0, \zeta_1$ are the same as in \cref{thm:main_formula}.
\end{corollary}

\begin{example}\label{ex:gPD}
    Consider the equation
    \[y(t) = \dot{x}(t) + 2\int_0^t\cos(t-\tau)x(\tau)\,d\tau,\qquad x(0) = 1.\]
    This is an integro-differential equation of the type \ref{eq:integrodiff_PD}, with $c_1=1$, $c_0=0$, and integral kernel
    \[K(t) = 2\cos(t) = \cF[\lambda](t),\]
    where $d\lambda(t) = \delta(t-1)\,dt + \delta(t+1)\,dt$. From \cref{cor:discrete_formula}, we see that there are three atoms in the measure $\mu$:
    \[\alpha_1 = -\sqrt{3},\qquad \alpha_2 = 0,\qquad \alpha_3 = \sqrt{3},\]
    with corresponding weights $\beta_1 = \beta_2 =\beta_3 = \pi^2/3$. We thus deduce that $\mu(s) = \sum_i \beta_i\delta(s - \alpha_i)$ and obtain the following solution:
    \[\pi^2x(t) = \int_0^t J(t-\tau)y(\tau)\,d\tau, \quad J(t) = \cF[\mu](t) = \frac{\pi^2}{3}\left(1+2\cos(\sqrt{3}t)\right).\]
    This example is depicted in \cref{fig:main_inversion}.
\end{example}

For completeness' sake, we offer a similar result in the case where the measure is perturbed by a \emph{positive, real} parameter\footnote{\dave{Since $\lambda + c_0\in\cM_+^{(2)}(\RR)$ for any $\lambda\in\cM_+^{(1)}(\RR)$ and $c_0\in\HH$, this result can be seen to form a special case of \cref{thm:main_regularizedH_formula} below.}}, or equivalently, $\Im c_0>0$ in \ref{eq:integrodiff_PD}:
\begin{restatable}[$\cB_\RR$ with complex $c_0$]{proposition}{propeasyone}\label{prop:easyone}
    Suppose $\lambda\in L^*(\RR) + \cM_c(\RR)\subset\cM_+^{(1)}(\RR)$, as in \cref{thm:main_formula}. For any $c_0 \in \HH$ (that is, with $\Im c_0 >0$), there is a unique signed measure $\mu\in\cM^{(1)}(\RR)$ such that
    \begin{equation}
        \left(Q_\RR[\lambda](z) - i\pi^{-1}c_0\right)\left(Q_\RR[\mu](z) - i\pi^{-1}\zeta_0\right)\equiv 1,
    \end{equation}
    where $\zeta_0 = -\pi^2/c_0\in\HH$. Moreover, $\mu$ is absolutely continuous, and its continuous density $\mu_c$ is given by
    \begin{equation}
        \mu_c(s) = \frac{\lambda_c(s) + \pi^{-1}\Im{c_0}}{\big(\lambda_c(s)+\pi^{-1}\Im{c_0}\big)^2 + \big(H_\RR[\lambda](s) - \pi^{-1}\Re{c_0}\big)^2} - \pi\frac{\Im c_0}{|c_0|^2}.
    \end{equation}
    Similarly, for any $c_1>0$ and $c_0\in\HH$, there is a unique $\mu'\in\cM_+^{(1)}(\RR)$ such that
    \begin{equation}
        \left(Q_\RR[\lambda](z) - i\pi^{-1}c_0 - i\pi^{-1}c_1z\right)Q_\RR[\mu'](z)\equiv 1.
    \end{equation}
    It is again absolutely continuous, with density
    \begin{equation}
        \mu_c'(s) = \frac{\lambda_c(s) + \pi^{-1}\Im{c_0}}{\big(\lambda_c(s)+\pi^{-1}\Im{c_0}\big)^2 + \big(H_\RR[\lambda](s) - \pi^{-1}\Re{c_0} - \pi^{-1}c_1s\big)^2}.
    \end{equation}
\end{restatable}

Notably, this result does \emph{not} guarantee that $\mu$ or $\mu'$ lies in $\cM_\mathrm{exp}^{(1)}(\RR)$. Thus, while it can safely be employed in conjunction with \cref{prop:main_PD} to solve equations of the form \ref{eq:integrodiff_PD}, it generically cannot be used with \cref{prop:main_CM} to solve equations of the form \ref{eq:integrodiff_CM}.
\begin{example}\label{ex:gPD_complex}
    Consider the equation
    \[y(t) = x(t) + \int_{-\infty}^t e^{-(t-\tau)^2}x(\tau)\,d\tau,\qquad \lim_{t\to-\infty}x(t) = 0.\]
    This is an integral equation of the form \ref{eq:integrodiff_PD}, with $c_0 = i$ and integral kernel
    \[K(t) = e^{-t^2} = \cF[\lambda](t),\]
    where $\lambda = \tfrac{1}{2\sqrt{\pi}}e^{-t^2/4}\,dt$. Now we use the fact that
    \begin{equation}
        H_\RR[e^{-t^2/a}] = \frac{2}{\sqrt{\pi}}D(t/\sqrt{a})
    \end{equation}
    where $D(x) = e^{-x^2}\int_0^xe^{t^2}dt$ is the Dawson function. \Cref{prop:easyone} thus implies that
    \[\zeta_0=\pi^2i,\qquad \mu_c(t) = \frac{\tfrac{1}{2\sqrt{\pi}}e^{-t^2/4} + \pi^{-1}}{(\tfrac{1}{2\sqrt{\pi}}e^{-t^2/4} + \pi^{-1})^2 + \frac{1}{\pi^2}D(t/2)^2}- \pi,\]
    and hence we obtain
    \[x(t) = y(t) + \frac{1}{\pi^2}\int_{-\infty}^tJ(t-s)y(s)\,ds, \quad J = \cF[\mu].\]
    Note in this example that $J$ is not a PD kernel, but $-J$ is; this is allowed by the stipulation in \cref{prop:easyone} that $\mu$ is signed. \dave{This example} is shown in \cref{fig:main_inversion}.
\end{example}

Next, we prove important continuity properties of the map $\cB_\RR$, \dave{mirroring the weak continuity of the map $\cB$ on the circle}. We show, for one, that $\cB_\RR$ is well-defined on a wider class of measures than allowed by \cref{thm:main_formula}, and that it is continuous \dave{on this class} with respect to natural variants of the weak topology. For this, we define the following topologies:

\begin{restatable}[Variants of the weak topology]{definition}{defwinf}\label{def:winf}
	We say that $\lambda_j\in\cM_+^{(n)}(\RR)$ converges to $\lambda\in\cM_+^{(n)}(\RR)$ in the \emph{$W_{-n}$-topology} if
	\[(1+s^2)^{-n/2}\,d\lambda_j(s)\rightharpoonup (1+s^2)^{-n/2}\,d\lambda(s)\]
	weakly. Likewise, we say that $\lambda_j\in\cM_c(\RR)$ converges to $\lambda\in\cM_c(\RR)$ in the \emph{$W_\infty$-topology} if
	\[\int f\,d\lambda_j\to\int f\,d\lambda\]
	for all continuous (but not necessarily bounded) functions $f\in\cC(\RR)$. 
\end{restatable}
\begin{remark}\label{rem:winf}
    Restricted to the set of probability measures with finite $n^\mathit{th}$ moments, the $W_{+n}$ topology agrees with the classical Wasserstein-$n$ topology~\cite[Ch.~5]{santambrogio2015optimal}. On the other hand, the $W_\infty$ topology is strictly weaker than the Wasserstein-$\infty$ topology, but strictly stronger than the limit of the $W_{+n}$ topologies as $n\to +\infty$.
    
    To compare the $W_\infty$ topology against the Wasserstein-$\infty$ topology, consider the measures $\mu_j = (1-e^{-j})\delta_0 + e^{-j}\delta_j$, where we write $\delta_x$ for the Dirac measure at $x\in\RR$. From \cref{prop:topology}, we will see that $\mu_j\to\delta_0$ in the $W_\infty$ topology. On the other hand, the Wasserstein-$\infty$ distance between $\mu_j$ and $\delta_0$ is always $1$, so the sequence does not converge.

    To compare the $W_\infty$ topology against the limit of the $W_{+n}$ topologies, consider the measures $\lambda_j = \delta_0 + e^{-j}\delta_j$, where we write $\delta_x$ for the Dirac measure at $x\in\RR$. It is clear that $(1+s^2)^{n/2}\,d\lambda_j(s)\rightharpoonup (1+s^2)^{n/2}\,d\lambda(s)$ weakly for any fixed $n\in\RR$, so we see that $\lambda_j\to\lambda$ in $W_{+n}$. On the other hand, it will follow from \cref{prop:topology} below that $\lambda_j$ does not converge in $W_\infty$.

    Classical Wasserstein-$n$ topologies will always be denoted by the script notation $\mathcal{W}_n, \mathcal{W}_\infty$ to distinguish from the weak topologies $W_{+n}, W_\infty$ defined above.
\end{remark}

We discuss these topologies further in \cref{sec:topologies}, and we characterize them in both the spectral domain and the time domain. In one direction, we see that convergence in $W_{-n}$ corresponds to pointwise convergence of mollifications of the Fourier and Laplace transforms, and implies locally uniform convergence of the same (\cref{prop:time_cont_PD}); it also implies locally uniform convergence of the Laplace transform and all of its derivatives (\cref{prop:time_cont_CM}). In another direction, we see that, if $\lambda_j$ are uniformly supported in a fixed compact interval $I\subset\RR$, convergence in $W_\infty$ is equivalent to pointwise convergence of either the Fourier and Laplace transforms, and it implies locally uniform convergence of both transforms and all of their derivatives (\cref{lem:time_cont_infty}). Finally, we also see that convergence of integral kernels in weighted $L^p$ spaces can be controlled by reweighted Wasserstein-$p$ metrics in the spectral domain (\cref{prop:wasserstein1}). These metrics are equivalent to the $W_\infty$ topology on a fixed compact interval, so this result helps make our notion of continuity in that setting more quantitative.

Our primary topological result for gCM and gPD equations is the following, which we prove in \cref{sec:line}:
\begin{restatable}[Existence and weak continuity of $\cB_\RR$]{theorem}{thmmaincontinuity}\label{thm:main_continuity}
	Write $U^{0}=\{0\}\times\{0\}$, $U^1=(\RR\setminus\{0\})\times\{0\}$, and $U^2=\RR\times\RR_+$; these sets form a disjoint partition of $\RR\times\RR_+$. Respectively, the set $U^0$ corresponds to the choice $c_0=c_1=0$, the set $U^1$ to the choice $c_1=0$ but $c_0\neq 0$, and $U^2$ to the choice $c_1>0$. Then $\cB_\RR$ is well-defined on the following spaces:
 	\[\cB_\RR:\cM_\mathrm{exp}^{(1)}(\RR)\times U^1\to\cM_\mathrm{exp}^{(1)}\times U^1,\qquad \cB_\RR:\cM_\mathrm{exp}^{(1)}(\RR)\times U^2\to\cM_\mathrm{exp}^{(1)}\times U^0,\]
    applicable to gCM equations, and
    \[\cB_\RR:\cM_c(\RR)\times U^i\to\cM_c(\RR)\times U^{2-i},\qquad i \in \{0,1,2\},\]
    applicable to both gCM and gPD equations. The restriction to $\cM_\mathrm{exp}^{(1)}(\RR)\times U^2$ is continuous from the $W_{-2}$ topology on $\cM_+^{(1)}(\RR)$ and the standard topology on $U^2$ to the $W_{-r}$ topology on $\cM_+^{(1)}(\RR)$, for any $r>2$. The restriction to $\cM_c(\RR)\times U^i$ is continuous in product of the $W_\infty$-topology on $\cM_c(\RR)$ and the standard topology on each $U^j$. 
\end{restatable}
\begin{remark}\label{rem:Br_failure}
    Notably, this result does not make any claims about the application of $\cB_\RR$ to $\cM_\mathrm{exp}^{(1)}(\RR)\times U^0$. In brief, the obstacle to such a result is that the interconverted equation can pick up a term corresponding to a fractional derivative. Such equations are handled neatly by our `regularized' theory in \cref{sec:regularizedH}, and we see there how fractional derivatives naturally complete the definition of $\cB_\RR$; a striking example of this form arises in Abel-type equations (see \cref{ex:abeltype}).
\end{remark}
Although written in an abstract form, \cref{thm:main_continuity} has practical applications in solving Volterra equations. For one, it guarantees that gCM Volterra equations are closed under interconversion whenever either (a) the measure $\lambda$ is compactly supported or (b) the measure $\lambda$ has support bounded below and either $c_0$ or $c_1$ is nonzero. Its statement of continuity justifies, for instance, the approximation of \ref{eq:integrodiff_CM} using {Prony series}~\cite{tschoegl2012phenomenological, serra2019viscoelastic, polyanin2017integral, gross1968mathematical}. We refer the reader to \cref{fig:inverse_cont} for a numerical illustration of the continuity of the map $\cB_\RR$.

The limitations of \cref{thm:main_continuity} also reflect important principles of the conditioning of Volterra equations, as we can see through the following example:
\begin{example}\label{ex:easyvolterra}
    Consider the simple Volterra equation
    \begin{equation}\label{eq:easyvolterra}
        y(t) = \int_0^t x(\tau)\,d\tau,
    \end{equation}
    which fits into the class \ref{eq:integrodiff_PD} with $c_0=c_1=0$ and $\lambda = \delta_0\in\cM_+(\RR)$. On one hand, the results of \cref{prop:main_CM} yield the familiar closed-form solution $x(t) = \dot{y}(t)$; since the interconversion formula is exact, any numerical error in solving \eqref{eq:easyvolterra} is folded into computing the time derivative of $y$. On the other hand, one could attempt to solve the equation using an appropriate quadrature scheme. For instance, the discretized equation
    \begin{equation}\label{eq:baddiscretization}
        y(t) = \eps\left(\tfrac{1}{2}x(t) + x(t-\eps) + \cdots + x(t-\eps\lfloor t/\eps\rfloor)\right)
    \end{equation}
    is solved in closed form (up to rescaling) in \cref{sec:diraccomb}, with solution
    \[x(t) = \frac{2}{\eps}y(t) - \frac{4}{\eps}y(t-\eps) + \frac{4}{\eps}y(t-2\eps)\pm\cdots\pm \frac{4}{\eps}y(t-\eps\lfloor t/\eps\rfloor).\]
    This solution converges pointwise to the true limit in certain cases, and the local mean of the solution converges more generally. In any case, this approximation is far from the $W_{-2}$ convergence guarantee of \cref{thm:main_continuity}, in either the spectral domain or the time domain.

    That such a discretization fails to converge is well-known, and related to the ill-posedness of Volterra equations of the first kind. \Cref{thm:main_continuity} provides two hints as to why this discretization might fail. For one, even though \cref{eq:easyvolterra} has leading coefficient $c_0 = 2/\eps\neq 0$, the limit has $c_0=0$; the sequence thus tends to the boundary of $U^1$, outside the continuity guarantees of \cref{thm:main_continuity}. Secondly, even though the kernels of \cref{eq:baddiscretization} are positive definite (in the `regularized' sense of the following section), their Fourier transforms are not compactly supported, so the sequence cannot converge in $\cM_c(\RR)$. We will be able to make sense of this weaker form of convergence in the following subsection, using the regularized Hilbert transform (see \cref{rem:coolremark}).
\end{example}

\subsection{Generalized Delay and Fractional Differential Equations}\label{sec:regularizedH}
Finally, we \dave{treat} the fully general case of measures $\lambda\in\cM_+^{(2)}(\RR)$, for which the interconversion map $\cB_\RR$ is not necessarily defined. In this case, we can still apply the map $\cB$ of \cref{thm:main_circ_formula} to recover a interconversion formula in $\cM_+^{(2)}(\RR)$, but we can no longer guarantee that the result lies in $\cM_+^{(1)}(\RR)$. As such, we cannot make use of the standard Hilbert transform \cref{eq:hilberttransform} on the real line, so the application to Volterra equations requires more care.

As a first step, we note a critical \dave{element} of our circle theory: the relation \cref{eq:H_real2circ} \dave{indicates} how the Hilbert transform can be regularized to apply to functions $f\in L^\infty(\RR)$, a result \dave{first discovered in the singular integral operator theory of Calder\'on and Zygmund}~\cite{10.1007/BF02392130}. Namely, for any bounded $f\in L^\infty(\RR)$, the image of $\lambda = f(s)\,ds$ under $\psi$ is simply
    \[\psi[\lambda] = (2\pi)^{-1}f(\phi(e^{i\theta}))\,d\theta .\]
    This is a continuous measure with bounded density, so it must lie in $\cM(S^1)$. Pulling back the Hilbert transform $H[\psi[\lambda]]$ yields (in fact, for any $\lambda\in\cM_+^{(2)}(\RR)$)
    \begin{equation}\label{eq:regularizedH}
        H_\mathrm{reg}[\lambda](t)\doteq H[\psi[\lambda]](\phi^{-1}(t)) = \pv\int \frac{1}{\pi}\left(\frac{1}{t-s} + \frac{s}{1+s^2}\right)\,d\lambda(s),
    \end{equation}
    refraining now from splitting the integral because we generically have $\lambda\notin\cM_+^{(1)}(\RR)$. \dave{This notion agrees (up to an additive constant) with the standard Hilbert transform where the latter is defined, and it extends to a regularized Cauchy transform}
    \begin{equation}\label{eq:regularizedQ}
        Q_\mathrm{reg}[\lambda](z)\doteq Q[\psi[\lambda]](\phi^{-1}(z)) = \int \frac{i}{\pi}\left(\frac{1}{z-s} + \frac{s}{1+s^2}\right)\,d\lambda(s)
    \end{equation}
    on the upper half-plane. By \cref{cor:uniqueness}, $Q_\mathrm{reg}[\lambda](z)$ is uniquely defined within the family $Q_\sigma[\psi[\lambda]](\phi^{-1}(z))$, $\sigma\in\RR$, by the property that $\Im Q_S[\lambda](i) = 0$.

    At present, we aim to understand how the regularized Hilbert transform can extend the class of Volterra equations covered by our theory. There are two directions we can take this investigation, which correspond to (generalized classes of) delay differential equations and fractional differential equations, respectively. 

First, we develop an analogue of \cref{thm:main_circ_formula} for the regularized transform $Q_\mathrm{reg}$. To state this result, we make use of the following, regularized form of~\cref{eq:Psi_map}:
\begin{equation}\label{eq:embedding_reg}
\begin{gathered}
\Psi_\mathrm{reg}:\cM_+^{(2)}\times\RR\times\RR_+\to \cM_+(S^1)\times\RR,\\
    \Psi_\mathrm{reg}[\lambda,c_0,c_1] = (\psi[\lambda] + \pi^{-1}c_1\delta_{-1}, -\pi^{-1}c_0).
    \end{gathered}
\end{equation}
The following result can be deduced straightforwardly from \cref{thm:main_circ_formula}; we discussed such a result at the beginning of \cref{sec:main_line}, but it is instructive to formalize it in terms of $Q_\mathrm{reg}$:
\begin{theorem}[Definition of $\cB_\mathrm{reg}$]\label{thm:main_rPD}
    Suppose $\lambda\in\cM_+^{(2)}(\RR)$, and fix $c_1\geq 0$ and $c_0\in\RR$. There is a unique measure $\mu\in\cM_+^{(2)}(\RR)$ and unique values $\zeta_1\geq 0$ and $\zeta_0\in\RR$ such that
    \[(Q_\mathrm{reg}[\lambda](z) - i\pi^{-1}(c_0+c_1z))(Q_\mathrm{reg}[\mu](z) - i\pi^{-1}(\zeta_0+\zeta_1z))\equiv 1\]
    for $z\in\HH$. In this context, we write $\cB_\mathrm{reg}[\lambda,c_0,c_1] = (\mu,\zeta_0,\zeta_1)$. The map $\cB_\mathrm{reg}$ is continuous in the pullback of the weak topology on $\cM_+(S^1)$ under $\Psi_\mathrm{reg}$.

    If $\lambda$ is even and $c_0=0$, then $\mu$ is even and $\zeta_0=0$.
\end{theorem}
\begin{remark}\label{rem:coolremark}
    The topological statement of this theorem is distinct from the $W_{-2}$ topology of \cref{def:winf} in the following way. Consider a sequence $\theta_j\in [0,\pi)$ converging to $\pi$, and consider the Dirac measures $\delta_{e^{i\theta_j}}\in\cM_+(S^1)$ converging weakly to $\delta_{-1}$. These atoms pull back under $\Psi_\mathrm{reg}$ to the measures
    \[\pi \sec^2(\theta_j/2)\,\delta\left(s - \tan(\theta_j/2)\right)\,ds\in\cM_+^{(2)}(\RR),\]
    which do not converge in $W_{-2}$. In the pullback of the weak topology on $\cM_+(S^1)$, however, these measures converge to the pair $\lambda=0$, $c_1=\pi$.

    Another interesting example is that of \cref{ex:easyvolterra}. In the pullback of the spectral domain to $S^1$, the approximations \cref{eq:baddiscretization} converge weakly to $\lambda\doteq\delta_1\in\cM_+(S^1)$, exactly corresponding to the equation \cref{eq:easyvolterra}. In turn, $\lambda$ maps to $\mu\doteq\delta_{-1}\in\cM_+(S^1)$ under $\cB$, yielding the interconverted equation $\dot{y} = x$. We thus see that the regularized interconversion map allows us to make sense of limits that \dave{previously appeared singular}.

    As a final note, we could alternatively state the theorem in terms of convergence in the $W_{-r}$ topology for $r>2$, as we did in \cref{thm:main_continuity}, but this choice no longer illustrates the asymptotic behavior of \dave{our involution}.
\end{remark}

Likewise, we can recover a closed-form formula for $\cB_\mathrm{reg}$ over a wide class of measures $\lambda$. Pulling back the proof of \cref{thm:main_circ_formula}, we find the following result:
\begin{theorem}[Closed form of $\cB_\mathrm{reg}$]\label{thm:main_regularizedH_formula}
    Suppose $\lambda\in\cM_+^{(2)}(\RR)$, fix $c_0\in\RR$ and $c_1\geq 0$, and write 
    \[\Psi_\mathrm{reg}[\lambda,c_0,c_1] = (\widetilde{\lambda},\widetilde{c}_0),\qquad\cB_\mathrm{reg}[\lambda,c_0,c_1] = (\mu,\zeta_0,\zeta_1).\]
    Suppose that
    \[Z'\doteq \left(N_0(\widetilde{\lambda})\cap\supp\widetilde{\lambda}\right)\cup\{z\notin\supp\lambda\;|\;H[\widetilde{\lambda}](z) + i\widetilde{c}_0= 0\}\]
    is discrete (i.e., it does not contain any of its limit points), and write $\lambda_c$ for the continuous density of $\lambda$. Then we find
    \[
        d\mu(s) = \mu_c(s)\,ds + \sum_{\alpha_i\in Z} \beta_i\delta(s-\alpha_i)\,ds,\]
    with the following identities:
    \[
        \mu_c(s) = \frac{\lambda_c(s)}{\lambda_c(s)^2 + \big(H_\mathrm{reg}[\lambda](s) - \pi^{-1}(c_1s+c_0)\big)^2},\]
        \[\beta_i = \pi^2\left(c_1 + \int\frac{d\lambda(\tau)}{(\tau-\alpha_i)^2}\right)^{-1},
    \]
    \[
        Z = N_0(\lambda)\cap\big\{s\in\RR\;\big|\;H_\mathrm{reg}[\lambda](s) - \pi^{-1}(c_1s + c_0)= 0\big\}.
    \]
    Furthermore, we have
    \[\zeta_0 = -\pi^2\Im\left(\int\frac{d\lambda(s)}{1+s^2} + c_1 - ic_0\right)^{-1}.\]
    Finally, if $c_0=c_1=0$, then we have
    \[\zeta_1 = \frac{\pi^2}{\|\lambda\|},\]
    taking $\zeta_1=0$ if $\|\lambda\|=\infty$. If either of $c_0$ or $c_1$ is nonzero, then $\zeta_1=0$.
\end{theorem}

We split now into two cases. First, we study the setting \ref{eq:integrodiff_rPD}, which generalizes \ref{eq:integrodiff_CM} to account for delay terms. Indeed, it is easy to see that rPD kernels correspond to inverse Laplace transforms in $\cM_+^{(2)}(\RR)$:
\begin{remark}
        From Bochner's theorem (\cref{lem:bochner}), a kernel $K$ is \emph{rPD} if and only if $K = \cF[\lambda]$ for some $\lambda\in\cM_+^{(2)}(\RR)$.
    \end{remark}
For the sake of clarity, we have phrased \ref{eq:integrodiff_rPD} only in the case that $\lambda$ is even, corresponding to a {real} rPD kernel $K=\cF[\lambda]$. We note that the class \ref{eq:integrodiff_dPD} can be extended more broadly---for instance, our analysis works equally well when $\lambda = \lambda_e + \lambda_o$ for an even $\lambda_e\in\cM_+^{(2)}(\RR)$ and an odd $\lambda_o\in\cM_+^{(1)}(\RR)$. One could consider an even broader class of measures, where $\sigma_\RR(\lambda)$ diverges, though we do not treat it here.

With only mild regularity requirements on $K$, the map $\cB_\mathrm{reg}$ allows us to solve \ref{eq:integrodiff_rPD} in much the same way as our other classes of integro-differential equations. \dave{We prove the following in \cref{sec:laplace}:}
\begin{restatable}[Solution of \ref{eq:integrodiff_rPD}]{proposition}{propmainrPD}\label{prop:rPD}
Suppose $K:\RR\to\CC$ is a rPD kernel for which $\lambda\doteq\cF^{-1}[K]$ is even, and fix $c_1\geq 0$. Write $\cB_\mathrm{reg}[\lambda,0,c_1] = (\mu,0,\zeta_1)$ and $J = \cF[\mu]$. If $K$ and $J$ both restrict to measures in a neighborhood of the origin, then \ref{eq:integrodiff_rPD} is satisfied by
    \[\pi^{2}x(t) = \zeta_1\dot{y}(t)  + \frac{1}{2}\int_{-t}^t J(\tau)y(|t-\tau|)\,d\tau + c_1x_0J(t).\]
\end{restatable}

As discussed in \cref{sec:history}, the class \ref{eq:integrodiff_rPD} contains a wide variety of delay differential equations:
\begin{example}\label{ex:rPD}
    Consider the equation
    \[y(t) = c_1\dot{x}(t) + x(t) + x(t-1)\]
    with $c_1>0$. This falls into the class \ref{eq:integrodiff_rPD} with
    \[K(t) = 2\delta(t) + \delta(t-1),\qquad d\lambda(s) = \pi^{-1}(1+\cos s)\,ds\in\cM_+^{(2)}(\RR).\]
    We can calculate $H_\mathrm{reg}[\lambda](s) = \pi^{-1}\sin s$, and thus
    \[d\mu(s) = \frac{\pi(1+\cos s)\,ds}{(1+\cos s)^2 + (\sin s - c_1s)^2}.\]
    This expression is $L^1$-integrable, so we can define the Fourier transform as $J(t) = \int e^{-ist}d\mu(s)$. This example is shown in \cref{fig:main_inversion}.
\end{example}

It also allows us to solve \emph{negative} CM equations\footnote{Although the negative CM class is a strict subset of the gPD class of \cref{sec:main_line}, such equations do not generally satisfy the hypotheses of Theorems~\ref{thm:main_formula} or \ref{thm:main_continuity}, so they must be treated with our more general rPD theory.}---i.e., equations of the form
\[y(t) = c_1\dot{x}(t) + \int_0^t K(t-\tau)x(\tau)\,d\tau,\]
where $c_1\geq 0$ and $K$ is CM (but not gCM). To see how, consider how the Fourier transform acts on a Cauchy distribution:
\[\cF\left[t\mapsto \frac{a}{a^2+t^2}\right](s) = \pi e^{-a|s|},\]
where $a>0$. \dave{Given $\lambda_\mathrm{CM}\in\cM_+^{(1)}(\RR)$ supported on $\RR_+$, one can show that}
\[\cF\left[t\mapsto \frac{1}{\pi}\int \frac{a\,d\lambda_\mathrm{CM}(a)}{a^2+t^2}\right](s) = \int e^{-a|s|}\,d\lambda_\mathrm{CM}(a) = \cL[\lambda_\mathrm{CM}](s)\]
for $s>0$. More simply, we can write
\[\cF[t\mapsto \Im Q_\RR[\lambda_\mathrm{CM}](it)] = \cL[\lambda_\mathrm{CM}],\]
allowing us to represent generic CM kernels as Fourier transforms of non-negative functions (i.e., \dave{as} PD kernels).

Although `negative' CM equations represent only a sign flip from the \ref{eq:integrodiff_CM} class, we see now that they are best understood within the class of PD kernels. In particular, we see from \cref{prop:rPD} that the interconversions of such equations are themselves in the rPD class, but do \emph{not} necessarily feature CM kernels themselves. This result explains why the program of Hannsgen and Wheeler~\cite{doi:10.1137/0513067} fails to find a CM resolvent to such equations, and---at least in the scalar case---it characterizes the resolvents that \emph{can} arise.

\begin{example}\label{ex:negativeCM}
    Consider the equation
    \[y(t) = -\sum_{i=1}^n b_i\int_0^t e^{-a_i(t-\tau)}x(\tau)\,d\tau,\]
    where $a_i,b_i>0$. This equation can easily be recast in the form \ref{eq:integrodiff_CM}, but we treat it now as an equation of the form \ref{eq:integrodiff_PD} in order to understand how rPD kernels can arise in the resolvent equation.
    
    From the argument above, a finite sum of exponentials in the time domain corresponds to a weighted sum of Cauchy distributions in the spectral domain:
    \begin{equation}
        d\lambda(s) = \sum_{i=1}^n\frac{1}{\pi}\frac{b_ia_i}{s^2 + a_i^2}\,ds,\qquad H_\mathrm{reg}[\lambda](s) = H_\RR[\lambda](s) = \sum_{i=1}^n\frac{1}{\pi}\frac{b_ia_i s}{s^2 + a_i^2}.
    \end{equation}
    Now, it is important to note that this kernel does \emph{not} satisfy the hypotheses of \cref{thm:main_formula}, as it decays too slowly to lie in $L^*(\RR)$. As such, we need to use the more general theory of rPD kernels to handle it. From \cref{thm:main_regularizedH_formula}, we find
    \[\zeta_1 = \frac{\pi^2}{\sum_{i=1}^n b_i},\qquad \zeta_0 =0,\qquad d\mu(s) = \frac{\pi}{1+s^2}\,\bigg(\sum_{i=1}^n\frac{b_ia_i}{s^2+a_i^2}\bigg)^{-1}\,ds.\]
    In particular, we have
    \[d\mu(s) = \frac{\pi}{\sum_{i=1}^n a_ib_i}\,ds - \widetilde{\mu}_c(s)\,ds,\]
    where $\mu_c(s) = O(s^{-2})$. Applying \cref{prop:rPD} to map these expressions back to the time domain, we find
    \[-\pi^2x(t) = \zeta_1\dot{y}(t) + \widetilde{\zeta}_0y(t) - \int_0^t \widetilde{J}(t-\tau)y(\tau)\,d\tau,\]
    where $\widetilde{\zeta}_0 = \pi^2\big(\sum_{i=1}^n a_ib_i\big)^{-1}$ and $\widetilde{J} = \cF[\widetilde{\mu}_c]$. Already, we can see that the expressions for $\zeta_1$ and $\widetilde{\zeta}_0$ agree with the results of \cref{prop:main_CM}. The same is true of $\widetilde{J}$, of course, though we do not investigate the matter further at present.
\end{example}

In another direction, we can extend the class of CM equations to incorporate a generalized class of \emph{fractional differential equations}. For this, we define the following subset of $\cM_\mathrm{exp}^{(2)}(\RR)$:
\begin{definition}
    Given $\lambda\in\cM_\mathrm{exp}^{(2)}(\RR)$, we say that $\lambda\in\cM_\mathrm{frac}(\RR)$ if \dave{$\supp\lambda\subset\RR_+$} and if $t^{-1}\,d\lambda(t)\in\cM_{+,\mathrm{loc}}(\RR)$, or equivalently, if the restriction of $t^{-1}\,d\lambda(t)$ to a neighborhood of $t=0$ is a finite measure. If $\lambda\in\cM_\mathrm{frac}(\RR)$, we define
    \[\xi_\mathrm{frac}(\lambda) = \frac{1}{\pi}\int\frac{d\lambda(s)}{s(1+s^2)}\in\RR.\]
\end{definition}

We prove the following result in \cref{sec:laplace}:
\begin{restatable}[Solution of \ref{eq:integrodiff_rCM}]{proposition}{propmainrCM}\label{prop:rCM}
Suppose $K_1=\cL_b[\lambda_1]$ is a gCM kernel with $\lambda_1\in\cM_\mathrm{exp}^{(1)}(\RR)$, and
\[K_2(t) = \cL[s^{-1}\,d\lambda(s)](t) = \int e^{-ts}s^{-1}\,d\lambda_2(s)\]
for some $\lambda_2\in\cM_\mathrm{frac}(\RR)$. Fix $c_0\in\RR$ and $c_1\geq 0$, and write 
\[\cB_\mathrm{reg}[\lambda_1+\lambda_2,c_0-\pi\sigma_\RR(\lambda_1) - \pi\xi_\mathrm{frac}(\lambda_2),c_1] = (\mu, \zeta_0',\zeta_1).\]
The measure $\mu\in\cM_\mathrm{exp}^{(2)}(\RR)$ can be decomposed as $\mu=\mu_1+\mu_2$, where $\mu_1\in\cM_\mathrm{exp}^{(1)}(\RR)$ and $\mu_2\in\cM_\mathrm{frac}(\RR)$. Given any such decomposition, let $J_1 = \cL_b[\mu_1]$ and $J_2=\cL[s^{-1}\,d\mu_2(s)]$, and write $\zeta_0 = \zeta'_0 + \pi\sigma_\RR(\mu_1) + \pi\xi_\mathrm{frac}(\mu_2)$. Then \ref{eq:integrodiff_rCM} is satisfied by
\begin{align*}-\pi^2x(t) &= \zeta_1\dot{y}(t) - \zeta_0y(t) - \int_0^t J_1(t-\tau)y(\tau)\,d\tau + \frac{d}{dt}\int_0^t J_2(t-\tau)y(\tau)\,d\tau\\
&\qquad -c_1x_0(J_1(t) - \dot{J}_2(t)).
\end{align*}
\end{restatable}

This result clarifies that the two kernels $K_1$ and $K_2$ in \ref{eq:integrodiff_rCM} should be seen as two components of the same object, corresponding to $\lambda = \lambda_1+\lambda_2$ in the spectral domain. The decomposition itself is generally non-unique, so only the sum of the two objects is fundamental to the equation.

\begin{example}\label{ex:fractional}
    Consider the fractional differential equation
    \[y(t) = \dot{x}(t) + D^{1/2}x(t)= \dot{x}(t) + \frac{1}{\sqrt{\pi}}\frac{d}{dt}\int_0^t \frac{x(\tau)}{\sqrt{t-\tau}}\,d\tau,\]
    defining the Riemann--Liouville fractional derivative as in \cref{eq:frac_deriv}. This is of the form \ref{eq:integrodiff_rCM} with $\lambda_1=0$ and 
    \[d\lambda_2(s) = \pi^{-1}\chi_{[0,\infty)}(s)\sqrt{s}\,ds,\]
    and we can verify from \cref{eq:regularizedQ} that
    \[Q_\mathrm{reg}[\lambda_2](z) = \pi^{-1}\sqrt{z} - \pi^{-1}2^{-1/2}i,\]
    with $\sqrt{z}$ denoting the principal value of the square root. Similarly, we find $\xi_\mathrm{frac}(\lambda_2) = \pi^{-1}2^{-1/2}$, so \cref{thm:main_regularizedH_formula} yields 
    \[\zeta_1 = \zeta_0 = 0,\qquad d\mu(s) = \frac{\pi}{s^{1/2}+s^{3/2}}\,\chi_{[0,\infty)}(s)\,ds.\]
    The Laplace transform of $\mu$ is the Mittag--Leffler kernel~\cite{haubold2011mittag}
    \begin{equation}\label{eq:mittagleffler}
        \cL[\mu](s) = \pi^2 E_{1/2}(-t^{1/2}),\qquad E_\alpha(z) \doteq \sum_{k=0}^\infty \frac{z^k}{\Gamma(\alpha k + 1)},
    \end{equation}
    which gives the classical result~\cite[Sec.~7]{haubold2011mittag}
    \[x(t) = \int_0^t E_{1/2}(-(t-\tau)^{1/2})y(\tau)\,d\tau.\]
    It has been previously noted that the Mittag--Leffler kernel is completely monotone~\cite{miller1999note}, but the example presented here highlights the critical importance of that property. We solve this example numerically in \cref{sec:fractional}.
\end{example}

As a final note, one can also solve \emph{Abel-type} integral equations using the same procedure:
\begin{example}[Abel's integral equation]\label{ex:abeltype}
    \dave{Consider Abel's integral equation}~\cite{BRUNNER199783}:
    \begin{equation}\label{eq:abeltypeeq}
        y(t) = \frac{1}{\Gamma(\alpha)}\int_0^t\frac{x(\tau)}{(t-\tau)^{\alpha}}\,d\tau,
    \end{equation}
    where $0<\alpha<1$. This equation is of the class \ref{eq:integrodiff_CM}, with $c_0=c_1=0$ and integral kernel
    \[K(t) = \Gamma(\alpha)^{-1}t^{-\alpha} = \cL[\lambda](t),\qquad \lambda = \Gamma(\alpha)^{-2}s^{\alpha-1}\chi_{[0,\infty)}(s)\,ds \doteq \lambda_c(s)\,ds.\]
    Even though $\lambda\in\cM_+^{(1)}(\RR)$, it does not satisfy the hypotheses of either \cref{thm:main_formula} or \cref{thm:main_continuity}, and we must use the theory of the present section to solve it.

    In the language of \cref{prop:rCM}, we have $\lambda_2=0$ and $c_0=c_1=0$, so we look to compute
    \[(\mu,\zeta_0',\zeta_1) = \cB_\mathrm{reg}[\lambda,-\pi\sigma_\RR(\lambda),0].\]
    As a first step, we note that the Cauchy transform of $\lambda$ is
    \[Q_\RR[\lambda](z) = \frac{e^{-i\alpha\pi} z^{\alpha-1}}{i\sin(\alpha\pi)\Gamma(\alpha)^2}.\]
    Indeed, this function pulls back to a holomorphic function on $\DD$ with the appropriate radial limit and with $Q_\RR[\lambda](\infty) = Q_\RR[\lambda](\phi(-1)) = 0$, so \dave{the expression for $Q_\RR[\lambda]$} follows from \cref{cor:uniqueness}. Taking the imaginary trace along $\RR$, we find
    \[H_\mathrm{reg}[\lambda](s) +\sigma_\RR(\lambda) = H_\RR[\lambda](s) = \begin{cases}
        s>0: & \Gamma(\alpha)^{-2}\cot(\alpha\pi)|s|^{\alpha-1}\\
        s<0: & -\Gamma(\alpha)^{-2}\csc(\alpha\pi)|s|^{\alpha-1}
    \end{cases}.\]
    From \cref{thm:main_regularizedH_formula}, we see that $\mu = \mu_c(s)\,ds$ has no singular terms, and its continuous density is
    \[\mu_c(s) = \frac{\lambda_c(s)}{\lambda_c(s)^2 + H_\mathrm{reg}[\lambda](s)^2} = \Gamma(\alpha)^2\sin^2(\alpha\pi)\chi_{[0,\infty)}(s) s^{1-\alpha}\,ds.\]
    It is easy to confirm that $\zeta_0=\zeta_1=0$, and we evaluate
    \[J(t) = \cL[s^{-1}\,d\mu(s)](t) = \Gamma(1-\alpha)\Gamma(\alpha)^2\sin^2(\alpha\pi) t^{\alpha-1} = \frac{\pi^2}{\Gamma(1-\alpha)t^{1-\alpha}}.\]
    Comparing against \cref{prop:rCM}, we recover the classical solution
    \[x(t) = \frac{1}{\Gamma(1-\alpha)}\frac{d}{dt}\int_0^t\frac{y(\tau)}{(t-\tau)^{1-\alpha}}\,d\tau.\]
\end{example}

\Cref{ex:abeltype} is particularly striking in light of the gCM theory of \cref{sec:main_line}. The gCM theory offers a natural stratification of the three core subclasses of \ref{eq:integrodiff_CM}: first-kind integral equations (with $c_0=c_1=0$), second-kind integral equations (with $c_1=0$ but $c_0\neq 0$), and integro-differential equations (with $c_1\neq 0$). Under the hypotheses laid out in Theorems~\ref{thm:main_formula} and~\ref{thm:main_continuity}, we saw how $\cB_\RR$ interchanges these three strata: it pairs second-kind integral equations with other second-kind integral equations (reducing to the results of Loy \& Anderssen~\cite{loy_anderssen}) and pairs first-kind integral equations with integro-differential equations.

The regularized theory of the present section blurs the lines between these strata. Although the Abel-type equation \cref{eq:abeltypeeq} is \dave{a gCM equation} of the first kind, \dave{its spectrum} $\lambda\in\cM_+^{(1)}(\RR)$ carries substantial mass near infinity. As such, its interconversion \emph{cannot} carry a derivative term (corresponding to an `atom at infinity'), and cannot be a proper integro-differential equation. As we saw above, it instead picks up a fractional derivative; just as we anticipated in \cref{rem:Br_failure}, (generalized) fractional derivatives `complete' the definition of $\cB_\RR$ in a natural way, but discard the neat stratification of gCM equations in the process.

%% file: 5_laplace.tex
\section{Volterra Equations in the Spectral Domain}\label{sec:laplace}
\cref{prop:main_dPD,prop:main_CM,prop:main_PD,prop:rPD,prop:rCM} provide the connecting link between our harmonic analysis in later sections and the Volterra equations of interest. We prove all five in the present section.

The first of these five results relates measures on the circle to discrete-time Volterra equations, using power series expansions. This equivalence is otherwise known as the \emph{Z-transform} in signal processing~\cite{bhattacharyya2018handbook}; if $y= \{y_0,y_1,...\}\subset\CC$ is a discrete signal, the Z-transform $Y(z)$ of $x$ can be defined as
\[Y(z) = \cZ[y](z) \doteq \sum_{j\geq 0} y_j z^j,\]
as a formal power series\footnote{The usual convention for the Z-transform constructs a power series in $z^{-1}$ rather than $z$. Our convention ensures that time series are mapped to holomorphic functions on the disc, rather than its exterior.}. That $Y$ converges for any non-zero $z$ is not guaranteed, of course. Fortunately, the equation \ref{eq:integrodiff_dPD} is \emph{causal}, so the value $x(n)$ depends only on the finite set $\{y(0),...,y(n)\}$. As such, we can safely restrict to cases where $y$ is a finite time series, and thus $Y(z)$ is a polynomial in $z$. We recall the statement of \cref{prop:main_dPD}:
\propmaindPD*
\begin{proof}
    We assume without loss of generality that the prescribed change of parameters $c_0\mapsto c'_0$, $K(n)\mapsto K'(n)$ has already been performed, so that $\Re c_0 = -\frac{1}{2}K(0)$. We assume also that $y(j)$ has only finitely many nonzero values; since $x(n)$ depends only on $y(j)$ for $j\leq n$, the general formula follows directly.
    
    Let $Y(z)$ and $X(z)$ be the Z-transforms of $y(n)$ and $x(n)$, respectively. Then we find
    \begin{equation}\label{eq:Ztransform_proof}
        Y(z) = \left(c_0 + \cZ[K](z)\right)X(z)
    \end{equation}
    formally. In turn, since $K = \cF[\lambda]$ for $\lambda\in\cM_+(S^1)$, we note that
    \[|K(n)| = \left|\int_0^{2\pi} e^{-in\theta}\,d\lambda(\theta)\right| \leq \|\lambda\|,\]
    so that, in particular, $\cZ[K](z)$ converges absolutely for each $z$ in the open unit disc $\DD$. We thus find
    \[\cZ[K](z) = \sum_{j\geq 0}\int_0^{2\pi} e^{-ij\theta}z^j\,d\lambda(\theta) = \int_0^{2\pi}\frac{1}{1-ze^{-i\theta}}\,d\lambda(\theta) = \frac{1}{2}Q[\lambda](z) + \frac{1}{2}\|\lambda\|,\]
    employing \dave{a partial fraction decomposition} in the last step. Since $\Re c_0 = -\frac{1}{2}K(0) = -\frac{1}{2}\|\lambda\|$, this reduces \cref{eq:Ztransform_proof} to
    \[Y(z) = \frac{1}{2}\left(Q[\lambda](z) + i\widetilde{c}_0\right)X(z),\]
    where $\widetilde{c}_0 = 2\Im c_0 = -2ic_0 - iK(0)$. If $\cB[\lambda,\widetilde{c}_0] = (\mu,\widetilde{\zeta}_0)$ for some $\mu\in\cM_+(S^1)$ and $\widetilde{\zeta}_0\in\RR$, then
    \[\left(Q[\mu](z) + i\widetilde{\zeta}_0\right) 2Y(z) = X(z),\]
    implying as well that $X(z)$ converges in $\DD$. Working the same logic backwards proves the formula.
\end{proof}

The continuous-time results follow a similar line of reasoning, but using the (bilateral) Laplace transform in place of the Z-transform. We prove both of these results in the case $x_0=0$; the general case follows by considering forcing terms $y(t)$ with Dirac delta functions at $t=0$.
\propmainCM*
\begin{proof}
    Let $Y(s)$ and $X(s)$ be the Laplace transforms of $y(t)$ and $x(t)$, respectively; we suppose that $y(t)$ is growing at most exponentially in $t$. Applying a Laplace transform to~\ref{eq:integrodiff_CM} yields
    \[Y(s) = \left(c_1s - c_0 - \cL[K](s)\right)X(s),\]
    and, applying Fubini's theorem, we calculate
    \[\cL[K](s) = \int_0^\infty e^{-t s}\int e^{-\sigma t}\,d\lambda(\sigma)dt = \int \int_0^\infty e^{-(\sigma+s) t}\,dt  d\lambda(\sigma) = \int\frac{d\lambda(\sigma)}{s + \sigma}\]
    for $\Re s> -\inf\supp\lambda$; the latter integral converges by our hypothesis that $\lambda\in\cM_+^{(1)}(\RR)$. Noting that
    \[c_1s - c_0 - \cL[K](s) = -i\pi\left(Q_\RR[\lambda](-s) - i\pi^{-1}c_0 + i\pi^{-1}c_1s\right),\]
    the result follows\footnote{More precisely, we deduce that the correct formula holds in a quadrant of the Laplace domain, where $\Re s>-\inf\supp\lambda$ and $\Im s\leq 0$. Standard uniqueness results for the Laplace transform yield the full proposition.}.
\end{proof}

The proof of \cref{prop:main_PD} is complicated only by the fact that the Fourier transform might not exist classically when $\lambda$ is not a finite measure. By interpreting the transform weakly, we push the result through similarly.
\propmainPD*
\begin{proof}
    As before, we find
    \[Y(s) = \left(c_1s -ic_0 + \cL[K](s)\right)X(s),\]
    but now,
    \[\cL[K](s) = \int_0^\infty e^{-t s}\cF[\lambda](t)\,dt = \int_{-\infty}^{\infty}u(t)e^{-ts}\cF[\lambda](t)\,dt,\]
    where $u(t)$ is Heaviside's step function. If we knew that $\lambda$ was finite (i.e., $\lambda\in\cM_+(\RR)$), we could complete the proof in much the same way as that of \cref{prop:main_CM}, using an integral form of $\cF[\lambda]$; as it stands, however, we need to interpret $\lambda$ as a tempered distribution and employ the Plancherel theorem. Consider the family of Schwartz functions 
    \[\eta_{\eps,s}(t) = \frac{1}{\sqrt{2\pi}\eps}\int_{0}^\infty e^{-t's -(t-t')^2/2\eps^2}\,dt'\]
    converging to $u(t)e^{-ts}$ pointwise; for any $s$ with $\Re s > 0$, we find that
    \[\int_{-\infty}^{\infty}\eta_{\eps,s}(t)\cF[\lambda](t)\,dt = \int \cF[\eta_{\eps,s}](-t)\,d\lambda(t) = \int \frac{e^{-\eps t^2/2}}{s+it}\,d\lambda(t),\]
    and thus, by dominated convergence, that
    \[\cL[K](s) = \int \frac{d\lambda(t)}{s+it}.\]
    The remainder of the proof follows as before.
\end{proof}

We turn now to our two results relating the solution of \ref{eq:integrodiff_rPD} and \ref{eq:integrodiff_rCM} to the \emph{regularized} Hilbert transform, as discussed in \cref{sec:regularizedH}. The first of these employs the distributional Fourier transform, so it requires a similar convergence argument as used in the proof of \cref{prop:main_PD}:
\propmainrPD*
\begin{proof}
    Since $K$ restricts to a measure in the neighborhood of $t=0$, we can define $\alpha = K(\{0\})$ as the measure of $K$ at $0$. Then we can rewrite \ref{eq:integrodiff_rPD} as
    \[y(t) = c_1\dot{x}(t) - \frac{\alpha}{2} x(t)  + \int_{0}^t K(\tau)x(t-\tau)\,d\tau,\]
    with the integral taken over the closed interval $[0,t]$. In the Laplace domain, we thus find
    \[Y(s) = (c_1s - \alpha/2 + \cL[K](s))\,X(s),\]
where
\[\cL[K](s) = \int_0^\infty e^{-st}K(t)\,dt = \frac{1}{2}\int_{-\infty}^\infty e^{-s|t|}K(t)\,dt + \frac{\alpha}{2},\]
    using the fact that $K$ is even. Now define the family of Schwartz functions 
    \[\eta_{\eps,s}(t) = \frac{1}{\sqrt{2\pi}\eps}\int_{-\infty}^\infty e^{-|t'|s -(t-t')^2/2\eps^2}\,dt',\]
    converging to $e^{-|t|s}$ pointwise. As before, for any $s$ with $\Re s>0$, we find that
    \[\int_{-\infty}^{\infty}\eta_{\eps,s}(t)K(t)\,dt = \int \cF[\eta_{\eps,s}](-t)\,d\lambda(t) = \int \frac{2se^{-\eps t^2/2}}{s^2+t^2}\,d\lambda(t),\]
    and again by dominated convergence that
    \[\cL[K](s) = \int \frac{s\,d\lambda(t)}{s^2+t^2} + \frac{\alpha}{2}.\]
    Since $\lambda$ is even, however, we find\begin{align*}\int\frac{s\,d\lambda(t)}{s^2+t^2}& = \frac{i}{2}\int \left(\frac{1}{is-t} + \frac{1}{is+t}\right)\,d\lambda(t)\\
    &= \frac{i}{2}\int\left(\frac{1}{is-t}-\frac{t}{1+t^2} + \frac{1}{is+t}+\frac{t}{1+t^2}\right)\,d\lambda(t)\\
    &= iQ_\mathrm{reg}[\lambda](is),
    \end{align*}
    and the proof follows as before.
\end{proof}

Finally, we prove \cref{prop:rCM}, which involves two, distinct integral kernels. This proof makes non-trivial use of the spectral theory developed in later sections---this does not cause a conflict, however, as the following result is not used to develop any of the theory that follows.
\propmainrCM*
\begin{proof}
Taking the Laplace transform of \ref{eq:integrodiff_rCM}, we find
\[Y(s) = \left(c_1s - c_0 - \cL[K_1](s) + s\cL[K_2](s)\right)X(s).\]
The expression $\cL[K_1]$ has been calculated in the proof of \cref{prop:main_CM}, and we similarly find
\[s\cL[K_2](s) = \int\frac{s\sigma^{-1}}{s+\sigma}\,d\lambda(\sigma) = \int \left(\frac{1}{\sigma} - \frac{1}{s+\sigma}\right)\,d\lambda(\sigma) = \pi\xi_\mathrm{reg}(\lambda) - i \pi Q_\mathrm{reg}[\lambda](-s).\]
The remainder of the proof goes through as before. The only statement to verify is that $\mu\in\cM_\mathrm{exp}^{(2)}(\RR)$, which will follow from our spectral theory (which does not depend upon the present result); indeed, \cref{thm:main_regularizedH_formula} implies that $\mu\in\cM_+^{(2)}(\RR)$, and \cref{prop:support} implies that $\supp\mu$ is bounded below if $\supp\lambda$ is.
\end{proof}

%% file: 6_topologies.tex
\section{Topologies on Spaces of Volterra Equations}\label{sec:topologies}
By placing appropriate topologies on \dave{each of our classes of Volterra equations}, we can understand how Volterra equations can be continuously mapped to one another and interconverted. In the present section, we introduce the topologies we employ for each class of equations, discuss how they relate to one another, and show how these topologies can be understood in both the time and spectral domains.

The most basic example is that of discrete-time Volterra equations with positive definite kernels~\ref{eq:integrodiff_dPD}, \dave{which} can be identified with pairs $(\lambda,c_0)\in\cM_+(S^1)\times\RR$. In \cref{thm:main_circle}, we show that interconversion of such equations corresponds to a map $\cB:(\lambda,c_0)\mapsto (\mu,\zeta_0)$, continuous with respect to the weak topology on $\cM_+(S^1)$ and the standard topology on $\zeta_0$. \dave{We can understand this continuity in the time domain using classical Fourier analysis}~\cite[Sec.~26]{billingsley2017probability}:
\begin{proposition}\label{prop:time_cont_circ}
    A sequence $\lambda_j\in\cM_+(S^1)$ converges weakly to $\lambda\in\cM_+(S^1)$ if and only if the discrete kernels $K_j = \cF[\lambda_j]$ converge pointwise to $K=\cF[\lambda]$.
\end{proposition}
The continuity statement of \cref{thm:main_circle} thus takes the following, perhaps obvious form: if the discrete PD kernels $K_j$ converge pointwise to $K$ and the real parameters $c_{0,j}$ converge to $c_0$, then the interconverted kernels $J_j = \cF[\mu_j]$ and parameters $\zeta_{0,j}$ converge pointwise to $(J,\zeta_0)$. \dave{Formalized properly, this argument would prove the continuity of $\cB$, but we give a direct proof in \cref{sec:disc} nonetheless.}

Inspired by this success, we attempt to carry out the same for integral and integro-differential equations. Recall from \cref{sec:main_line} that the class \ref{eq:integrodiff_CM} can be identified with triples $(\lambda,c_0,c_1)\in\cM_\mathrm{exp}^{(1)}(\RR)\times\RR\times\RR_+$. Our remaining classes of equations can be identified likewise, but with looser restrictions on $\lambda$: for \ref{eq:integrodiff_PD}, we require only that $\lambda\in\cM_+^{(1)}(\RR)$; for \ref{eq:integrodiff_rCM}, we require that $\lambda\in\cM_\mathrm{exp}^{(2)}(\RR)$; and for \ref{eq:integrodiff_rPD}, we require that $\lambda\in\cM_+^{(2)}(\RR)$ be \emph{even} and that $c_0=0$. As discussed in \cref{sec:regularizedH}, the latter class can be extended straightforwardly to \dave{all} of $\cM_+^{(2)}(\RR)$.

In all cases, we are interested in attaching topologies to the sets $\cM_+^{(n)}(\RR)$, which can be done as in \cref{def:winf}; \dave{we repeat it below for convenience:}
\defwinf*

We discuss these topologies in turn. For one, L\'evy's continuity theorem~\cite[Thm.~26.3]{billingsley2017probability} allows us to relate $W_{-n}$ naturally to pointwise convergence of (mollifications of) integral kernels. The case of \ref{eq:integrodiff_PD} and \ref{eq:integrodiff_rPD} goes as follows:
\begin{proposition}\label{prop:time_cont_PD}
    Suppose $\lambda_j,\lambda\in\cM_+^{(n)}(\RR)$, and write $K_j=\cF[\lambda_j]$ and $K=\cF[\lambda]$ for their (weak) Fourier transforms. Then $\lambda_j\to\lambda$ in the $W_{-n}$ topology if and only if $(1-d_t^2)^{-n/2}K_j\to (1-d_t^2)^{-n/2}K$ pointwise. In this case, the convergence $(1-d_t^2)^{-n/2}K_j\to (1-d_t^2)^{-n/2}K$ is locally uniform on $\RR$.
\end{proposition}

\begin{proof}
    \dave{The first claim follows from L\'evy's Continuity Theorem, noting that $\widetilde{\lambda}_j\doteq (1+s^2)^{-n/2}\,d\lambda_j(s)$ are finite measures converging weakly to $\widetilde{\lambda}\doteq(1+s^2)^{-n/2}\,d\lambda(s)$}.
    
    \dave{Next, let $\widetilde{K}_j = \cF[\widetilde{\lambda}_j] = (1-d_t^2)^{-n/2}K_j$ and $\widetilde{K} = \cF[\widetilde{\lambda}]$, and write $M = \sup\{|\widetilde{\lambda}_j|\}$. Fix $\eps>0$, choose an interval $I\subset\RR$ such that $\widetilde{\lambda}(\RR\setminus I)<\eps/13M$, and choose an $N\geq 1$ such that $\widetilde{\lambda}_j(\RR\setminus I)<\eps/12M$ for all $j\geq N$. Let $c=\sup\{|\omega|\;|\;\omega\in I\}$ and $\delta = \eps/6cM$. Fix $t\in\RR$. Increasing $N$ if necessary, we can assume that $|\widetilde{K}_j(t)-\widetilde{K}(t)|<\eps/3$ for all $j\geq N$. Then, for any $t'\in\RR$ with $|t-t'|<\delta$, we find
    \begin{align*}|\widetilde{K}_j(t') - \widetilde{K}_j(t)| &\leq \int_{I} |e^{-i\omega t'} - e^{-i\omega t}|\,d\widetilde{\lambda}_j(\omega) + \int_{\RR\setminus I} |e^{-i\omega t'} - e^{-i\omega t}|\,d\widetilde{\lambda}_j(\omega)\\
    &\leq cM|t-t'| + 2M(\eps/12M) \leq \eps/3,
    \end{align*}
    and similarly for $\widetilde{K}$. The proposition follows by noting that, for any $j\geq N$ and any $t'\in\RR$ with $|t-t'|<\delta$, we have 
    \begin{align*}|\widetilde{K}_j(t') - \widetilde{K}(t')|&\leq |\widetilde{K}_j(t') - \widetilde{K}_j(t)| + |\widetilde{K}_j(t) - \widetilde{K}(t)| + |\widetilde{K}(t) - \widetilde{K}(t')|\leq \eps.
    \end{align*}}
\end{proof}

\dave{An equivalent} statement for \ref{eq:integrodiff_CM} and \ref{eq:integrodiff_rCM} follows similarly:
\begin{proposition}\label{prop:time_cont_CM}
    Suppose $\lambda_j,\lambda\in\cM_\mathrm{exp}^{(n)}(\RR)$ have support uniformly bounded below, i.e., that there is a $M>0$ such that $\inf\supp\lambda_j,\inf\supp\lambda > -M$. Write $K_j=\cL_b[\lambda_j]$ and $K=\cL_b[\lambda]$ for their bilateral Laplace transforms. Then $\lambda_j\to\lambda$ in the $W_{-n}$ topology if and only if $(1+d_t^2)^{-n/2}K_j\to (1+d_t^2)^{-n/2}K$ pointwise on $[0,\infty)$. In this case, we also have locally uniform convergence $(1+d_t^2)^{k-n/2}K_j\to (1+d_t^2)^{k-n/2}K$ on $(0,\infty)$ for any $k\geq 0$. In particular, if $n\geq 0$, we have locally uniform convergence $d_t^kK_j\to d_t^k K$ on $(0,\infty)$ for any $k\geq 0$.
\end{proposition}
\begin{proof}
    The equivalence between $W_{-n}$ convergence and pointwise convergence follows as before, applying a continuity theorem for the Laplace transform~\cite[Ex.~5.5]{billingsley2013convergence} in place of L\'evy's continuity theorem. From the weak convergence of $\widetilde{\lambda}_j = (1+s^2)^{-n/2}\,d\lambda_j(s)$ to $\widetilde{\lambda} = (1+s^2)^{-n/2}\,d\lambda(s)$, it also follows that $(1+d_t^2)^{k-n/2}K_j\to (1+d_t^2)^{k-n/2}K$ pointwise on $(0,\infty)$ for any $k\geq 0$. Indeed, this corresponds to the convergence of $\int f\,d\widetilde{\lambda}_j$ to $\int f\,d\widetilde{\lambda}$, where $f(s) = (1+s^2)^{k}e^{-st}$ is a continuous, bounded function on $[-M,\infty)$ for each $t>0$. Uniform convergence can be proven as before.
\end{proof}

This result allows us to re-interpret half of \cref{thm:main_continuity} in terms of the time domain. Namely, suppose the triples $(K_j,c_{0,j},c_{1,j})$ are such that $K_j(t) = O(e^{Mt})$ for some $M>0$ and all $j$. Suppose that $c_{0,j}\to c_0$ in $\RR$ and that $c_{1,j}\to c_1$ in $(0,\infty)$, and that there is a kernel $K(t)=O(e^{Mt})$ such that $(1+d_t^2)^{-1/2}K_j\to (1+d_t^2)^{-1/2}K$ pointwise. Then \cref{prop:time_cont_CM} tells us that, if the interconverted kernels $J_j$ and $J$ satisfy $J_j,J = O(e^{Mt})$ for a possibly-increased value $M>0$, the triples $(J_j,\zeta_{0,j},\zeta_{1,j})$ converge to $(J,\zeta_0,\zeta_1)$ in the same way. Moreover, we can deduce that $d_t^kJ_j$ converges locally uniformly on $(0,\infty)$ to $d_t^k J$ for any $k\geq 0$.

Moving on now to the $W_\infty$ topology for compactly supported measures, we note that it admits a slightly more practical characterization in terms of the size of this compact support:
\begin{proposition}\label{prop:topology}
    Let $\mu_n,\mu\in\cM_c(\RR)$. Then $\mu_n\to\mu$ in $W_\infty$ if and only if $\mu_n\rightharpoonup\mu$ weakly and the sets $\supp\mu_n$ are uniformly bounded.
\end{proposition}
\begin{remark}
    Connecting back to \cref{rem:winf}, this result shows that the $W_\infty$ topology is strictly stronger than the limit of $W_{+n}$ as $n\to +\infty$.
\end{remark}
\begin{proof}
    In one direction, suppose that $\mu_n\rightharpoonup\mu$ weakly and $\supp\mu_n,\supp\mu\subset I$ for a fixed interval $I\subset\RR$. For any continuous $f\in\cC(\RR)$, define a \emph{bounded} continuous function $\widetilde{f}\in\cC(\RR)$ such that $\widetilde{f}|_I\equiv f|_I$; for instance, we can extend $f$ by its values on the endpoints of $I$. Then we know that
    \[\int f\,d\mu_n = \int \widetilde{f}\,d\mu_n\to \int \widetilde{f}\,d\mu = \int f\,d\mu,\]
    so that $\mu_n\to\mu$ in $W_\infty$.

    Conversely, suppose that $\mu_n\to\mu$ in $W_\infty$, but that the sets $\supp\mu_n$ are \emph{not} uniformly bounded. For each integer $N\geq 1$, choose $n_N\geq 1$ such that $\supp \mu_{n_N}\not\subset [-N,N]$, and let
    \[\eps'_N \doteq \int_{\RR\setminus [-N,N]}d\mu_{n_N} > 0.\]
    \dave{Inductively, we define $\eps_N = \min(\eps'_N, \eps_{m<N})$, so that $\eps_N$ is non-increasing with $N$.} Then, define the function $f$ as follows; set $f(\pm N) = N/\eps_N$ for any positive integer $N$, set $f(0)=0$, and let $f(s)$ linearly interpolate between its values at adjacent integers. Then we find
    \[\int f\,d\mu_{n_N} \geq \int_{\RR\setminus[-N,N]} f\,d\mu_{n_N} \geq \frac{N}{\eps_N}\int_{\RR\setminus[-N,N]} d\mu_{n_N} = N\eps'_N/\eps_N\geq N.\]
    This contradicts the $W_\infty$-convergence of $\mu_n$, and the proposition follows.
\end{proof}

A practical case of interest occurs when we know the size of the compact support \emph{a priori}---for instance, if we are attempting to approximate $\lambda$ with discrete measures over a fixed, compact domain $I\subset\RR$. In this case, we recover strong control over our integral kernels in the time domain. For this, we recall the definition of the Wasserstein-$p$ metric between probability measures:
\begin{definition}[Wasserstein metrics]\label{def:wasserstein_metric}
    Write $\cM_1(I)$ for the space of Borel probability measures on a metric space $I$. If $\mu, \nu \in \cM_1(I)$, a \emph{coupling} between $\mu$ and $\nu$ is a probability measure $\pi \in \cM_1(I\times I)$ such that the marginal distribution of $\pi$ along the first copy of $I$ is $\mu$ and that along the second copy of $I$ is $\nu$. With $p \geq 1$, the \emph{Wasserstein-$p$} metric~\cite{santambrogio2015optimal} between $\mu$ and $\nu$ is
    \[\mathcal{W}_p(\mu, \nu) = \inf_\pi\left(\int d(x,y)^p\,d\pi(x,y)\right)^{1/p},\]
    where the infimum is taken over all couplings $\pi$ of $\mu$ and $\nu$.
\end{definition}

Then the following lemma is clear:
\begin{lemma}\label{lem:time_cont_infty}
    Suppose $\lambda_n,\lambda\in\cM_+(I)$ are non-negative measures supported in the compact domain $I\subset\RR$. Then the following statements are equivalent:
    \begin{enumerate}
        \item $\lambda_n\to\lambda$ weakly, or equivalently, in the $W_{\infty}$ topology.
        \item $\lambda_n\to\lambda$ in the $W_{-n}$ topology for any $n\in\RR$.
        \item $\cF[\lambda_n]\to\cF[\lambda]$ pointwise.
        \item For any $k\geq 0$, $d_t^k\cF[\lambda_n]\to d_t^k\cF[\lambda]$ locally uniformly.
        \item $\cL_b[\lambda_n]\to\cL_b[\lambda]$ pointwise.
        \item For any $k\geq 0$, $d_t^k\cL_b[\lambda_n](t)\to d_t^k\cL_b[\lambda](t)$ locally uniformly.
        \item $\|\lambda_n\|\to\|\lambda\|$, and $\lambda_n/\|\lambda_n\|\to\lambda/\|\lambda\|$ in the Wasserstein-$p$ metric for any $p\geq 1$.
        \item $\|\lambda_n\|\to\|\lambda\|$, and, for any monotonic, bounded, continuous $f:I\to\RR$ and any $p\geq 1$, we have $f_*(\lambda_n/\|\lambda_n\|)\to f_*(\lambda/\|\lambda\|)$ in the Wasserstein-$p$ metric.
    \end{enumerate}
\end{lemma}

Finally, we turn to a quantitative result, relating a reweighted Wasserstein-$p$ metric on $\cM_+(\RR_+)$ to a weighted $L^p$ convergence of integral kernels. By \cref{lem:time_cont_infty}, the restriction of this result to $\cM_+(I)$ for any compact $I\subset\RR_+$ provides a metric on the $W_\infty$ topology. Suppose we have two CM integral kernels
\begin{align*}
    K_{\mu}(t) = \cL[\mu](t) = \int e^{-\alpha t}\,d\mu(\alpha), \qquad K_{\nu}(t) = \cL[\nu](t) = \int e^{-\alpha t}\,d\nu(\alpha),
\end{align*}
where $\mu, \nu \in \cM_1(\RR_+)$ are probability measures on $\RR_+=[0,\infty)$. We study the following \emph{$\varepsilon$-regularized} $L^p$ distance between these kernels:
\begin{align*}
    \|K_{\mu} - K_{\nu}\|_{L_{\varepsilon}^p} \doteq \Big(\int_0^\infty e^{-\varepsilon pt}\big|K_{\mu}(t) - K_{\nu}(t)\big|^p\,dt\Big)^\frac{1}{p},
\end{align*}
for any $\varepsilon > 0$. Define the function $f^\eps(\alpha) = \tfrac{1}{\alpha + \varepsilon}$ on $\RR_+$; consistent with the final statement of \cref{lem:time_cont_infty}, $f$ is monotonic, bounded, and continuous. For $\varepsilon = 0$, we denote $L^p \doteq  L_0^p$ and $f \doteq f^0$. The following theorem shows that we can control the $L_{\varepsilon}^p$ distance by the Wasserstein-1 metric between $f^\epsilon_*\mu$ and $f^\varepsilon_*\nu$.
\begin{proposition}[Wasserstein-1 bound on CM kernels]\label{prop:wasserstein1}
    Let $f^\varepsilon(\alpha) = \tfrac{1}{\alpha + \varepsilon}$, and fix $\eps>0$ and $p\geq 1$. Then we have
    \begin{align*}
        \|K_{\mu} - K_{\nu}\|_{L_{\varepsilon}^p} \leq c\mathcal{W}_1(f^\varepsilon_*\mu, f^\varepsilon_*\nu)^\frac{1}{p}
    \end{align*}
    for any $\mu, \nu \in \cM_1(\RR_+)$, and with $c=2$ in the general case. This result holds with $c=2(\tfrac{1}{2p})^\frac{1}{p}$ for $p$ odd, and in particular, is equal to one for $p = 1$. For positive measures $\mu, \nu \in \cM_+(\RR_+)$ with equal mass $m>0$, the result continues to hold with $c=2m$ or $c = 2(\tfrac{1}{2p})^\frac{1}{p}m$, respectively. If $\supp\mu$ and $\supp\nu$ are both bounded away from zero, the result holds with $\eps=0$.
\end{proposition}
\begin{proof}
    Without loss of generality, translate $\mu, \nu \in \cM_1(\RR_+)$ by $+\eps$ (so both are supported in $[\eps,\infty)$) and set $\varepsilon = 0$. For any coupling $\pi \in \cM_1(\RR_+\times\RR_+)$ of $\mu$ and $\nu$, we have
    \begin{align*}
        \|K_\mu - K_\nu\|_{L^p} &= \Big\|\int_0^\infty\int_0^\infty e^{-\alpha t}\,d\pi(\alpha, \beta) - \int_0^\infty\int_0^\infty e^{-\beta t}\,d\pi(\alpha, \beta)\Big\|_{L^p}\\
        &\leq \int_0^\infty\int_0^\infty \|e^{-\alpha t} - e^{-\beta t}\|_{L^p}\,d\pi(\alpha, \beta)\\
        &\leq \Big(\int_0^\infty\int_0^\infty \|e^{-\alpha t} - e^{-\beta t}\|_{L^p}^p\,d\pi(\alpha, \beta)\Big)^\frac{1}{p}
    \end{align*}
    where the second line follows from the triangle inequality and the third from Jensen's inequality. For $\alpha < \beta$, we can write
    \begin{align*}
        \|e^{-\alpha t} - e^{-\beta t}\|_{L^p}^p = \int_0^\infty\big|e^{-\alpha t} - e^{-\beta t}\big|^p\,dt = \sum_{k=0}^p \binom{p}{k} \frac{(-1)^{p-k}}{k\alpha + (p-k)\beta}.
    \end{align*}
    If $p$ is odd, we can pair the terms in this summation to bound
    \begin{align*}
        \sum_{k=0}^p \binom{p}{k} \frac{(-1)^{p-k}}{k\alpha + (p-k)\beta} &\leq \sum_{k=0}^\frac{p-1}{2} \binom{p}{k} \Big|\frac{1}{k\alpha + (p-k)\beta} - \frac{1}{(p-k)\alpha + k\beta}\Big|\\
        &\leq \frac{1}{p}\sum_{k=0}^\frac{p-1}{2} \binom{p}{k}\Big|\frac{1}{\alpha} - \frac{1}{\beta}\Big| = \frac{2^{p-1}}{p}\Big|\frac{1}{\alpha} - \frac{1}{\beta}\Big|,
    \end{align*}
    and similarly for $\beta<\alpha$. Substituting this into the bound for $\|K_\mu - K_\nu\|_{L^p}$ above yields
    \begin{align*}
        \|K_\mu - K_\nu\|_{L^p} &\leq 2\big(\tfrac{1}{2p}\big)^\frac{1}{p} \Big(\int_0^\infty\int_0^\infty\Big|\frac{1}{\alpha} - \frac{1}{\beta}\Big|\,d\pi(\alpha, \beta)\Big)^\frac{1}{p}.
    \end{align*}
    Since this bound holds for all couplings $\pi$ of $\mu, \nu$, then taking the infimum over couplings proves that
    \begin{align*}
        \|K_\mu - K_\nu\|_{L^p} \leq 2\big(\tfrac{1}{2p}\big)^\frac{1}{p} \mathcal{W}_1(f_*\mu, f_*\nu)^\frac{1}{p},
    \end{align*}
    where $f(\alpha) = \tfrac{1}{\alpha}$. 
    
    Now, consider a general $p\geq 1$, and let $p_0 \leq p \leq p_1$ be odd integers. Using the log-convexity of $L^p$ norms, we find
    \begin{align*}
        \|K_\mu - K_\nu\|_{L^p} \leq \|K_\mu - K_\nu\|_{L^{p_0}}^{1-\theta}\|K_\mu - K_\nu\|_{L^{p_1}}^{\theta}
    \end{align*}
    where $\theta$ satisfies $\tfrac{1}{p} = \tfrac{1-\theta}{p_0} + \tfrac{\theta}{p_1}$. Since $p_0$ and $p_1$ are odd, \dave{the result follows}.
\end{proof}
This bound is tight for $p = 1$; for $\mu = \delta_\alpha$ and $\nu = \delta_\beta$, we have
\begin{align*}
    \|K_\mu - K_\nu\|_{L^1} = \|e^{-\alpha t} - e^{-\beta t}\|_{L^1} = \Big|\frac{1}{\alpha} - \frac{1}{\beta}\Big| = \mathcal{W}_1(f_*\mu, f_*\nu).
\end{align*}
On the other hand, this bound can likely be improved for $p > 1$. In this direction, it is possible to show that
\begin{equation}\label{eq:wasserstein_L2}
    \|K_\mu - K_\nu\|_{L_\varepsilon^2} \leq C\mathcal{W}_2(g^\varepsilon_{*}\mu, g^\varepsilon_{*}\nu),
\end{equation}
where $g^\varepsilon(\alpha) = \tfrac{1}{2}(\alpha + \varepsilon)^{-\frac{1}{2}}$ and $C>0$ is independent of $\mu$ and $\nu$.

Given the Wasserstein bounds derived above, a natural problem to study is the approximation of CM kernels by discrete sums of exponentials, also known as Prony series. We prove the following result:
\begin{corollary}[Approximation of CM kernel by Prony series]\label{cor:quantization}
    Let $\cD_n\subset\cM_+(\RR_+)$ be the set of discrete measures on $\RR_+$ with $n$ atoms. Given $\mu\in\cM_+(\RR_+)$ and an associated CM kernel $K_\mu=\cL[\mu]$, the following bound holds:
    \begin{align*}
        \inf_{\mu_n \in \mathcal{D}_n}\|K_{\mu} - K_{\mu_n}\|_{L_\varepsilon^1} \leq \frac{\|\mu\|}{2(\inf\supp\mu + \varepsilon)}\,\frac{1}{n}.
    \end{align*}
    If $\supp\mu$ is bounded away from zero, the result holds with $\eps=0$.
\end{corollary}
\begin{remark}
    Approximating $\mu \in \cM(\RR_+)$ in the Wasserstein metric by discrete measures is the classical problem of optimal quantization~\cite{bourne2018semi}, and can be solved in practice through Lloyd's algorithm (or k-means clustering). 
    
    Secondly, although the result is stated in terms of the $L^1$ norm, a similar $O(1/n)$ bound can be proven for other values of $p\geq 1$ using results of the form \cref{eq:wasserstein_L2}. 
    
\end{remark}
\begin{proof}
    Without loss of generality, suppose $\mu\in\cM_1(\RR_+)$ is a probability measure. \Cref{prop:wasserstein1} implies
    \begin{align*}
        \inf_{\mu_n \in \mathcal{D}_n}\|K_{\mu} - K_{\mu_n}\|_{L_\varepsilon^1} \leq \inf_{\rho \in \mathcal{D}_n}\mathcal{W}_1(f^\varepsilon_*\mu, \rho),
    \end{align*}
    where \dave{$f^\eps(\alpha) = (\alpha+\eps)^{-1}$. Define $b=(\inf\supp\mu + \varepsilon)^{-1}=\sup\supp f_*^\eps\mu$, fix a discrete measure $\rho \in \mathcal{D}_n$ of the form
    \begin{align*}
        \rho = \sum_{k=1}^n \rho_k\delta(\alpha - \alpha_k)\,d\alpha, \quad \rho_k = \int_{b\frac{k-1}{n}}^{b\frac{k}{n}} d(f^\varepsilon_*\mu), \quad \alpha_k = \frac{b}{2n}(2k-1),
    \end{align*}
    and define} a coupling between $f_*^\eps\mu$ and $\rho$ by
    \begin{align*}
        \pi(\alpha, \beta) = \chi_{[b\frac{k-1}{n}, b\frac{k}{n}]}(\alpha)\delta(\beta - \alpha_k)\,d(f^\varepsilon_*\mu)(\alpha)d\beta.
    \end{align*}
    The corollary then follows from the following bound:
    \begin{align*}
        \cW_1(f^\varepsilon_*\mu, \rho) \leq \int_0^\infty\int_0^\infty |\alpha - \beta|\,d\pi(\alpha, \beta) \leq \frac{b}{2n} \int_0^\infty df^\varepsilon_*\mu(\alpha) = \frac{b}{2n}.
    \end{align*}
\end{proof}

These results give quantitative bounds on the modulus of continuity of CM integral and integro-differential operators:
\begin{corollary}[Approximation of CM equations]\label{cor:conv_diff}
    Assume that $x: \RR_+ \to \RR$ is a locally-bounded trajectory and that $\mu, \nu \in \cM_1(\RR_+)$. For any $\eps>0$, we have
    \begin{align*}
        |(K_\mu * x)(t) - (K_\nu * x)(t)| \leq \big(\sup_{\tau<t}|x(\tau)|\big)e^{\varepsilon t}\, \mathcal{W}_1(f^\varepsilon_*\mu, f^\varepsilon_*\nu)
    \end{align*}
    for all $t\geq 0$, where $f^\varepsilon(\alpha) = (\alpha + \varepsilon)^{-1}$. Furthermore, if $\cD_n\subset\cM_+(\RR_+)$ is the set of discrete measures with $n$ atoms, we have
    \begin{align*}
        \inf_{\mu_n \in \mathcal{D}_n}|(K_\mu * x)(t) - (K_{\mu_n} * x)(t)| \leq \frac{\sup_{\tau<t}|x(\tau)|}{2(\inf\supp \mu + \varepsilon)}\,\frac{e^{\varepsilon t} }{n}
    \end{align*}
    for all $t \geq 0$. We can set $\varepsilon = 0$ in the above bounds if $\mu$ and $\nu$ are bounded away from zero. Furthermore, the bounds above hold even when $\mu, \nu$ have equal mass $m \neq 1$, in which case both bounds must be rescaled by this constant.

    Hence, considering the CM Volterra equations
    \begin{align*}
        y_\mu(t) &= c_1\dot{x}(t) - c_0x(t) - \int_0^t K_\mu(t-\tau)x(\tau)\,d\tau,\\
        y_\nu(t) &= c_1\dot{x}(t) - c_0x(t) - \int_0^t K_\nu(t-\tau)x(\tau)\,d\tau,
    \end{align*}
    with the same input $x$, we can bound
    \begin{align*}
        |y_\mu(t) - y_\nu(t)| \leq |(K_\mu * x)(t) - (K_\nu * x)(t)|
    \end{align*}
    for all $t \geq 0$, and the bounds above apply.
\end{corollary}
\begin{remark}
    This result can \dave{also be translated} to the language of linear time-invariant systems. Choose discrete approximations $\mu_n =\tsum_{i=1}^n \beta_i\delta(\alpha - \alpha_i)\in\cD_n$ to $\mu\in\cM_+(\RR_+)$. Fixing an input trajectory $x$, we can rewrite\footnote{See \cref{sec:LTI} for more details on \dave{this construction}.} the corresponding CM Volterra equation for the output trajectory $y_n$ as
    \begin{align*}
        y_n(t) &= c_1\dot{x}(t) - c_0x(t) - \sum_{i=1}^n\xi_i(t)\\
        \dot{\xi}_i(t) &= -\alpha_i\xi_i(t) + \beta_ix(t), \quad \xi_i(0) = 0.
    \end{align*}
    \dave{The dynamics of $y_n$ converge to those of $y_\mu$} at a rate $O(1/n)$ if $\mu_n$ is chosen as a `quantizer'~\cite{bourne2018semi} of $\mu$, as constructed in \cref{cor:quantization}. Approximation of Volterra equations by Markovian differential equations arises in the simulation of material deformations~\cite{bhattacharya2023learning}, and the appearance of strong memory effects in high-dimensional dynamical systems is central in the study of multiscale physical processes~\cite{givon2004extracting}.
\end{remark}

\Cref{cor:conv_diff} shows that the output trajectories $y$ of Volterra equations can be approximated by Volterra equations with finite spectra. By applying the same logic after interconverting, one can show that the solutions $x$ of Volterra equations can be approximated similarly. Consider the equation
\[y(t) = c_1\dot{x}(t) - c_0x(t) - \int_0^t K(t-\tau)x(\tau)\,d\tau,\]
where $K = \cL[\lambda]$ for some $\lambda\in\cM_+(\RR_+)$. Write $(\mu,\zeta_0,\zeta_1)=\cB_\RR[\lambda,c_0,c_1]$. Fix discrete approximations $\mu_n\in\cD_n$ to $\mu$, as in \cref{cor:quantization}, and write $(\lambda_n,c_0^n,c_1^n) \doteq \cB_\RR[\mu_n,\zeta_0,\zeta_1]$ for the interconverted triples and $K_n=\cL[\lambda_n]$ for the corresponding integral kernels; note that there is no guarantee that $(c_0^n, c_1^n) = (c_0, c_1)$. From \cref{cor:discrete_formula}, the measures $\lambda_n$ lie in either $\cD_{n-1}$, $\cD_n$, or $\cD_{n+1}$, depending on the values of $c_0$ and $c_1$. In any case, \cref{cor:conv_diff} implies that, for any sufficiently well-behaved forcing $y$, the solutions $x_{n}$ to the approximate gCM equations 
\[y(t) = c_1\dot{x}_n(t) - c_0x_n(t) - \int_0^t K_n(t-\tau)x_n(\tau)\,d\tau\]
converge locally uniformly to their limit $x$ at a rate $O(1/n)$.

\Cref{prop:wasserstein1} and Corollaries~\ref{cor:quantization} and~\ref{cor:conv_diff} can certainly be extended to the case of gCM Volterra equations, which correspond to finite measures $\lambda\in\cM_+([-R,\infty))$ for $R>0$; so long as $\eps> R$, the above results hold as stated. These results can also likely be extended to the other classes of Volterra equations under consideration, \dave{pulling back to the circle when necessary}. 

%% file: 7_circle.tex
\section{Hardy Spaces and Integral Transforms on the Circle}\label{sec:disc}
In the present section, we work to develop our spectral theory on the circle; as we saw in \cref{sec:main_circ}, the set $\cM_+(S^1)$ of positive Borel measures on the unit circle offers a natural language with which to study difference equations of the form \ref{eq:integrodiff_dPD}. Although we are primarily interested in understanding the involution $\cB$ of \cref{thm:main_circle}, we proceed by studying how the Cauchy transform behaves under a wide class of nonlinear maps. Interconversion will then follow as a special case.

Let $\HH_+ = -i\HH$ be the open right half-plane, and $\ovl{\HH}_+$ be its closure. Below, we say that a map $S:\ovl{\HH}_+\to \ovl{\HH}_+\cup\{\infty\}$ is \emph{admissible} if $S|_{\HH_+}$ is holomorphic, if $\op{clos}S^{-1}(\infty)\subset\partial \HH_+$ is countable (if non-empty), and if $S$ is continuous on the complement $\ovl{\HH}_+\setminus \op{clos}S^{-1}(\infty)$. Examples of these maps include affine transformations and circular inversions:
\[\qquad z\mapsto az + z_0,\qquad z\mapsto \frac{a}{z-i\zeta},\]
where $a>0$, $z_0\in \ovl{\HH}_+$, and $\zeta\in\RR$. Interconversion corresponds to $S:z\mapsto 1/z$.

\begin{remark}\label{rem:map_algebra}
If the singular component of $\lambda\in\cM_+(S^1)$ has countable support, then the composition $z\mapsto(Q_\sigma[\lambda]\circ\phi^{-1})(iz)$ is an admissible map. One can \dave{thus consider `composing' multiple non-negative measures on the circle, though we do not discuss the topic further at present.}
\end{remark}

Consider the nonlinearly-transformed data $S\circ Q_\sigma[\lambda]$, for $\lambda\in\cM_+(S^1)$ and $\sigma\in\RR$. By construction, this data forms a holomorphic function in $\DD$ with positive real part, so \cref{prop:classic}.\ref{prop:classic1} guarantees that
\[S\circ Q_\sigma[\lambda] = Q_{\sigma'}[\mu]\]
for some $\sigma'\in\RR$ and $\mu\in\cM_+(S^1)$. Our first goal is to understand what form $\sigma'$ and $\lambda'$ take, and in particular, to show how $S$ can be seen to ``commute'' with the Cauchy and Hilbert transforms. As a first step, we show how admissible maps preserve integrability, in an appropriate sense:

\begin{lemma}\label{lem:firstone}
    Suppose $\lambda\in\cM_+(S^1)$, and write $\lambda_c\in L^1(S^1)$ for the density of its continuous part, furnished by the Lebesgue decomposition~\cite{rudin1974real}. Fix an admissible map $S$, and let $S^{\Re}=\Re S$. Then $S^{\Re}\circ(\lambda_c + iH_\sigma[\lambda])\in L^1(S^1)$, and moreover,
    \begin{equation}\label{eq:L1_bound}
    \|S^{\Re}\circ(\lambda_c+iH_\sigma[\lambda])\|_{S^1}\leq S^{\Re}(\|\lambda\|_{S^1}+i\sigma)
\end{equation}
for any $\sigma\in\RR$.
\end{lemma}

\begin{proof}
    \dave{It follows from the mean value property that $Q_\sigma[\lambda](0)=\|\lambda\|_{S^1}+i\sigma$, and from the maximum principle that $\Re Q_\sigma[\lambda]>0$ everywhere in $\DD$. Since $S(\HH_+)\subset \HH_+$, it follows that $S^{\Re}\circ Q_\sigma[\lambda]>0$ everywhere in $\DD$}.

    \dave{Let $\Sigma=\op{clos}S^{-1}(\infty)\subset\partial\HH_+$ be the \dave{(countable)} set of singularities of $S$}. Fix $y\in\Sigma$, and consider the function $q_y \doteq \exp(y-Q_\sigma[\lambda])-1$. This is a bounded holomorphic function on $\DD$, so it follows from a theorem of F.~and M.~Riesz~\cite{lars1988collected,e8589594-ae28-3f64-a921-def172de1f4c} that its zero set $\{q_y^{-1}(0)\}\supset\{Q_\sigma[\lambda]=y\}$ forms a set of measure zero in $S^1$; taking the union over $y\in\Sigma$, we see that $S\circ Q_\sigma[\lambda]$ is finite almost everywhere on $S^1$.
    
    Thus, since $S$ is continuous away from its singularities and $Q_\sigma[\lambda]\to \lambda_c+iH_\sigma[\lambda]$ almost everywhere in $S^1$ (along non-tangential paths), we know that 
    \[S^{\Re}\circ Q_\sigma[\lambda]\to S^{\Re}\circ(\lambda_c+iH_\sigma[\lambda])\]
    (along non-tangential paths) almost everywhere in $S^1$. \dave{Fatou's lemma thus implies}
    \begin{align*}\|S^{\Re}\circ(\lambda_c+iH_\sigma[\lambda])\|_{S^1} &\leq \lim_{r\to 1^-}\frac{1}{2\pi}\int_0^{2\pi} (S^{\Re}\circ Q_\sigma[\lambda])(re^{i\theta})\,d\theta \\
    &= (S^{\Re}\circ Q_\sigma[\lambda])(0) \\
    &= S^{\Re}(\|\lambda\|_{S^1}+i\sigma).
    \end{align*}
\end{proof}

We can derive a stronger result by leveraging \cref{prop:classic}.\ref{prop:classic4}; in short, if a positive harmonic function in $\DD$ has a known non-tangential limit almost everywhere in $S^1$, the remaining (measure zero) set \emph{must} carve out a unique, singular measure:

\begin{theorem}\label{thm:representation}
    Let $\lambda$ and $S$ be as in \cref{lem:firstone}, and fix $\sigma\in\RR$. There is a unique singular measure $\nu\in\cM_+(S^1)$ such that
    \begin{equation}\label{eq:thmrep0}
        S\circ Q_\sigma[\lambda] = Q[\nu] + Q\big[S^{\Re}\circ(\lambda_c + iH_\sigma[\lambda])\big]  + iS^{\Im}(\|\lambda\|_{S^1}+i\sigma),
    \end{equation}
    and equivalently,
    \begin{equation}\label{eq:thmrep1}
    S^{\Im}\circ(\lambda_c+iH_\sigma[\lambda]) = H[\nu] + H[S^{\Re}\circ(\lambda_c+iH_\sigma[\lambda])] + S^{\Im}(\|\lambda\|_{S^1}+i\sigma).
    \end{equation}
\end{theorem}
\begin{proof}
    Recall from the proof of \cref{lem:firstone} that 
    \[S^{\Re}\circ Q_\sigma[\lambda]\to S^{\Re}\circ (\lambda_c+iH_\sigma[\lambda])\]
    almost everywhere (along non-tangential directions) in $S^1$. Suppose that $\mu\in\mathcal{M}_+(S^1)$ is the unique finite (positive) Borel measure such that 
    \begin{equation}\label{eq:repthm_foo1}
        S^{\Re}\circ Q_\sigma[\lambda] = P[\mu] = \Re Q[\mu],
    \end{equation}
    furnished by \cref{prop:classic}, and let $\mu_c$ be the density of its continuous component with respect to the normalized Lebesgue measure $(2\pi)^{-1}\,d\theta$. From \cref{prop:classic}.\ref{prop:classic4}, then, \dave{we know that $\Re Q[\mu]\to \mu_c$ pointwise along non-tangential directions}, almost everywhere in $S^1$; we can thus identify
    \[S^{\Re}\circ (\lambda_c+iH_\sigma[\lambda]) = \mu_c,\]
    and define $\nu$ to be the (leftover) singular component of $\mu$.

    Now, recall from \cref{cor:uniqueness} that the conjugate harmonic function of $P[\mu]$ in $\DD$ is uniquely defined up to addition of imaginary constants; in particular, \cref{eq:repthm_foo1} shows that $S^{\Im}\circ Q_\sigma[\lambda]$ and $\Im Q[\mu]$ differ by a real constant. Identifying this constant by evaluating each at the origin, we deduce \cref{eq:thmrep0}, and taking the non-tangential limit at $r=1$, we deduce \cref{eq:thmrep1}.
\end{proof}

Our next goal is to understand the singular measure $\nu$ more concretely; if one could calculate $\nu$ from the base measure $\lambda$, then \cref{thm:representation} would yield an explicit formula for the Cauchy and Hilbert transforms of the nonlinearly-transformed data $S^{\Re}\circ(\lambda_c + iH_\sigma[\lambda])\in L^1(S^1)$. In this direction, we now investigate the \emph{support} of $\nu$; if we know the support to be countable, we can deduce that $\nu$ is discrete.

If $S$ is an admissible map, we further say that $S$ is \emph{highly admissible} if, for each $\eps>0$, the real part $\Re S(z)$ is uniformly bounded over the set
\[H_\eps\doteq \{z\in \HH_+\;|\;\eps < \Re z < 1/\eps\},\]
that is, $\sup_{z\in H_\eps} \Re S(z) < C_\eps$ for a fixed $C_\eps>0$. For instance, the maps $z\mapsto z$ and $z\mapsto 1/z$ are both highly admissible, but
\[S_0:z\mapsto \sum_{n\in\ZZ}\frac{n}{z-in^3}\]
is not; indeed, at the point $z=\eps+in_0^3$, we have
\[\Re S_0(z) = \sum_{n\in\ZZ}\frac{n\eps}{\eps^2+(n^3-n_0^3)^2}\geq n_0/\eps.\]
Choosing sequentially larger $n_0$ shows that $\Re S_0$ is not uniformly bounded on $H_\eps$.

We generalize \dave{the statement of} \cref{def:zeroset_circ} as follows:
\begin{definition}\label{def:problematic_circ}
    Suppose $\lambda\in\cM_+(S^1)$. Define the critical set of $\lambda$ as follows:
    \[N_{\infty}(\lambda) = \bigcap_{\eps>0}\op{clos}\Big\{e^{i\theta}\in S^1\;\Big|\;\liminf_{\delta\to 0}\lambda(\exp i[\theta-\delta,\theta+\delta])/2\delta > 1/\eps\Big\},\]
    \dave{and the \emph{problematic} set of $\lambda$ to be 
    \[N(\lambda)\doteq N_0(\lambda)\cup N_\infty(\lambda)\subset S^1,\]
    with $N_0(\lambda)$ the zero set of \cref{def:zeroset_circ} and $\supp\lambda\subset S^1$ the closed, essential support of $\lambda$}.
\end{definition}
The sets $N_\infty(\lambda)$ and $N(\lambda)$ allow us to treat general highly admissible maps, rather than just the particular case $S:z\mapsto 1/z$ corresponding to interconversion; we will see shortly that only $N_0(\lambda)$ plays a role for the latter.
\begin{lemma}\label{lem:containment}
    In the setting of \cref{lem:firstone}, suppose now that $S$ is highly admissible. Then the singular measure $\nu$ furnished by \cref{thm:representation} satisfies
    \[\supp \nu\subset N_0(\lambda)\cup N_\infty(\lambda),\]
    with $N_0(\lambda)$ and $N_\infty(\lambda)$ as defined in \cref{def:problematic_circ}.
\end{lemma}
\begin{proof}
    Write
    \begin{equation}\label{eq:define_eps_problematic}
        \begin{aligned}N_{\eps}(\lambda) &= \op{clos}\Big\{e^{i\theta}\in S^1\;\Big|\;\liminf_{\delta\to 0}\lambda(\exp i[\theta-\delta,\theta+\delta])/2\delta > 1/\eps\Big\}\\
    &\qquad\cup\op{clos}\Big\{e^{i\theta}\in S^1\;\Big|\;\limsup_{\delta\to 0}\lambda(\exp i[\theta-\delta,\theta+\delta])/2\delta < \eps\Big\},
        \end{aligned}
    \end{equation}
    so that $N_0(\lambda)\cup N_\infty(\lambda) = \bigcap_{\eps>0} N_\eps(\lambda)$. Suppose that $z\notin N_0(\lambda)\cup N_\infty(\lambda)$, so that, in particular, there is an $\eps>0$ such that $z\notin N_\eps(\lambda)$; since $N_\eps(\lambda)$ is closed, we can fix a closed interval $I\ni z$ in $S^1$ such that 
    \[I\cap N_\eps(\lambda) = \emptyset.\]
    In particular, the restriction $\lambda|_{I}$ is absolutely continuous, with density $\eps<\lambda_c<1/\eps$. As in the proof of \cref{thm:representation}, let $\mu\in\mathcal{M}_+(S^1)$ be the unique measure such that
    \[P[\mu] = S^{\Re}\circ Q_\sigma[\lambda]\]
    in $\DD$. Now, we decompose $\lambda$ as
    \[\lambda = \lambda_1 + \lambda_2,\]
    where $\supp\lambda_1\subset I$ and $\supp\lambda_2\subset S^1\setminus I$. 

    Since $\lambda_1$ is absolutely continuous with density $\lambda_c\chi_I$ (where $\chi_I$ is the characteristic function of $I$), our choice of $I$ guarantees that
    \[\eps\chi_I<\lambda_c\chi_I<(1/\eps)\chi_I\]
    almost everywhere. The maximum principle thus shows that
    \[\eps P[\chi_I] \leq P[\lambda_1] \leq (1/\eps)P[\chi_I]\]
    everywhere in $\ovl{\DD}$. Fix a small interval $I'\subset I$ containing $z$ and a small $\delta>0$, such that $P[\chi_I]$ is uniformly continuous in the neighborhood
    \[B_\delta \doteq \{z\in \ovl{\DD}\;|\;\|z-I'\|<\delta\}.\]
    For sufficiently small $\delta>0$, then, we can guarantee that $2\eps/3<P[\lambda_1]<1/\eps$ in $B_\delta$. Next, inspecting the Poisson kernel, we can see that---potentially shrinking $I'$ and $\delta$---the harmonic function $P[\lambda_2]$ is uniformly bounded in the neighborhood
    \[B_\delta \doteq \{z\in \ovl{\DD}\;|\;\|z-I'\|<\delta\}\]
    by $C\delta$, where $C>0$ is a constant independent of $\delta$. Fixing $\delta$ such that $C\delta <\eps/3$ and combining with our control on $P[\lambda_1]$, we find that
    \[\eps/3 < P[\lambda](z) = \Re Q_\sigma[\lambda](z) < 3/\eps\]
    for $z\in B_\delta$. Since $S$ is highly admissible, then, we find that
    \[(S^{\Re}\circ Q_\sigma[\lambda])(z) < C_{\eps/3}\]
    for $z\in B_\delta$; since it is uniformly bounded, there cannot be a singular component of $\mu$ in $I'$. But $z\notin N(\lambda)$ was general, so the claim follows.
\end{proof}
The above lemma can be refined slightly, in fact, if one knows more information about the singularities of $S$. The following lemma follows from a similar argument as above:
\begin{lemma}\label{lem:containment_refined}
    If $\Re S(z)$ is uniformly bounded over the set $\Re z> \eps$ for each $\eps>0$, we say it is \emph{lower highly admissible} (LHA), and a similar argument shows that
    \[\supp\nu\subset N_0(\lambda).\]
    Likewise, if $\Re S(z)$ is uniformly bounded over the set $\Re z < 1/\eps$ for each $\eps>0$, we say it is \emph{upper highly admissible} (UHA), and we find
    \[\supp\nu\subset N_\infty(\lambda).\]
\end{lemma}

Under appropriate conditions on $\lambda$, these lemmas allows us to deduce further structure on the measure $\nu$:

\begin{corollary}\label{cor:countable}
    Suppose $S$ is highly admissible. If $N(\lambda)\cap\supp\lambda$ is countable, then the measure $\nu$ furnished by \cref{thm:representation} is discrete, and its closed support is countable. Alternatively, if $N(\lambda)$ is finite, $\supp\lambda$ has finitely many components, and $S^{-1}(\infty)\subset S^1$ is finite, then $\nu$ is discrete, and its support is finite. In either case, define the set
    \begin{equation}\label{eq:poles}
        Z(\lambda) = N(\lambda)\cup\{z\notin\supp\lambda\;|\;iH_\sigma[\lambda](z)\in\Sigma\}.
    \end{equation}
    For any choice of $\sigma\in\RR$, we have
    \[S\circ Q_\sigma[\lambda](z) = Q[S^{\Re}\circ(\lambda_c + iH_\sigma[\lambda])](z) + \sum_{\alpha_j\in Z(\lambda)} \beta_j Q[\delta_{\alpha_j}](z) + i\zeta\]
    for unique values $\zeta\in\RR$ and $\beta_j>0$, where $\delta_{\alpha_j}$ is a unit Dirac measure at $\alpha_j\in S^1$. Equivalently,
    \[S^{\Im}\circ Q_\sigma[\lambda](z) = H[S^{\Re}\circ(\lambda_c + iH_\sigma[\lambda])](z) + \sum_{\alpha_j\in Z(\lambda)} \beta_j H[\delta_{\alpha_j}](z) + \zeta.\]
\end{corollary}
\begin{proof}
    From \cref{lem:containment}, we deduce that $\supp\nu\cap \supp\lambda$ is countable [resp., finite] and contained in $N(\lambda)\cap\supp\lambda$. That $\supp\nu\setminus\supp\lambda$ is countable [resp.,~finite] follows from \cref{lem:smooth}; indeed, since $H_\sigma[\lambda]$ is smooth and strictly decreasing outside of $\supp\lambda$, it can only intersect the singular region $S^{-1}(\infty)$ countably [resp., finitely] many times.
\end{proof}

Once again, the LHA condition of \cref{lem:containment_refined} allows for a refinement of this statement, with much the same argument:
\begin{corollary}\label{cor:countable_refined}
    If $S$ is LHA and $N_0(\lambda)\cap\supp\lambda$ is countable, then \cref{cor:countable} holds with $Z(\lambda)$ replaced by
    \[Z_0(\lambda) = \left(N_0(\lambda)\cap \supp\lambda\right)\cup\{z\notin\supp\lambda\;|\;iH_\sigma[\lambda](z)\in\Sigma\}.\]
\end{corollary}

We now study the support of $S\circ Q_\sigma[\lambda]$. Although the following two results are not used in the proof of \cref{thm:main_circle}, they will be necessary to understand the pullback of the theorem to the real line in \cref{sec:main_line}. For any function $g$ on $S^1$ defined only up to sets of measure zero, we write 
\[\supp g\doteq S^1\setminus\bigcup\{I\subset S^1\;\text{open}\;|\;g(z)=0\;\text{for almost all}\;z\in I\}\]
for its closed, essential support.

\begin{lemma}\label{lem:support}
    Suppose $\lambda\in\cM_+(S^1)$, and write $\lambda_c\in L^1(S^1)$ for the density of its continuous part with respect to the normalized Lebesgue measure $(2\pi)^{-1}\,d\theta$. Fix an admissible map $S$, and let $S^{\Re} = \Re(S)$; note that $S$ need not be \emph{highly} admissible. Then we find
    \[\supp \left[S^{\Re}\circ(\lambda_c+iH_\sigma[\lambda])\right] \supset\supp \lambda_c\]
    for any $\sigma\in\RR$.
\end{lemma}
\begin{proof}
    Suppose $z\in S^1$ satisfies $\lambda_c(z)>0$. Since $S^{\Re}(\HH_+)\subset \HH_+$, we know that $S^{\Re}(\lambda_c(z) + iH_\sigma[\lambda](z))>0$ wherever $iH_\sigma[\lambda](z)$ is finite; of course, this holds for almost all $z\in S^1$. Thus, if $S^{\Re}\circ(\lambda_c+iH_\sigma[\lambda]) \equiv 0$ almost everywhere on an open set $I\subset S^1$, the same must be true of $\lambda_c$; the claim follows.
\end{proof}

The converse of \cref{lem:support} requires a stronger hypothesis on $S$, i.e., that it restricts to a map $S:\partial \HH_+\to\partial \HH_+\cup\{\infty\}$. This hypothesis is independent of the highly admissible hypothesis used in \cref{cor:countable}. Examples of this form include
\[z\mapsto az,\qquad z\mapsto\frac{a}{z-i\zeta},\]
where $a>0$ and $\zeta\in\partial \HH_+$. Connecting to \cref{rem:map_algebra}, the map $z\mapsto (Q_\sigma[\lambda]\circ\phi^{-1})(iz)$ only satisfies this hypothesis if $\lambda$ is a discrete measure.
\begin{proposition}[Support of transformed data]\label{prop:support}
    In the setting of \cref{lem:support}, suppose that $S$ restricts to a function $S:\partial \HH_+\to\partial \HH_+\cup\{\infty\}$. Then
    \[\supp \left[S^{\Re}\circ(\lambda_c+iH_\sigma[\lambda])\right] =\supp \lambda_c.\]
\end{proposition}
\begin{proof}
    One direction of the proof is furnished by \cref{lem:support}. Conversely, suppose that $\lambda_c\equiv 0$ almost everywhere on an open $I\subset S^1$. Recall from the proof of \cref{lem:firstone} that $S^{\Re}\circ(\lambda_c+iH_\sigma[\lambda])$ is finite almost everywhere in $S^1$; fix a $z\in I$ for which this is true (and for which $\lambda_c(z)=0$), so that our hypothesis on $S$ ensures 
    \[S(\lambda_c(z) + iH_\sigma[f](z)) = S(iH_\sigma[f](z)) \in \partial \HH_+,\]
    and thus $S^{\Re}(\lambda_c(z) + iH_\sigma[f](z)) = 0$. Since $z$ was generic, the claim follows.
\end{proof}

Finally, we prove a generalized version of \cref{thm:main_circle}. Much of the result follows from the theory developed so far; for instance, \cref{prop:classic}.\ref{prop:classic1} and \cref{cor:uniqueness} together imply that the map $\cB$ is a well-defined involution of $\cM_+(S^1)\times\RR$, and \cref{thm:representation} gives an explicit representation of $\zeta_0$ and of the continuous component of $\mu$. What remains to be shown is the \emph{topological} claim of the theorem---i.e., that $\cB$ is weakly continuous---which we show here in generalized form.

Given an admissible $S:\ovl{\HH}_+\to\ovl{\HH}_+\cup\{\infty\}$, define the map $\mathcal{B}_{S}: \cM_+(S^1)\times\RR \to \cM_+(S^1)\times\RR$ by
\begin{equation}\label{eq:QBL}
    \cB_S[\lambda,\sigma] = (\mu,\xi),\qquad S\circ Q_\sigma[\lambda] = Q_\xi[\mu].
\end{equation}
By \cref{thm:representation}, we can express the map more explicitly as
\begin{equation*}
    \mu = S^{\Re}\circ(\lambda_c + iH_\sigma[\lambda]) + \nu_{S,\sigma}[\lambda],\qquad \xi = S^{\Im}(\|\lambda\|_{S^1}+i\sigma),
\end{equation*}
where $\nu_{S,\sigma}[\lambda]$ is the singular measure furnished by the theorem. \Cref{thm:main_circle} follows as a special case of the following proposition:
\begin{proposition}[Weak continuity of admissible maps]\label{prop:weakcont}
    If $S$ is admissible, then $\cB_{S}:\cM_+(S^1)\times\RR\to\cM_+(S^1)\times\RR$ is continuous with respect to the product of the weak topology on $\cM_+(S^1)$ and the standard topology on $\RR$.
\end{proposition}

\begin{proof}
    Fix a nonzero $\lambda\in\cM_+(S^1)$ and $\sigma\in\RR$, and suppose $\lambda_n\in\cM_+(S^1)$ converges weakly to $\lambda$ and $\sigma_n$ converges to $\sigma$. For any $r<1$, define the following complex functions on $S^1$:
    \[f_{r,n}(e^{i\theta}) = Q_{\sigma_n}[\lambda_n](re^{i\theta}) = Q[\lambda_n](re^{i\theta}) + i\sigma_n.\]
    Notably, $f_{r,n}$ is smooth, and 
    \[\lim_{n\to\infty} f_{r,n}(e^{i\theta}) = f_r(e^{i\theta})\doteq Q_\sigma[\lambda](re^{i\theta})\]
    pointwise in $S^1$; this follows from the weak convergence of $\lambda_n$, as the Cauchy kernel is smooth and uniformly bounded along each fixed $r$. Suppose $\|\lambda\|_{S^1} = M>0$, and fix $N\geq 1$ such that $\|\lambda_{n}\|_{S^1}\leq 2M$ for all $n\geq N$. We then know that $\partial_\theta f_{r,n}$ is uniformly bounded in $n$, since
    \[|\partial_\theta f_{r,n}(e^{i\theta})| = \frac{1}{2\pi}\left|\int_0^{2\pi}\frac{1+ire^{i(\theta-\theta')}}{1-ire^{i(\theta-\theta')}}\,d\lambda_n(\theta')\right| \leq \frac{2M}{2\pi}\,\frac{1+r}{1-r},\]
    \dave{and so a standard argument shows that $\lim_{n\to\infty} f_{r,n}=f_r$
    uniformly, for fixed $r$}. Fix a neighborhood $U\supset f_r(S^1)$ in $\HH_+$; by increasing $N$, we can guarantee that 
    \[f_{r,n}(S^1)\subset U\]
    for $n\geq N$. However, $S$ is smooth on $\ovl{U}$, so in particular, it is uniformly Lipschitz on $U$; as such,
    \[\lim_{n\to\infty} S\circ f_{r,n} = S\circ f_r\]
    uniformly, for fixed $r$.

    Let $(\mu,\xi) = \cB_{S}[\lambda,\sigma]$ and $(\mu_n,\xi_n) = \cB_{S,\sigma}[\lambda_n,\sigma_n]$. By applying the convergence of $S\circ f_{r,n}$ to the case $r=0$, we see that $\xi_n\to\xi$ and $\|\lambda_n\|\to\|\lambda\|$ as $n\to\infty$. As such, if $\|\mu\|_{S^1} = M'$, we can increase $N$ to ensure that $\|\mu_n\|_{S^1}\leq 2M'$ for $n\geq N$. Next, define the following measures in $\cM_+(S^1)$:
    \[\mu_{r,n} = (2\pi)^{-1} (S^{\Re}\circ f_{r,n})(e^{i\theta})\,d\theta,\qquad \mu_r = (2\pi)^{-1} (S^{\Re}\circ f_r)(e^{i\theta})\,d\theta,\]
    and fix a bounded, continuous function $g:S^1\to\RR$. Since $P[g]$ is continuous on the compact set $\ovl{\DD}$, it is necessarily uniformly continuous. For any $\eps>0$, then, we can fix $r_\eps<1$ such that
    \[\sup\nolimits_\theta\left|g(e^{i\theta}) - P[g](re^{i\theta})\right|<\eps\]
    for $r\geq r_\eps$. Define $\widetilde{g}(e^{i\theta})=g(e^{-i\theta})$. Then we find
    \begin{align*}
        \int g\,d\mu_{r,n} &= (2\pi)^{-1}\int_0^{2\pi} g(\theta) P[\mu_n](re^{i\theta})\,d\theta\\
        &= (2\pi)^{-1}\int_0^{2\pi} g(\theta) \int_0^{2\pi} \Re\left(\frac{1+re^{i(\theta-\theta')}}{1-re^{i(\theta-\theta')}}\right)\,d\mu_n(\theta')\,d\theta\\
        &= \int_0^{2\pi} P[\widetilde{g}](re^{-i\theta'})\,d\mu_n(\theta')\\
        &= \int_0^{2\pi} P[g](re^{i\theta'})\,d\mu_n(\theta'),
    \end{align*}
    and similarly for $\mu_r$ and $\mu$; this implies
    \begin{align*}\left|\int g\,d(\mu_{n} - \mu) \right| &\leq \left|\int g\,d(\mu_{n,r_\eps} - \mu_n) \right| + \left|\int g\,d(\mu_{n,r_\eps}-\mu_{r_\eps}) \right| + \left|\int g \,d(\mu_{r_\eps} - \mu) \right|\\
    &\leq \left|\int_0^{2\pi} \left(g(e^{i\theta})-P[g](r_\eps e^{i\theta})\right)\,d\mu_n(\theta) \right| + \left|\int g\,d(\mu_{n,r_\eps}-\mu_{r_\eps}) \right| \\
    &\qquad+ \left|\int_0^{2\pi} \left(g(e^{i\theta})-P[g](r_\eps e^{i\theta})\right)\,d\mu(\theta) \right|\\
    &\leq \left|\int g\,d(\mu_{n,r_\eps}-\mu_{r_\eps}) \right| + 2\eps M'.
    \end{align*}
    Since $\eps$ was arbitrary and $\mu_{n,r_\eps}\rightharpoonup\mu_{r_\eps}$ weakly, we deduce that
    \[\int g\,d\mu_n\to \int g\,d\mu\]
    for any bounded, continuous function $g$. The proposition follows.
\end{proof}

%% file: 8_line.tex
\section{Integral Transforms on the Real Line}\label{sec:line}
We now develop our spectral theory on the real line, which provides a natural setting in which to study Volterra integral, integro-differential, delay differential, and fractional differential equations (see \cref{sec:main_line,sec:regularizedH}). As a starting point, we work to derive the explicit interconversion formula given by \cref{thm:main_formula}. One \emph{could} follow a similar logic as the preceding sections to derive a formula for \dave{general nonlinear maps of data on the real line}; for simplicity, however, we focus on the map $\cB_\RR$, which provides a solution to the convolution equations \ref{eq:integrodiff_CM} and \ref{eq:integrodiff_PD}. The results for the regularized map $\cB_\mathrm{reg}$ of \cref{sec:regularizedH} can be proven likewise\footnote{In fact, one can read the results for $\cB_\mathrm{reg}$ almost directly from \dave{the circle theory of \cref{sec:disc}}, pushing it forward to the line as necessary.}.

We first establish the following lemma:
\begin{lemma}\label{lem:Lstar}
    Let $\lambda\in L^*(\RR) +\cM_c(\RR)\subset\cM_+^{(1)}(\RR)$, in the sense that $\lambda=\lambda_1+\lambda_2$ for (possibly non-unique) $\lambda_1\in L^*(\RR)$ and $\lambda_2\in\cM_c(\RR)$. Then the Hilbert transform of $\lambda$ has the following asymptotic behavior:
    \begin{equation}\label{eq:asymptotic_Htransform}
        H_\RR[\lambda](s) = \frac{1}{\pi s} \int d\lambda + \frac{1}{\pi s^2}\int s'\,d\lambda(s') + o(s^{-2}).
    \end{equation}
    Moreover, if $\lambda_c\in L^1(\RR)$ is the continuous density of $\lambda$, then for any $c_1\geq 0$ and $c_0\in\RR$, we find
    \begin{equation}\label{eq:Lstar_in_L1}
        \frac{(1+s^2)^{-1/2}\lambda_c(s)}{\lambda_c(s)^2 + (H_\RR[\lambda] - \pi^{-1}c_0 - \pi^{-1}c_1s)^2}\in L^1(\RR).
    \end{equation}
\end{lemma}
\begin{proof}
    \dave{We prove these statements in turn. First, for any $s\in\RR$, we find
    \begin{align*}
        s\pi H_\RR[\lambda](s) = \pv\int\frac{t}{s-t}\,d\lambda(t) +  \int d\lambda= \pi H_\RR[t\mapsto t\,d\lambda(t)](s) +  \int d\lambda,
    \end{align*}
    and likewise}
    \[s\pi H_\RR[t\mapsto t\,d\lambda(t)](s) = \pi H_\RR[t\mapsto t^2\,d\lambda(t)](s) + \int t\,d\lambda(t).\]
    Now, $\lambda = \lambda_1 + \lambda_2$ for (possibly non-unique) $\lambda_1\in L^*(\RR)$ and $\lambda_2\in\cM_c(\RR)$. But since $\cF[t\mapsto t^2\,\lambda_1(t)] \in L^1(\RR)$ by hypothesis, it follows that $\cF[H_\RR[t\mapsto t^2\,\lambda_1(t)]]\in L^1(\RR)$, and the Riemann--Lebesgue lemma~\cite{royden2010real} shows that 
    \[H_\RR[t\mapsto t^2\,\lambda_1(t)](s)\in\cC_0(\RR)\]
    is a continuous function decaying to zero at infinity; since $\lambda_2$ is compactly supported, \cref{lem:smooth} likewise shows that 
    \[H_\RR[t\mapsto t^2\,\lambda_2(t)](s) = O(s^{-1})\]
    for large $s$. The asymptotic formula \cref{eq:asymptotic_Htransform} follows.

    By pulling \cref{lem:firstone} back under $\psi$, we see that the expression in \cref{eq:Lstar_in_L1} is \emph{locally} $L^1$; it remains only to check its behavior at infinity. But this follows from our asymptotic formula \cref{eq:asymptotic_Htransform}; if $H_\RR[\lambda](s)=O(s^{-1})$, then the full expression is of order $O(s\lambda_c(s))$; since $\lambda_1\in L^*(\RR)$ and $\lambda_2$ is compactly supported, this expression must be globally $L^1$.
\end{proof}

We are now in a place to prove our closed-form expression for $\cB_\RR$. Recall the statement of \cref{thm:main_formula}:
\thmmainformula*
\begin{proof}
    We prove the theorem in the case $c_0=c_1=0$, which is the most involved; the argument carries forward straightforwardly to cases where one or both parameters are nonzero.
    
    By pulling \cref{cor:countable_refined} back under $\psi$, we see that
    \[\frac{1}{Q_\RR[\lambda](z)} = Q_{\RR}\left[\frac{\lambda_c}{\lambda_c^2 + H_\sigma[\lambda]^2}\right](z) - i\pi^{-1}\zeta_0' - i\pi^{-1}\zeta_1 z +  \sum_{\alpha_j\in Z'} \beta_j Q[\psi[\delta_{\alpha_j}]](\phi^{-1}(z))\]
    for unique $\zeta_1,\beta_j\geq 0$ and $\zeta_0'\in\RR$. The term $i\pi^{-1}\zeta_1 z$ in the above equation arises from a pole at $-1 = \phi^{-1}(\infty)$ in the unit circle, as discussed in \cref{sec:main_line}. Now, although each atom $\delta_{\alpha_j}$ lies in $\cM_+^{(1)}(\RR)$, it is not necessarily clear that their sum does as well. To see that it does, note from \cref{eq:asymptotic_Htransform} that the zero set of $H_\RR[\lambda](s)$ must be bounded, and so $\sum_{\alpha_j}\delta_{\alpha_j}$ is compactly supported. Since we know it to lie in $\cM_+^{(2)}(\RR)$ (from pulling back the case of $S^1$ under $\psi$), we see that it is locally of bounded variation, and thus that $\sum_{\alpha_j}\delta_{\alpha_j}\in\cM_c(\RR)$. For a value $\zeta_0$ generally distinct from $\zeta_0'$, then, we find
    \begin{equation}\label{eq:genericform_real}
    \begin{aligned}
        \frac{1}{Q_\RR[\lambda](z)} &= Q_{\RR}\left[\frac{\lambda_c}{\lambda_c^2 + H_\sigma[\lambda]^2}\right](z) - i\pi^{-1}\zeta_0 - i\pi^{-1}\zeta_1 z +  \sum_{\alpha_j\in Z'} \beta_j Q_\RR[\delta_{\alpha_j}](z)\\
        &= Q_{\RR}\left[\frac{\lambda_c}{\lambda_c^2 + H_\sigma[\lambda]^2}\right](z) - i\pi^{-1}\zeta_0 - i\pi^{-1}\zeta_1 z +  \frac{i}{\pi}\sum_{\alpha_j\in Z'} \frac{\beta_j}{z-\alpha_j}.
    \end{aligned}
    \end{equation}
    The values $\zeta_0$ and $\zeta_1$ can be identified by studying the large-$s$ limit of $H_{\RR}[\lambda](s)$. Indeed, since the $\RR$-Hilbert transform must vanish at $\infty$, from \cref{lem:Lstar}, we know that $\zeta_0$ and $\zeta_1$ must be chosen to exactly cancel the asymptotic behavior of the imaginary component of $1/Q_\RR[\lambda](z)$. Inverting the leading terms on the right-hand side of \cref{eq:genericform_real}, we find
    \[\frac{\pi}{\zeta_1 s + \zeta_0} = \frac{\pi}{s}\left(\frac{1}{\zeta_1} - \frac{\zeta_0}{\zeta_1^2s}\right) + o(s^{-2}).\]
    By comparing against the asymptotic formula \cref{eq:asymptotic_Htransform}, we thus identify
    \[\frac{1}{\zeta_1} = \pi^{-2}\|\lambda\|,\qquad \frac{\zeta_0}{\zeta_1^2} = -\pi^{-2}\int \tau\,d\lambda(\tau),\]
    or more succinctly,
    \[\zeta_0 + \zeta_1s = \frac{\pi^2 s}{\|\lambda\|} - \frac{\pi^2}{\|\lambda\|^2}\int \tau\,d\lambda(\tau) = \frac{\pi^2}{\|\lambda\|^2}\int(s-\tau)\,d\lambda(\tau).\]
    
    We deal now with the singular contribution. Fix a value $\alpha_j\in Z'$ for which a nonzero pole exists in $1/Q_\RR[\lambda]$. For $z\in\ovl{\HH}$ in a sufficiently small neighborhood of $\alpha_j$, since $Z'$ is discrete, we have
    \[\frac{1}{Q_{\RR}[\lambda](z)} = \frac{i\beta_j/\pi}{z-\alpha_j} + o(\|z-\alpha_j\|^{-1}),\]
    and so
    \[Q_{\RR}[\lambda](z) = \frac{\pi}{i\beta_j}(z-\alpha_j) + o(\|z-\alpha_j\|).\]
    In particular, we see that
    \[H_{\RR}[\lambda](\alpha_j) \doteq \lim_{y\to 0} Q_{\RR}[\lambda](\alpha_j+iy) = 0,\]
    so any nonzero poles of $1/Q_{\RR}[\lambda]$ are contained in
    \[Z = N_0(\lambda) \cap \{s\in\RR\;|\;H_{\RR}[\lambda](s)=0\}\subset Z'.\]
    In any case, the residue theorem provides
    \[0 = \frac{1}{2\pi i}\int_{\Gamma}\frac{Q_{\RR}[\lambda](z)}{(z-\alpha_j)^2}\,dz,\]
    where $\Gamma= \Gamma_1\cup \Gamma_2\cup\Gamma_3$ is a union of (a) the horizontal line segment(s) $\{i\eps+s\;|\;\delta<|s-\alpha_j|<R\}$, (b) the intersection of $\{z\in\HH\;|\;\Im(z)>\eps\}$ with a semicircle of radius $\delta$ above $\alpha_j$, and (c) a semicircle of radius $R$ connecting the two end-points of $\Gamma_1$. The full contour is shown in \cref{fig:contour}.
    
    \begin{figure}
        \centering
        \includegraphics[width=0.75\linewidth]{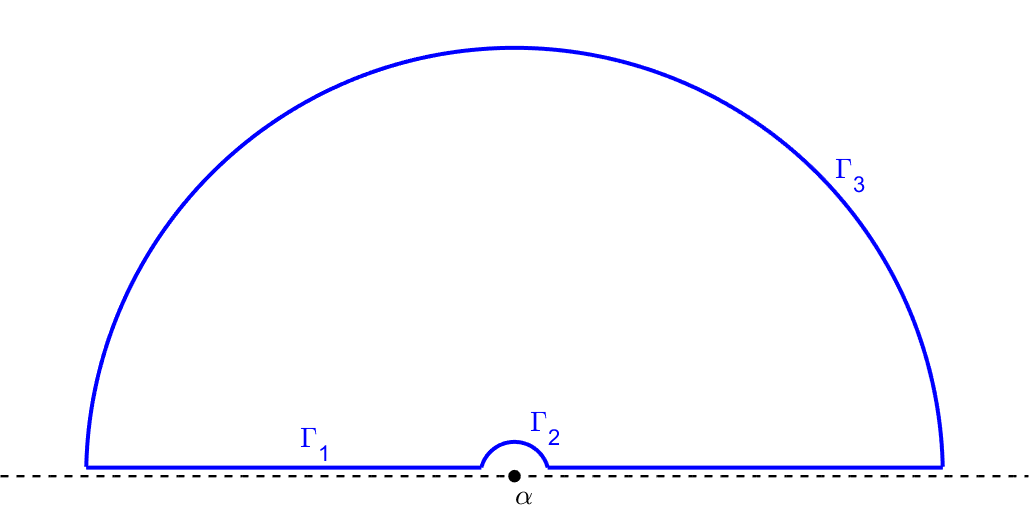}
        \caption[]{The contour $\Gamma=\Gamma_1\cup\Gamma_2\cup\Gamma_3$ applied in the proof of \cref{thm:main_formula}.}
        \label{fig:contour}
    \end{figure}
    
    As $R\to\infty$, the integral about $\Gamma_3$ tends to zero, since $Q_{\RR}[\lambda](z)$ is uniformly bounded as $z\to\infty$; as such, we discard $\Gamma_3$ and suppose that $R=\infty$. Next, taking $\eps\to 0$ with a fixed $\delta$, the $W_{-2}$ convergence (see \cref{def:winf}) of $\Re Q_{\RR}[\lambda](s+i\eps)\,ds$ to $\lambda$ shows that
    \[\lim_{\eps\to 0}\int_{\Gamma_1}\frac{\Re Q_{\RR}[\lambda](z)}{(z-\alpha_j)^2}\,dz = \int_{|s-\alpha_j|>\delta}\frac{d\lambda(s)}{(s-\alpha_j)^2}.\]
    Finally, looking at the integral about $\Gamma_2$ (which must be taken in the \emph{clockwise} direction), we see
    \[\frac{1}{2\pi i}\int_{\Gamma_2} \frac{Q_{\RR}[\lambda](z)}{(z-\alpha_j)^2}\,dz = -\frac{1}{2\pi i}\int_{0}^\pi \frac{(\pi/i\beta_j) \delta e^{i\theta} + o(\delta)}{\delta^2e^{2i\theta}}\,i\delta e^{i\theta}\,d\theta = -\frac{\pi}{2 i\beta_j} + o_\delta(1).\]
    Taking the $\eps\to 0$ limit of the (imaginary part of the) residue theorem, we find
    \[\frac{1}{2\pi}\int_{|s-\alpha_j|>\delta}\frac{d\lambda(s)}{(s-\alpha_j)^2} = \frac{\pi}{2\beta_j} + o_\delta(1),\]
    and thus
    \[\pi^{-2}\int\frac{d\lambda(s)}{(s-\alpha_j)^2} = \frac{1}{\beta_j}.\]
    Conversely, suppose that, for some $\alpha_j\in Z$, there is \emph{no} singular contribution at $\alpha_j$. Then we find $1/Q_{\RR}[\lambda](z) = o(\|z-\alpha_j\|^{-1})$ in a sufficiently small neighborhood of $\alpha_j$, and so
    \begin{equation}\label{eq:lowerboundproof}
        Q_{\RR}[\lambda](z) = \omega(\|z-\alpha_j\|).
    \end{equation}
    Suppose also that $\int(s-\alpha_j)^{-2}\,d\lambda(s)<\infty$; otherwise, we can self-consistently define $\beta_j=0$. On one hand, we find for $y>0$ that
    \begin{equation}\label{eq:boundreal}
        \Re Q_\RR[\lambda](\alpha_j+iy) =\frac{1}{\pi}\int \frac{y\,d\lambda(s)}{(s-\alpha_j)^2+y^2}\leq y\int\frac{d\lambda(s)}{(s-\alpha_j)^2} = O(y).
    \end{equation}
    On the other, consider the derivative
    \begin{equation}\label{eq:deriv_H}
        \partial_y\Im Q_{\RR}[\lambda](\alpha_j+iy) = \frac{2}{\pi}\int \frac{y(s-\alpha_j)\,d\lambda(s)}{((s-\alpha_j)^2+y^2)^2}.
    \end{equation}
    Define $\lambda_\mathrm{odd}\in\cM(\RR)$ by $d\lambda_\mathrm{odd}(s) = \frac{1}{2}\left(d\lambda(s) - d\lambda(2\alpha_j-s)\right)$, and define the sets
    \[\Lambda_+\doteq\Big\{s>\alpha_j\;\Big|\;\limsup_{\delta\to 0}\lambda_\mathrm{odd}([s-\delta,s+\delta])/2\delta > 0\Big\},\qquad \Lambda = \Lambda_+\cup\{2\alpha_j - \Lambda_+\},\]
    and the two measures
    \[\lambda_1'\doteq \lambda_\mathrm{odd}|_{\Lambda},\qquad \lambda_2'\doteq -\lambda_\mathrm{odd}|_{\RR\setminus\Lambda}.\]
    By construction, $\lambda_\mathrm{odd} = \lambda_1'-\lambda_2'$, and each of $\lambda_i'$ is non-negative over the set $[\alpha_j,\infty)\subset\RR$. Moreover, we know that
    \[\int\frac{|d\lambda_i'(s)|}{(s-\alpha_j)^2}<\infty\]
    is absolutely convergent, by our assumption of the same on $\lambda$, so that the measures $\widetilde{\lambda}_i$ defined by $d\widetilde{\lambda}_i(s) \doteq d\lambda_i'(s)/(s-\alpha_j)^2$ are each in $\cM(\RR)$. Then \cref{eq:deriv_H} can be reduced as follows:
    \[\partial_y\Im Q_{\RR}[\lambda](\alpha_j+iy) = \frac{4}{\pi}\int_{\alpha_j}^\infty \frac{y(s-\alpha_j)\,d\lambda_1'(s)}{((s-\alpha_j)^2+y^2)^2} - \frac{4}{\pi}\int_{\alpha_j}^\infty \frac{y(s-\alpha_j)\,d\lambda_2'(s)}{((s-\alpha_j)^2+y^2)^2},\]
    so that
    \begin{align*}\left|\partial_y\Im Q_{\RR}[\lambda](\alpha_j+iy)\right| &\leq \frac{4}{\pi}\int_{\alpha_j}^\infty \frac{y(s-\alpha_j)\,d\lambda_1'(s)}{((s-\alpha_j)^2+y^2)^2} + \frac{4}{\pi}\int_{\alpha_j}^\infty \frac{y(s-\alpha_j)\,d\lambda_2'(s)}{((s-\alpha_j)^2+y^2)^2}\\
    &\leq \frac{4y}{\pi}\int_{\alpha_j}^\infty \frac{(s-\alpha_j)\,d\widetilde{\lambda}_1(s)}{(s-\alpha_j)^2+y^2} + \frac{4y}{\pi}\int_{\alpha_j}^\infty \frac{(s-\alpha_j)\,d\widetilde{\lambda}_2(s)}{(s-\alpha_j)^2+y^2}\\
    &=2y\Im Q_{\RR}[\widetilde{\lambda}_1](\alpha_j+iy) - 2y\Im Q_{\RR}[\widetilde{\lambda}_2](\alpha_j+iy).
    \end{align*}
    But $Q_{\RR}[\widetilde{\lambda}_i](\alpha_j+iy)$ can approach the real line no faster than $O(1/y)$, so we see that the derivative of $\Im Q_{\RR}[\lambda](\alpha_j+iy)$ is uniformly bounded for small $y$. Since $H_\RR[\lambda](\alpha_j)=\lim_{y\to 0} \Im Q_\RR[\widetilde{\lambda}_i](\alpha_j+iy) = 0$ by hypothesis, we thus find that
    \[\Im Q_{\RR}[\lambda](\alpha_j+iy) = O(y).\]
    Together with \cref{eq:boundreal}, this violates the bound \cref{eq:lowerboundproof}, and we come to a contradiction. It follows that, if there is no singular component at $\alpha_j$, the integral defining $\beta_j^{-1}$ necessarily diverges.
\end{proof}

\Cref{prop:easyone} follows using \dave{a simpler version of the same argument}:
\propeasyone*
\begin{proof}
    Pulling \cref{thm:main_circle} back under $\psi$, we see that there is a unique $\widetilde{\mu}\in\cM_+^{(2)}(\RR)$ (along with $\zeta_0'$ and $\zeta_1$) satisfying
    \[\left(Q_\RR[\lambda](z) -i\pi^{-1}c_0 - i\pi^{-1}c_1z\right)\left(Q[\psi[\widetilde{\mu}]](\phi^{-1}(z)) - i\pi^{-1}\zeta_0' - i\pi^{-1}\zeta_1z\right)\equiv 1\]
    for any $c_1\geq 0$. Moreover, because 
    \[\left|Q_\RR[\lambda](z) -i\pi^{-1}c_0 - i\pi^{-1}c_1z\right| \geq \pi^{-1}\Im c_0 > 0\]
    everywhere in $\ovl{\HH}$, it is clear that $\widetilde{\mu}$ is absolutely continuous and that $\zeta_1 = 0$, and it must have density
    \[\widetilde{\mu}_c(s) = \frac{\lambda_c(s) + \pi^{-1}\Im c_0}{\big(\lambda_c(s) + \pi^{-1}\Im c_0\big)^2 + \big(H_\RR[\lambda](s) - \pi^{-1}(c_1s + \Re c_0)\big)^2}.\]

    Next, write $\lambda = \lambda_1+\lambda_2$ for some $\lambda_1\in L^*(\RR)$ and $\lambda_2\in\cM_c(\RR)$; since $\cF[t\mapsto t^2\lambda_1(t)]\in L^1(\RR)$, the Riemann--Lebesgue lemma~\cite{royden2010real} shows that $\lambda_1(s) = o(s^{-2})$ as $s\to\pm\infty$; combining with \cref{lem:Lstar}, we find that
    \[Q_\RR[\lambda](s) = O(s^{-1})\]
    as $s\to\pm\infty$ along the real line. Suppose first that $c_1 = 0$. Writing $\omega = -i\pi^{-1}c_0$, we expand
    \[\left(\omega + Q_\RR[\lambda](s)\right)^{-1} = \omega^{-1}\left(1 - \omega^{-1}Q_\RR[\lambda](s)\right) + O(s^{-2}),\]
    Taking real parts, we see that $\mu\doteq\Re\omega^{-1}-\widetilde{\mu}_c=O(s^{-1})$, and thus, as we know that $\widetilde{\mu}_c$ is bounded, that $\mu\in\cM_+^{(1)}(\RR)$. The $c_1>0$ case follows similarly.
\end{proof}

We now turn to \cref{thm:main_continuity}, which establishes the existence and continuity of $\cB_\RR$ on several sets of real measures. In \cref{fig:inverse_cont}, we show how this continuity can be applied to random samplings of a probability measure $\lambda$. In this example, $\lambda\in\cM_c(\RR)$ is a standard normal distribution, cropped to the set $\{t\in\RR\;|\;d\lambda(t)/dt \geq 10^{-15}\}$, and $\lambda^{(n)}$ are empirical distributions corresponding to $n$ i.i.d.~samples from $\lambda$. As $n$ increases, we see that
\[\mu^{(n)}\doteq \cB_\RR[\lambda^{(n)},0,1]\to\cB_\RR[\lambda,0,1] \doteq \mu.\]

We focus first on proving the continuity of $\cB_\RR$ on the space $\cM_c(\RR)$ of compactly-supported, non-negative measures on the real line. In this setting, we make use of the $W_\infty$ topology introduced in \cref{def:winf} and characterized by \cref{prop:topology}. 
The second half of \cref{thm:main_continuity} (relating to compactly-supported measures) can be proved as follows:

\begin{figure}
    \centering
    \includegraphics[width=\linewidth]{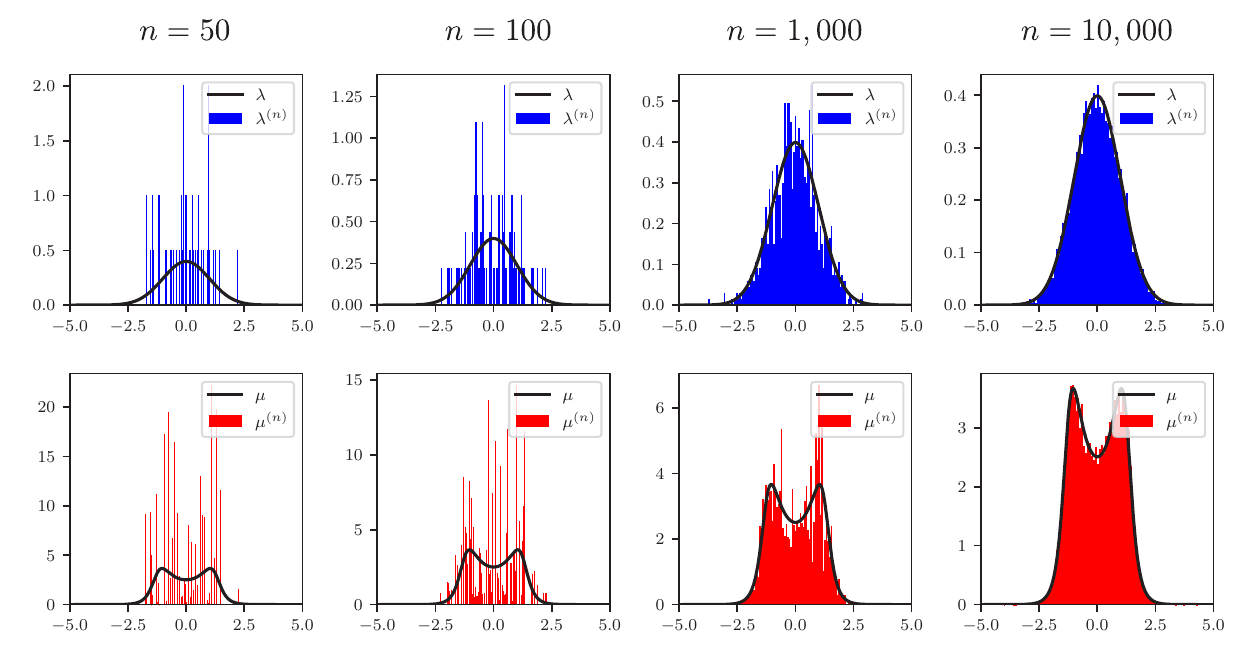}
    \caption[Continuity of the Interconversion Map]{Continuity of the map $\cB_\RR$, where $\lambda$ is the standard normal distribution (cropped such that $d\lambda(t)/dt > 10^{-15}$) and $\mu \simeq (\mu,0,0) = \cB_\RR[\lambda, 0, 1]$. Here, $\lambda^{(n)}$ is an empirical distribution of $n$ i.i.d.~samples from $\lambda$, and we define $\mu^{(n)} \simeq (\mu^{(n)},0,0)= \cB_\RR[\lambda^{(n)}, 0, 1]$. We see that $\mu^{(n)}$ converges to $\mu$ as $n \to \infty$, as predicted from \cref{thm:main_continuity}.}
    \label{fig:inverse_cont}
\end{figure}

\begin{lemma}\label{lem:compactcontinuity}
    Write $U^{0},U^1,U^2$ as in \cref{thm:main_continuity}. For each $i\in\{0,1,2\}$, the map $\cB_\RR:\cM_c(\RR)\times U^i\to\cM_c(\RR)\times U^{2-i}$ is well-defined and continuous in the product of the $W_\infty$-topology on $\cM_c$ and the standard topology on $U^i,U^{2-i}$.
\end{lemma}
\begin{remark}
    Although one could use \cref{thm:main_formula} to derive a result of this form, we opt instead to pull back the circle theory directly. As a result, the lemma does not require any knowledge of the size of $N_0(\lambda)$.
\end{remark}
\begin{proof}
    Let $\lambda\in\cM_c(\RR)$, and let $I\subset\RR$ be a finite interval containing $\supp\lambda$. Write
    \[\widetilde{\lambda} \doteq \psi[\lambda] + 2c_1\delta_{-1}\in\cM_+(S^1),\]
    with $\delta_{-1}$ a Dirac measure at $-1=\phi^{-1}(\infty)$. Then from \cref{thm:main_circle}, there is a unique pair $(\widetilde{\mu},\zeta_1')$ such that
    \begin{equation}\label{eq:continuity_pullback}
        (Q[\widetilde{\lambda}](z) - i\pi^{-1}c_0')(Q[\widetilde{\mu}](z) - i\pi^{-1}\zeta_0') \equiv 1
    \end{equation}
    on $\DD$, with $c'_0 = c_0-\pi\sigma(\lambda)$. From \cref{prop:support}, we know that the support of the absolutely continuous component $\widetilde{\mu}_c$ of $\widetilde{\mu}$ is bounded away from $-1$. On the other hand, the singular component of $\widetilde{\mu}$ is supported exactly where $\widetilde{\mu}_c(z)=0$ and $Q[\widetilde{\lambda}](z) = i\pi^{-1}c_0'$ on the unit circle. If $c_0=c_1=0$, there is thus an isolated pole at $-1$, and we find
    \[\widetilde{\mu} = \psi[\mu] + 2\zeta_1\delta_{-1}\]
    for a compactly-supported $\mu\in\RR$ and a nonzero $\zeta_1>0$; that $\zeta_1$ is nonzero can be deduced as in the proof of \cref{thm:main_formula}. This argument shows that $\cB_\RR$ maps $\cM_c(\RR)\times U^0$ to $\cM_c(\RR)\times U^{2}$. 
    
    If either $c_1$ or $c_0$ is nonzero, then from \cref{lem:smooth}, the singular component of $\widetilde{\mu}_c$ is supported on at most one point in each component of $S^1\setminus\left(\{-1\}\cup\supp\phi^{-1}(I)\right)$. Thus, $\cB_\RR$ maps $\cM_c(\RR)\times(U^1\cup U^2)$ to $\cM_c(\RR)\times(U^0\cup U^1)$. Suppose $c_1\neq 0$; then $Q[\widetilde{\lambda}](z) - i\pi^{-1}c_0'\to \infty$ as $z\to-1$ along non-tangential directions, but \cref{eq:continuity_pullback} thus implies that $Q[\widetilde{\mu}](z) - i\pi^{-1}\zeta_0'\to 0$ along the same. This is only possible if $\zeta_0=0$, so we find that $\cB_\RR$ maps $\cM_c(\RR)\times U^2$ to $\cM_c(\RR)\times U^0$. That $\cB_\RR$ maps $\cM_c(\RR)\times U^1$ to itself follows similarly.

    We turn now to the claim of continuity. Fix a nonzero $\lambda\in\cM_c(\RR)$ and $(c_0,c_1)\in U^i$, and suppose $\lambda_n\in\cM_c(\RR)$ converges in $W_\infty$ to $\lambda$, and $(c_{0,n},c_{1,n})\in U^i$ converges to $(c_0,c_1)$. For $\delta>0$, define the following complex functions on $\RR$:
    \[f_{\delta,n}(s) = Q_\RR[\lambda_n](s+i\delta) - i\pi^{-1}\left(c_{0,n}+c_{1,n}(s + i\delta)\right).\]
    Following the same argument as in \cref{prop:weakcont}, we can deduce that
    \[\lim_{n\to\infty} f_{\delta,n} = Q_\RR[\lambda](s+i\delta) - i\pi^{-1}\left(c_{0}+c_{1}(s + i\delta)\right) \doteq f_\delta\]
    \emph{locally} uniformly. Next, for any $R>0$, fix a neighborhood $U_R\supset f_\delta([-R,R])$, and choose $N_R\geq 1$ such that
    \[f_{\delta,n}([-R,R])\subset U_R\]
    for all $n>N_R$. Since $S:z\mapsto 1/z$ is smooth on $\ovl{U}_R$, we deduce (similar to \cref{prop:weakcont}) that
    \[\lim_{n\to\infty} 1/f_{\delta,n} = 1/f_\delta\]
    uniformly on $[-R,R]$, and thus locally uniformly.

    Now, fix $(\mu_n,\zeta_{0,n},\zeta_{1,n}) = \cB_\RR[\lambda_n,c_{0,n},c_{1,n}]$ and $(\mu,\zeta_0,\zeta_1)=\cB_\RR[\lambda,c_0,c_1]$. Notably, explicit formulas for $\zeta_0$ and $\zeta_1$ can be recovered as in the proof of \cref{thm:main_formula}, so the convergence of $\zeta_{0,n}\to\zeta_0$ and $\zeta_{1,n}\to\zeta_1$ is clear; the only non-trivial case occurs for $c_0=c_1=0$, for which convergence of such formulas follows from the $W_\infty$ convergence of $\lambda_n$ to $\lambda$. 
    
    It remains only to be seen that $\mu_n$ converge in $W_\infty$ to $\mu$. Fix $N\geq 1$ and an interval $I\subset\RR$ such that $\supp\lambda_n\subset I$ for all $n\geq N$; that such a choice is possible follows from \cref{def:winf}. Defining $M = \|\lambda\|$ and increasing $N$ if necessary, we can also suppose that $M/2\leq \|\lambda_n\|\leq 2M$ for all $n\geq N$. 

    Let $s_- = \inf I$ and $s_+ = \sup I$. Then for $s\geq s_+$, we find
    \[H_\RR[\lambda_n](s) = \frac{1}{\pi}\int\frac{d\lambda_n(s')}{s-s'}\leq \frac{1}{\pi}\int\frac{d\lambda_n(s')}{s-s_+} \leq \frac{2M/\pi}{s-s_+}.\]
    With a similar procedure, we find the string of inequalities
    \begin{equation}\label{eq:ineqs}
        \begin{gathered}
            0<\frac{M/2\pi}{s-s_-}\leq H_\RR[\lambda_n](s) \leq \frac{2M/\pi}{s-s_+},\qquad s\geq s_+,\\
            \frac{2M/\pi}{s-s_-}\leq H_\RR[\lambda_n](s) \leq \frac{M/2\pi}{s-s_+}<0, \qquad s\leq s_-,
        \end{gathered}
    \end{equation}
    and likewise for $H_\RR[\lambda]$. Of course, since $H_\RR[\lambda_n]$ is smooth (and thus everywhere well-defined) outside of $I$, any poles of 
    \[Q_\RR[\mu_n](z) - i\pi^{-1}(\zeta_{n,0} +\zeta_{n,1}z)=\left(Q_\RR[\lambda_n](z) - i\pi^{-1}(c_{n,0} + c_{n,1}z)\right)^{-1}\]
    in $\RR\setminus I$ must correspond to zeroes of \[H_\RR[\lambda_n](s) - i\pi^{-1}(c_{n,0} + c_{n,1}s).\] Combining this argument with \cref{prop:support}, we find that
    \[
        \supp\mu_n,\supp\mu\subset I\cup\{s\in\RR\setminus I\;|\;H_\RR[\lambda_n](s) - i\pi^{-1}(c_{n,0} + c_{n,1}s) = 0\}.
    \]
    First, in the case $c_0=c_1=0$, this argument shows that $\supp\mu_n,\supp\mu\subset I$. In the case that $c_0\neq 0$ (regardless of the value of $c_1$), increase $N$ such that $|c_0-c_{n,0}|<|c_0|/2$ for all $n\geq N$. Then the inequalities \cref{eq:ineqs} show that
    \[\supp\mu_n,\supp\mu\subset[-R_1,R_1],\qquad R_1 = \frac{4M}{|c_0|} +|s_+|+|s_-|.\]
    In the case that $c_1>0$ but $c_0=0$, increase $N$ once again such that $|c_1-c_{n,1}|<|c_1|/2$ for all $n\geq N$; the same inequalities then show that
    \[\supp\mu_n,\supp\mu\subset[-R_2,R_2],\qquad R_2 = \sqrt{|s_+|^2+|s_-|^2+4M/|c_1|}+|s_+|+|s_-|.\]
    By expanding $I$ appropriately, then, we can suppose that
    \[\supp\mu_n,\supp\mu\subset I\]
    for $n\geq N$; note that the inequalities \cref{eq:ineqs} continue to hold with the new definition of $I$. Fix a neighborhood $I'\supset I$, and define the measures
    \[\mu_{n,\delta} = P[\mu_n](s+i\eps)\chi_{I'}(s)\,ds\in\cM_c(\RR),\]
    where $\chi_{I'}$ is the characteristic function of $I'$. These measures converge weakly to $\mu_n$ as $\delta\to 0$, by \cref{prop:classic}.\ref{prop:classic3}. Fix a bounded, continuous function $g:\RR\to\RR$. For $\eps>0$, fix $\delta_\eps>0$ such that 
    \[\sup_{s\in I',\;\delta<\delta_\eps}\left|g(s) - P[g](s+i\delta)\right| < \eps,\]
    using the locally uniform continuity of $P[g]$. Decreasing $\delta_{\eps}$ if necessary, we can ensure also that 
    \[\int_{\RR\setminus I'} P[\mu_n](s+i\delta)\,ds \leq \frac{2M}{\pi}\int_{\RR\setminus I'}\frac{\delta}{s^2+\delta^2}\left(\delta(s-s_-)+\delta(s-s_+)\right)\,ds < \eps \]
    for all $n$ and all $\delta<\delta_\eps$. With this choice, we can apply the same argument as we did in \cref{prop:weakcont} to deduce that $\mu_n\rightharpoonup\mu$ weakly. But we also know that $\mu_n$ are uniformly compactly supported, so we recover convergence in $W_\infty$.
\end{proof}

We can now prove \cref{thm:main_continuity} in full:
\thmmaincontinuity*
\begin{proof}
    The claim about $\cM_c(\RR)$ is proven in \cref{lem:compactcontinuity}, so we prove only the statements about $\cM_\text{exp}^{(1)}(\RR)$ here. In general, it is clear that the restriction of the embedding $\Psi$ (defined by \cref{eq:embedding}) to $\cM_+^{(1)}\times U^i$ is continuous from the $W_{-2}$ topology on $\cM_+^{(1)}(\RR)\subset\cM_+^{(2)}(\RR)$ and the standard topology on $U^i$ to the weak topology on $\cM_+(S^1)$ and the standard topology on $\RR$. Write $\Psi[\lambda,c_0,c_1] = (\widetilde{\lambda},c'_0)$. From \cref{thm:main_circle}, we thus see that there is a unique pair $(\widetilde{\mu},\zeta_0')$ such that
    \begin{equation}\label{eq:continuity_pullback2}
        (Q[\widetilde{\lambda}](z) + ic_0')(Q[\widetilde{\mu}](z) + i\zeta_0') \equiv 1,
    \end{equation}
    and that the map $(\lambda,c_0,c_1)\mapsto (\widetilde{\mu},\zeta_0')$ is continuous in the same topologies. Moreover, so long as either $c_0$ or $c_1$ is nonzero, we can follow the same logic as \cref{lem:compactcontinuity} to deduce that $\widetilde{\mu}$ has no atom at $-1$, and thus that $\widetilde{\mu}=\psi[\mu]$ for some $\mu\in\cM_+^{(2)}(\RR)$.
    
    Since $\lambda\in\cM_\mathrm{exp}^{(1)}(\RR)$, we further deduce that $\mu\in\cM_\mathrm{exp}^{(2)}(\RR)$, as it can have at most one atom to the left of $\inf\supp\lambda$. Next, suppose that $\mu\notin\cM_+^{(1)}(\RR)$, and calculate
    \[Q[\psi[\mu]](\phi^{-1}(z)) = \frac{i}{\pi}\int\left(\frac{1}{z-s} + \frac{s}{1+s^2}\right)\,d\mu(s).\]
    Fix $R>0$ sufficiently large and $z_0<\inf\supp\mu$ sufficiently small that $(z-s)^{-1}>-s(1+s^2)^{-1}$ for all $s>R$ and $z<z_0$; for instance, $R=2$ and $z_0 = \min(-1,\inf\supp\mu)$ is sufficient. Then we find
    \[\lim_{z\to-\infty}Q[\psi[\mu]](\phi^{-1}(z)) = \frac{i}{\pi}\int_{-R}^R\frac{s\,d\mu(s)}{1+s^2} + \lim_{z\to-\infty}\frac{i}{\pi}\int_{R}^{\infty}\left(\frac{1}{z-s} + \frac{s}{1+s^2}\right)\,d\mu(s),\]
    but now that the latter integrand is positive, a standard application of Fatou's lemma shows that the second term diverges. Since at least one of $c_0$ and $c_1$ is nonzero by hypothesis, this contradicts the (pushforward of the) relation \cref{eq:continuity_pullback2}. Thus, $\mu\in\cM_\mathrm{exp}^{(1)}(\RR)$.
    
    Finally, let $r>2$, and consider a sequence $\widetilde{\mu}_n\in\cM_+(S^1)$ converging weakly to $\widetilde{\mu}\in\cM_+(S^1)$. Let $\mu_n,\mu$ be such that $\psi[\mu_n]+\pi^{-1}\widetilde{c}_{1,n}\delta_{-1}=\widetilde{\mu}_n$ and $\psi[\mu]+\pi^{-1}\widetilde{c}_1\delta_{-1} = \widetilde{\mu}$, for some values $\widetilde{c}_n$ and $\widetilde{c}$. For any bounded, continuous $f\in\cC(\RR)$, the function
    \[\widetilde{f}(z)\doteq (1+\phi(z)^2)^{1-r/2}(f\circ\phi)(z)\]
    is bounded and continuous on $S^1$, and so
    \[\int  (1+s^2)^{-r/2} f(s)\,d\mu_n(s) = \pi\int \widetilde{f}\,d\widetilde{\mu}_n\to \pi\int\widetilde{f}\,d\widetilde{\mu} = \int  (1+s^2)^{-r/2} f(s)\,d\mu(s)\]
    as $n\to\infty$. It follows that the projection $\psi^{-1}$ is continuous from the weak topology on $S^1$ to the $W_{-r}$ topology on $\cM_+^{(2)}(\RR)$, so the map $(\lambda,c_0,c_1)\mapsto \mu$ is continuous. The theorem follows.
\end{proof}

%% file: 9_numerics.tex
\section{Applications and Numerics of Interconversion Theory}\label{sec:numerics}
In this section, we show how the theory developed so far can be implemented numerically, giving rise to a simple-but-powerful spectral approach for working with scalar Volterra equations of \dave{all classes under consideration}.

Central to our approach is the \george{Adaptive Antoulas--Anderson, or \emph{AAA} (`triple-A'),} algorithm for rational approximation~\cite{nakatsukasa2018aaa}, which we use for two reasons. For one, the measures we work with---as well as their integral transforms---are generally non-smooth, so traditional (e.g., polynomial) approximation schemes are ill-suited to our problem. Equally important is, in handling time series, we are often interested in equispaced (or \dave{arbitrarily-spaced}) samples on the real line. We require high-order methods in order to accurately approximate integral transforms of such data, but polynomial methods give rise to large, non-physical oscillations when applied to equispaced sample points~\cite{doi:10.1137/1.9781611975949}. Although somewhat less foolproof than polynomial interpolation, rational approximation is able to cleanly resolve discontinuities, poles, and branch points, and it does not depend nearly as strongly as polynomial methods on the distribution of sample points.

\dave{The numerical results presented here make use of low-order quadrature schemes to compute certain intermediate expressions, so they generally achieve only low-order accuracy. Even still, we will see sharp improvement over existing approaches in various contexts, such as Volterra integral equations of the first kind (\cref{subsec:num_trajectories}) and discrete-time Volterra equations with non-decaying kernels (\cref{sec:discrete_num}). In the sequel, we extend our approach to achieve high-order accuracy and improved time complexity, and we see how it applies to problems well beyond the scope of our analytical results~\cite{sieve}}.

Our codebase, complete with all examples presented here, has been made available at the following GitHub link:
\ourgithublink
All numerical experiments are performed on a 2021 MacBook Pro personal computer (Apple M1 Pro) with 10 CPU cores and 16 GB of memory.

\davebegin
\subsection{Numerical Methods}\label{sec:B_implem}
We present the core numerical methods needed to compute Cauchy transforms, Hilbert transforms, and interconversion maps of measures. We begin by reviewing the AAA algorithm, following Nakatsukasa et al.~\cite{nakatsukasa2018aaa}.

\subsubsection{AAA Approximation}

AAA aims to approximate a complex-valued function $f(z)$ given its values at a finite set of sample points $Z \subseteq \CC$. Specifically, its goal is to produce a rational approximation $r(z)$ of $f$ such that
\begin{equation}\label{eq:AAA_tolerance}
    \max_{z\in Z}|r(z)-f(z)|<\eps
\end{equation}
for a given tolerance $\eps>0$.

The algorithm proceeds iteratively. At every iteration $m = 1, 2, 3, \dots$, we begin with a candidate rational approximant $r_{m-1}(z)$ and a list of \emph{support points} $S_{m-1} = [z_1,...,z_{m-1}]$, initialized with $r_0(z)\equiv 0$ and $S_0 =\emptyset$, respectively. We identify a previously-unvisited support point $z_m$ by maximizing $|f(z_m)-r_{m-1}(z_m)|$ over $Z\setminus S_{m-1}$. Writing $f_i = f(z_i)$, we then aim to choose an approximant of the form
\begin{equation}\label{eq:bary}
    r_m(z) = \frac{n_m(z)}{d_m(z)} = \sum_{i=1}^m \frac{w_if_i}{z - z_i} \Big/ \sum_{i=1}^m \frac{w_i}{z - z_i},
\end{equation}
where the $w_i$ are as-yet-undetermined complex numbers. It is easy to verify that such a \textit{barycentric representation} necessarily interpolates $f$ at $S_m$. We determine the weights $w = (w_1, \dots, w_m)^\top$ by solving the constrained least-squares problem
\begin{equation}
    \min_{w \in \CC^{m}} \|f(z)d_m(z) - n_m(z)\|_{z\in Z \setminus S_m}, \quad \|w\| = 1,
\end{equation}
denoting by $\|\cdot\|$ the Euclidean norm. This problem can in turn be solved by evaluating the SVD of an appropriate Loewner matrix. 

The algorithm terminates when we reach the bound \cref{eq:AAA_tolerance}; it is clear that it must terminate within $M/2$ iterations. In pseudocode, the algorithm reads as follows:

\bigskip

\begin{algorithm}[H]
\DontPrintSemicolon
\KwIn{$\{z_i, f_i = f(z_i)\}_{i=1}^M$, $\eps>0$}
\KwOut{$r(z)$}
Set $S \leftarrow []$\\
Set $d(z) = 1$, $n(z) = 0$, and $r(z) = n(z)/d(z)$\\
\For{$m = 1$ \KwTo $M/2$}{
\begin{enumerate}
    \item Append $S \leftarrow S + [z_m]$ for
    $$z_m =\argmax_{z \in Z \setminus S} |f(z) - r(z)|$$\\
    \item Write 
    $$n(z) = \sum_{i=1}^m \frac{w_if_i}{z - z_i}, \quad d(z) = \sum_{i=1}^m \frac{w_i}{z - z_i}$$
    with $w_i$ undetermined, then find least singular vector
    $$w =\argmin_{\|w\| = 1} \|f(z)d(z) - n(z)\|_{z\in Z\setminus S_m}$$\\
    \item Set rational function $r(z) = n(z)/d(z)$\\
    \item \textbf{if} $\max_Z|f - r| \leq \eps$ \textbf{then break}
\end{enumerate}
}
\Return{$r(z)$}
\caption{Adaptive Antoulas–Anderson (AAA)}
\end{algorithm}

\bigskip

Notably, the poles, residues, and zeroes of the rational function $r$ can be computed through a simple generalized eigenvalue problem, which we do not discuss here. Poles with small residues or closely-spaced pole-zero pairs, known as Froissart doublets, are typically removed in post-processing. We refer the reader to Nakatsukasa et al.~\cite[Sec.~3]{nakatsukasa2018aaa} for additional details of the algorithm.

\subsubsection{Hilbert and Cauchy Transforms, and Interconversion in the Spectral Domain}\label{subsubsec:cauchy}

An important variation on the AAA algorithm is the \emph{AAA Least Squares} (AAA-LS) algorithm formulated by Costa and Trefethen~\cite{costa2023aaa}. Given a domain $\Omega\subset \CC$ and a real-valued function $f(z)$ on $\partial\Omega$, the AAA-LS algorithm attempts to find a holomorphic function $h(z)$ on $\Omega$ such that $\Re h(z)\approx f(z)$ on $\partial\Omega$. We use a similar approach as AAA-LS to approximate the Cauchy and Hilbert transforms of measures on $\RR$ or $S^1$; below, we assume our measures consist of a continuous density and only finitely many discrete atoms:

\bigskip

\davebegin
\begin{algorithm}[H]
\DontPrintSemicolon
\KwIn{$\Omega = \HH$ or $\DD$, $\lambda = \lambda_c + \sum_{i}b_i\delta_{a_i}$, $\{z_i \in \partial \Omega\}_{i=1}^M$, $\eps>0$}
\KwOut{$Q^{(\Omega)}(z), H^{(\Omega)}(z)$, i.e., either $(Q_\RR,H_\RR)$ or $(Q,H)$}
\begin{enumerate}
\item Run AAA on continuous density
$$r(z) = \text{AAA}\big(\{z_i, \lambda_c(z_i)\}_{i=1}^M, \eps\big),$$
and let $p_i\in\CC\setminus\Omega$ be the poles of $r(z)$ outside $\Omega$\\
\item Create a holomorphic function on $\Omega$ of the form
$$Q_\text{cont}(z) = \sum_{i=0}^d a_iz^i + \sum\frac{c_i}{z - p_i},$$
with $d\geq 0$ a desired degree, and perform a least-squares optimization on $\{a_i, c_i\}$ to minimize the Euclidean error $\|\Re h(z_i)-\lambda_c(z_i)\|$
\item Compute analytical Cauchy transform of discrete part
$$Q_\text{disc}(z) = \sum b_iQ^{(\Omega)}[\delta_{a_i}](z)$$\\
\item Compute the Cauchy and Hilbert transforms
$$Q^{(\Omega)}(z) = Q_\text{cont}(z) + Q_\text{disc}(z), \quad H^{(\Omega)}(z) = \Im Q^{(\Omega)}(z)\big|_{\partial D}$$
\end{enumerate}
\Return{$Q^{(\Omega)}(z), H^{(\Omega)}(z)$}
\caption{Cauchy and Hilbert Transforms via Rational Approximation}\label{alg:cauchy}
\end{algorithm}

\bigskip

Next, we implement the triple of involutions $\cB$, $\cB_\RR$, and $\cB_\mathrm{reg}$ introduced in \cref{sec:main}. In turn, these maps allow us to solve difference equations (see \cref{prop:main_dPD}), integral and integro-differential equations (see \cref{prop:main_CM,prop:main_PD}), and delay and fractional differential equations (see \cref{prop:rPD,prop:rCM}).

The map $\cB$ takes a measure $\lambda\in\cM_+(S^1)$ and an offset $c_0\in\RR$ and returns the `interconverted' pair $(\mu,\zeta_0)\in\cM_+(S^1)\times\RR$ defined by \cref{thm:main_circle}. We can compute it in closed form by following \cref{thm:main_circ_formula}. In turn, we need two intermediate expressions: the Hilbert transform of $\lambda$, which we can now compute via Algorithm~\ref{alg:cauchy}, and the zeroes of $H_\RR[\lambda]+c_0$, which we can compute either with a standard rootfinding procedure or with AAA itself. Rootfinding is particularly easy in this context, since each component of $S^1\setminus\supp\lambda$ has at most one root.

The map $\cB_\RR$ is computed similarly, but on the real line instead of the circle. Recall that $\cB_\RR$ takes a measure $\lambda\in\cM_+^{(1)}(\RR)$, a constant offset $c_0\in\RR$, and a linear offset $c_1\geq 0$, and returns the interconverted triple $(\mu,\zeta_0,\zeta_1)\in\cM_+^{(1)}(\RR)\times\RR\times\RR_+$. In numerical tests involving $\cB_\RR$, we treat only measures satisfying the hypotheses of \cref{thm:main_formula}, so we know with certainty that $\cB_\RR$ is well-defined.

The map $\cB_\mathrm{reg}$ takes a measure $\lambda\in\cM_+^{(2)}(\RR)$, a constant offset $c_0\in\RR$, and a linear offset $c_1\geq 0$, and it returns the interconverted triple $(\mu,\zeta_0,\zeta_1)\in\cM_+^{(2)}(\RR)\times\RR\times\RR_+$. We implement $\cB_\mathrm{reg}$ by pulling back our existing implementation of $\cB$, i.e., using the formula $\cB_\mathrm{reg} = \Psi^{-1}_\mathrm{reg}\circ\cB\circ\Psi_\mathrm{reg}$, where the embedding $\Psi_\mathrm{reg}$ is defined by \cref{eq:embedding_reg}.

\subsubsection{Mapping Between the Spectral and Time Domains}

Finally, there are several approaches one can take to map between the spectral domain and the time domain. At present, we use trapezoidal quadrature to numerically compute Laplace and (continuous) Fourier transforms of known measures; we develop an improved, spectral approach for this calculation in the sequel~\cite{sieve}. In the other direction, our approach differs somewhat between the \ref{eq:integrodiff_CM} and \ref{eq:integrodiff_PD} cases.

Given values of a completely monotone kernel $K(t)$ in the time domain, we need to compute its inverse Laplace transform $\lambda\in\cM^{(1)}(\RR)$. Computing inverse Laplace transforms is closely related to the problem of fitting sums of exponentials from data; indeed, if $\lambda = \sum \lambda_i\delta(s-\alpha_i)\,ds$, then our approximation $\lambda\approx\cL^{-1}[K]$ corresponds to the exponential sum
\begin{equation}\label{eq:sieve_expseries}
    K(t)\approx \sum \lambda_i e^{-\alpha_i t}.
\end{equation}
The problem of fitting sums of exponentials is well-explored and known to be very ill-posed~\cite{van1993resolvability, waterfall2006sloppy, hauer1990initial, yeramian1987analysis}. Here, we explore the use of the AAA algorithm~\cite{nakatsukasa2018aaa} to solve this ill-posed problem. The approach we describe below is closely related to the method of Pad\'{e}--Laplace approximation~\cite{yeramian1987analysis}, which uses a one-point rational approximant (analogous to a Taylor series) to compute integral transforms.

Taking a (one-sided) Laplace transform of $K$, we find (as in \cref{sec:laplace})
\begin{equation}\label{eq:LK}
    \cL[K](s) = \cL[\cL_b[\lambda]](s) = \int_0^\infty\int e^{-st}e^{-tu}\,d\lambda(u)dt = \int \frac{d\lambda(u)}{s+u},
\end{equation}
suggesting that we can construct a discrete approximation of $\lambda$ by first approximating $\cL[K]$ with an appropriate rational function. Instead of computing $\cL[K]$ with direct quadrature, as we do elsewhere, we now use AAA to fit a rational approximation to the sample data $\{(t_i, K(t_i))\}_{i=1}^n$ to obtain\footnote{Generically, AAA may give constant or polynomial terms. We discuss how to correct this behavior in the sequel.}
\begin{align*}
    K(t) \approx \widehat{K}(t) \doteq \sum \frac{w_i}{t - z_i},\qquad \cL[\widehat{K}](s) = \sum w_ie^{-z_is}E_1(-z_is),
\end{align*}
with $w_i, z_i \in \mathbb{C}$, and where $E_1(x) = \int_x^\infty \frac{e^{-u}}{u}du$ denotes the exponential integral~\cite{abramowitz1965handbook}. We now apply AAA a \emph{second} time, to approximate $\cL[\widehat{K}]$ on a set of chosen (typically logarithmically-spaced) quadrature points $\{s_i\}_{i=1}^m$. This procedure yields
\begin{align*}
    \cL[\widehat{K}](s) \approx \sum \frac{\rho_i}{s - \zeta_i},
\end{align*}
with $\rho_i, \zeta_i \in \mathbb{C}$. Already, this calculation can be transformed back to the time domain to yield an exponential approximant of the form~\cref{eq:sieve_expseries}; we investigate this \emph{AAA-Laplace} algorithm in the sequel~\cite{sieve}.

At present, we develop the algorithm a step further to ensure that the resulting approximant is itself CM. In approximating CM kernels in the spectral domain, AAA tends to concentrate the poles $\zeta_i$ along the negative real axis\footnote{This is a manifestation of a more general principle of AAA, namely, that it concentrates zeroes and poles along singular sets of the target function.}. We thus construct an initial approximation of $\lambda$ by projecting these poles to the real line:
\begin{equation}
    \widehat{\lambda}(s) = \sum_{i=1}^m \lambda_i\delta(s - \alpha_i)ds, \qquad \alpha_i = -\Re[\zeta_i], \ \lambda_i = \Re[\rho_i].
\end{equation}
Finally, we set to zero those weights $\lambda_i$ which are negative, and we optimize the remaining pairs $(\alpha_i, \lambda_i)$ through a projected gradient descent, minimizing mean squared error while constraining $\lambda_i\geq 0$. 

This procedure can be applied even when $K$ exhibits exponential growth (i.e., if $K$ is gCM rather than CM). If $K$ grows at a rate $\tau>0$, then we perform an AAA-based Laplace transform on $\widehat{K}(t) = e^{-\tau t}K(t)$ and recover the transform of $K$ as $\cL[K](s) = \cL[\widehat{K}](s-\tau)$. We summarize our inverse Laplace transform algorithm in the following pseudocode:

\bigskip

\begin{algorithm}[H]
\DontPrintSemicolon
\KwIn{$\{t_i, K(t_i)\}_{i=1}^n, \{s_i\}_{i=1}^m, \tau \geq 0$}
\KwOut{$\widehat{\lambda} = \sum_{i=1}^m \lambda_i\delta_{\alpha_i}$}
\begin{enumerate}
\item Rescale kernel $K(t) \leftarrow e^{-\tau t}K(t)$
\item First application of AAA to approximate kernel at time samples $\{t_i\}_{i=1}^n$
$$K(t) \approx \widehat{K}(t) \doteq \sum\frac{w_i}{t - z_i}$$\\
\item Take analytical Laplace transform of rational approximant
$$\cL[\widehat{K}](s) = \sum w_ie^{-z_is}E_1(-z_is)$$
\item Second application of AAA at spectral sample points $\{s_i\}_{i=1}^m$
$$\cL[\widehat{K}](s) \approx \sum \frac{\rho_i}{s - \zeta_i}$$
\item Round poles and residues to real axis, $\alpha_i = -\Re[\zeta_i]$, $\lambda_i = \Re[\rho_i]$, and estimate empirical measure and associated kernel
$$\widehat{\lambda}(s) = \sum \lambda_i\delta(s - \alpha_i)ds, \qquad \widehat{K}(t) = \cL_b[\widehat{\lambda}](t) = \sum \lambda_ie^{-\alpha_i t}$$
\item Reoptimize real values $\{\alpha_i, \lambda_i\}$ through gradient descent (Adam) and optionally project these values to be positive after each step, minimizing the least squares objective on the time samples $\{t_i\}_{i=1}^m$
$$\min_{\alpha_i, \lambda_i}\|\widehat{K} - K\|_{\{t_i\}_{i=1}^m}$$
\end{enumerate}
\Return{$\widehat{\lambda} = \sum_{i=1}^m \lambda_i\delta_{\alpha_i}$}
\caption{AAA-Based Inverse Laplace Transform of gCM Kernels}
\label{alg:aaa_laplace}
\end{algorithm}

\bigskip

An AAA-based approach could be used equally well in the setting of \ref{eq:integrodiff_PD}, but we instead demonstrate an alternate approximation scheme using the discrete cosine transform (DCT):

\vspace{5pt}

\begin{algorithm}[H]
\DontPrintSemicolon
\KwIn{$\{t_j = (j-1)\Delta t, K(t_j)\}_{j=1}^n$}
\KwOut{$\widehat{\lambda} = \sum_{i=1}^n \tfrac{\lambda_i}{2}(\delta_{-\omega_i} + \delta_{\omega_i})$}
\begin{enumerate}
\item Perform an inverse DCT on the sample vector $\{(K(t_i)\}_{i=1}^n$ to obtain
\begin{align*}
    \widehat{\lambda}(s) = \sum \frac{\lambda_i}{2}\Big[\delta(s - \omega_i) + \delta(s + \omega_i)\Big]\,ds,
\end{align*}
with weights and frequencies
\begin{align*}
    \lambda_j = \frac{2}{n}\Big(\frac{1}{2}K(0) + \sum K(t_k)\cos(\omega_jt_k)\Big), \qquad \omega_j = \frac{\pi(2j-1)}{2n\Delta t}.
\end{align*}
\item Write an estimate for the kernel
$$\widehat{K}(t) = \cF[\lambda](t) = \sum\lambda_i\cos(\omega_it).$$
Optimize real values $\{\omega_i, \lambda_i\}$ through gradient descent (Adam) and optionally project the $\lambda_i$ to be positive after each step, minimizing the least squares objective on the time samples $\{t_i\}_{i=1}^n$
$$\min_{\omega_i, \lambda_i}\|\widehat{K} - K\|_{\{t_i\}_{i=1}^n}$$
\end{enumerate}
\Return{$\widehat{\lambda} = \sum_{i=1}^n \tfrac{\lambda_i}{2}(\delta_{-\omega_i} + \delta_{\omega_i})$}
\caption{DCT-Based Inverse Fourier Transform of gPD Kernels}
\label{alg:dct_fft}
\end{algorithm}

\bigskip

\daveend

\subsection{Numerical Examples of \texorpdfstring{$\cB_\RR$}{BR}}
We develop an intuition for the behavior of $\cB_\RR$ \dave{(whose implementation is described in \cref{subsubsec:cauchy}) by applying it to four different triples $(\lambda,c_0,c_1)$}. These are as follows: a purely discrete measure,
\begin{equation*}
    \lambda_1(s) = \sum_{i=1}^6\beta_i\delta(s-\alpha_i), \qquad \begin{array}{l}\alpha = (5.0, 7.0, 8.1, 10.3, 12.2, 15.0)\\\beta = (1.0, 2.3, 0.5, 0.7, 2.0, 0.4)\end{array},\qquad \begin{array}{l}c_0 = -10\\c_1 = 5\end{array},
\end{equation*}
a sum of two parabolic kernels,
\begin{equation*}
    \lambda_2(s) = \chi_{[4, 6]}(s)(1 - (s-5)^2) + \chi_{[14, 16]}(s)(1 - (s-15)^2), \qquad \begin{array}{l}c_0 = 1\\c_1 = 0\end{array},
\end{equation*}
a fully-supported measure with both continuous and discrete parts,
\begin{equation*}
    \lambda_3(s) = e^{-|s-6|} + \sum_{i=1}^3\beta_i\delta(s-\alpha_i), \qquad \begin{array}{l}\alpha = (3.0, 5.0, 7.0)\\\beta = (0.3, 0.4, 0.2)\end{array},\qquad \begin{array}{l}c_0 = 0\\c_1 = 0\end{array},
\end{equation*}
and a sum of two triangular kernels and several atoms,
\begin{equation*}
\begin{gathered}
    \lambda_4(s) = \chi_{[-6, -4]}(s)(1 - |s+5|) + \chi_{[4, 6]}(s)(1 - |s-5|)+ \sum_{i=1}^3\beta_i\delta(s-\alpha_i),\\
    \qquad \begin{array}{l}\alpha = (-2, 0, 2)\\\beta = (1.2, 0.2, 1.3)\end{array},\qquad \begin{array}{l}c_0 = 1+i\\c_1 = 1\end{array}.
\end{gathered}
\end{equation*}

In \cref{fig:lambda_mu_fig}, we show how each of the measures $\lambda_i$ is mapped to the interconverted measure $\mu_i$ under $\cB_\RR$, with the parameters $c_0$ and $c_1$ as indicated. As evidenced by the example of $\lambda_1$, discrete measures are always mapped to discrete measures for real $c_0$, and the atoms of $\lambda$ and $\mu$ must interlace (see \cref{cor:discrete_formula}). \dave{This interlacing phenomenon is well-known in the context of materials science, where discrete measures correspond to materials with piecewise-constant microstructure~\cite{gross1968mathematical}; see \cref{sec:materials} for more details. The behavior of $\cB_\RR$ grows more interesting for continuous and mixed measures, which correspond to materials with more general microstructures.} As the example of $\lambda_2$ demonstrates, continuous measures with compact support are mapped to other measures with the same support, along with \dave{one or more} atoms added on (see \cref{prop:support}). The example of $\lambda_3$ shows that any measure with everywhere nonzero continuous density is mapped to a continuous measure with full support; this is implied by \cref{thm:main_formula}, as $N_0(\lambda)$ is empty. Finally, the example of $\lambda_4$ shows that the same is true for \emph{any} measure when $c_0$ has positive imaginary part (see \cref{prop:easyone}).

\begin{figure}
    \centering
    \includegraphics[width=\linewidth]{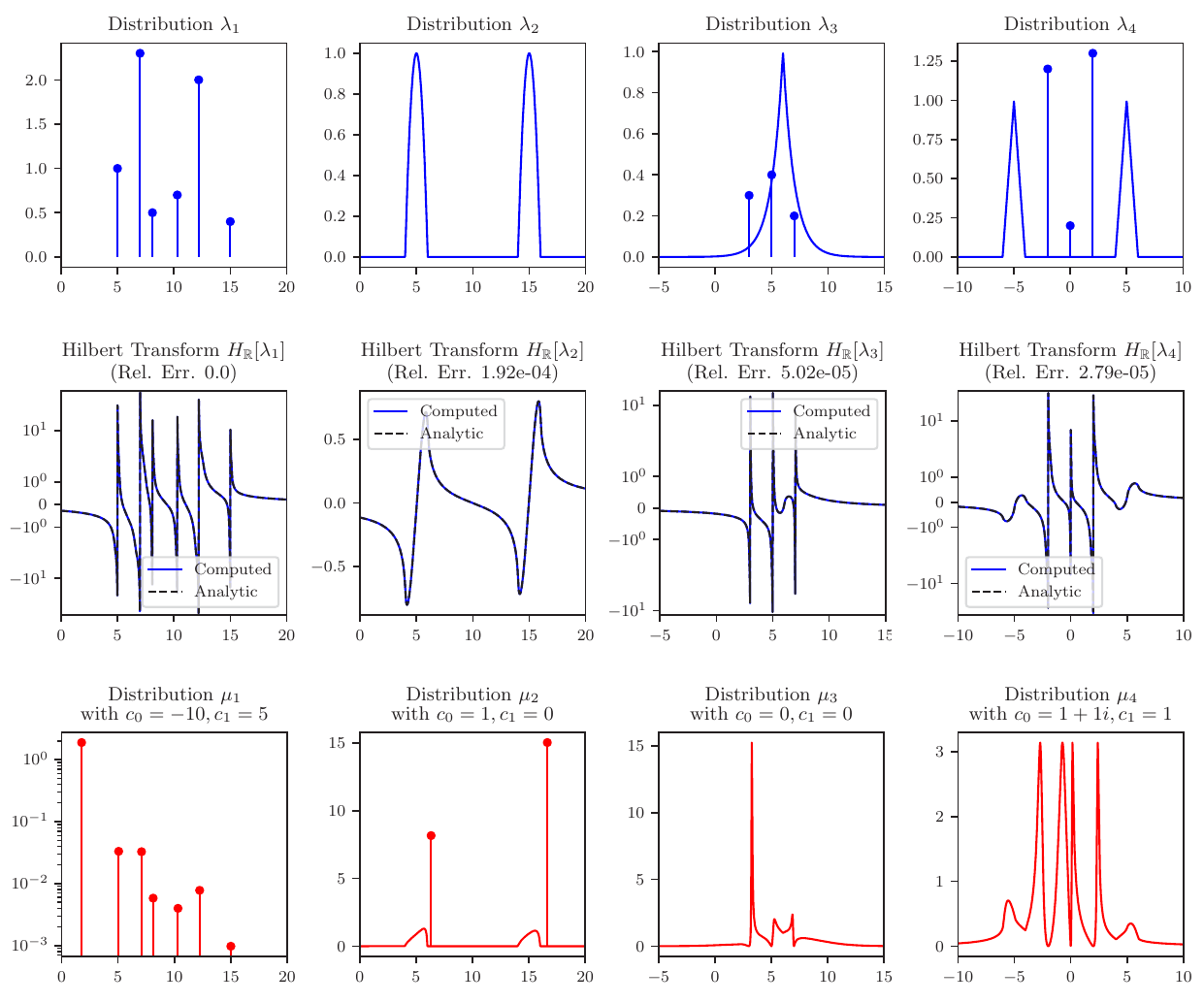}
    \caption[Examples of $\cB_\RR$ on the Real Line]{We display the diverse set of behaviors exhibited by the map $\cB_\RR$. The four examples shown above---discussed in \cref{sec:B_implem}---show a discrete measure, a continuous measure of compact support, a `mixed' measure with full support, and a mixed measure with compact support. Notably, the choice of $c_0$ and $c_1$ significantly affects the behavior of $\cB_\RR$.}
    \label{fig:lambda_mu_fig}
\end{figure}

We now apply our spectral method to solve the interconversion problem for the same four kernels\footnote{Since $c_0\notin\RR$ in the case corresponding to $\lambda_4$, note that the associated \ref{eq:integrodiff_CM} problem is not well-behaved. In that case, we treat only the \ref{eq:integrodiff_PD} problem.}.
\dave{We first consider the \ref{eq:integrodiff_CM} context, solved formally in \cref{prop:main_CM}}; here, $x(t)$ and $y(t)$ satisfy the pair of equations
\begin{equation*}
    \begin{gathered}
        y(t) = c_1\dot{x}(t) - c_0x(t) - \int_0^t K(t-s)x(s)ds,\\ -\pi^2x(t) = \zeta_1\dot{y}(t) - \zeta_0y(t) - \int_0^t J(t-s)y(s)ds,
    \end{gathered}
\end{equation*}
where $K = \cL_b[\lambda]$ and $J=\cL_b[\mu]$ are gCM integral kernels, with $\lambda,\mu\in\cM_+^{(1)}(\RR)$, and where $\cB_\RR[\lambda,c_0,c_1] = (\mu,\zeta_0,\zeta_1)$. By combining these two equations, we can see that the kernels $K$ and $J$ must satisfy the following \emph{resolvent equations}:
\begin{equation}\label{eq:resolvent_cm_detailed}
\begin{aligned}
    \zeta_1\dot{K} - \zeta_0K - K * J = 0, \ K(0) = \pi^2/\zeta_1, \quad &\text{if } c_1 = c_0 = 0,\\
     c_0J + \zeta_0K + K * J = 0,  \quad &\text{if } c_1 = 0, c_0 \neq 0,\\
    c_1\dot{J} - c_0J - K * J = 0, \ J(0) = \pi^2/c_1, \quad &\text{if } c_1 > 0.
\end{aligned}
\end{equation}
In order to evaluate the accuracy to which an estimated kernel $J$ satisfies the resolvent equations above, we can compute the relative $L^2$ error
\begin{equation}
    \cE_{\mathrm{gCM}}(J) = \begin{cases} \frac{\|\zeta_1\dot{K} - \zeta_0K - K * J\|_{L^2}}{\|\zeta_1\dot{K} - \zeta_0K\|_{L^2}}, & c_1 = c_0 = 0\\ \frac{\|\zeta_0K + c_0J + K * J\|_{L^2}}{\|\zeta_0K\|_{L^2}}, & c_1 = 0, c_0 \neq 0\\ \frac{\|c_1\dot{J} - c_0J - K * J\|_{L^2}}{\|c_1\dot{J}\|_{L^2}}, & c_1 > 0\end{cases}.
\end{equation}
The error in the first two cases is chosen to weigh against the total contribution of terms in the resolvent equation that do not involve $J$; since no such terms appear in the final case, the error is chosen to weigh against the `most irregular' expression in the resolvent equation, $c_1\dot{J}$.

\dave{We can carry out a similar program in the \ref{eq:integrodiff_PD} setting, solved formally in \cref{prop:main_PD}}; here, $x(t)$ and $y(t)$ satisfy the pair of equations
\begin{equation*}
    \begin{gathered}
        y(t) = c_1\dot{x}(t) - ic_0x(t) + \int_0^t K(t-s)x(s)ds,\\ \pi^2x(t) = \zeta_1\dot{y}(t) -i\zeta_0y(t) + \int_0^t J(t-s)y(s)ds,
    \end{gathered}
\end{equation*}
now with $K = \cF[\lambda]$ and $J=\cF[\mu]$ gPD integral kernels. The resolvent equations in this context are as follows:
\begin{equation}\label{eq:resolvent_pd_detailed}
    \begin{aligned}
    \zeta_1\dot{K} - i\zeta_0K + K * J = 0, \ K(0) = \pi^2/\zeta_1,\qquad &\text{if } c_1 = c_0 = 0\\
    i\zeta_0K = -ic_0J + K * J, \qquad &\text{if } c_1 = 0, c_0 \neq 0\\
    c_1\dot{J}- ic_0J + K * J = 0, \ J(0) = \pi^2/c_1 \qquad &\text{if } c_1 > 0.
\end{aligned}
\end{equation}
with the resulting relative $L^2$ error expression
\begin{equation}
    \cE_{\mathrm{gPD}}(J) = \begin{cases} \frac{\|i\zeta_0K - \zeta_1\dot{K} - K * J\|_{L^2}}{\|i\zeta_0K - \zeta_1\dot{K}\|_{L^2}}, & c_1 = c_0 = 0\\ \frac{\|i\zeta_0K + ic_0J - K * J\|_{L^2}}{\|i\zeta_0K\|_{L^2}}, & c_1 = 0, c_0 \neq 0\\ \frac{\|c_1\dot{J} - ic_0J + K * J\|_{L^2}}{\|c_1\dot{J}\|_{L^2}}, & c_1 > 0\end{cases}.
\end{equation}

\begin{figure}
    \centering
    \includegraphics[width=\linewidth]{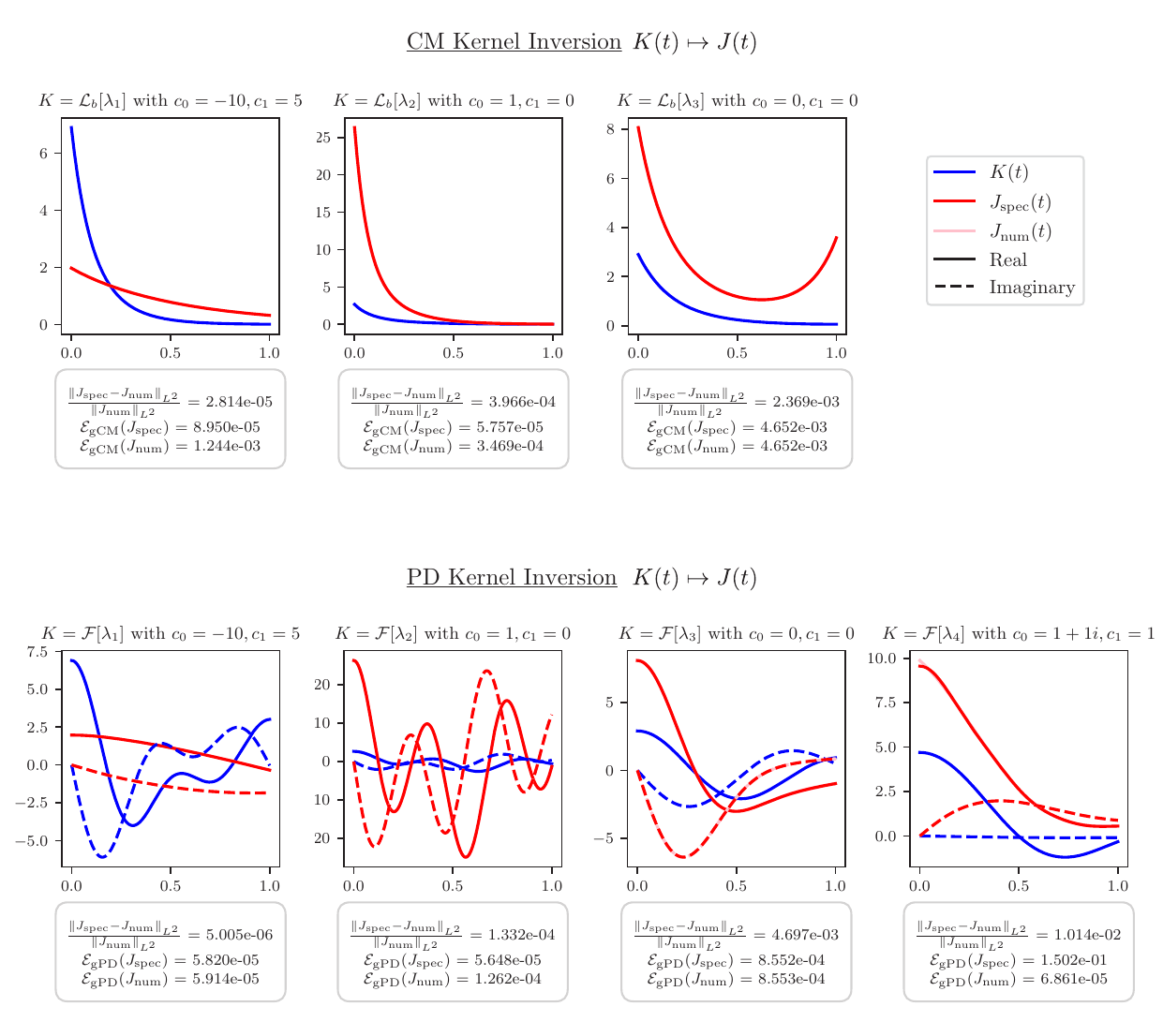}
    \caption[Interconversion of Volterra Integral and Integro-Differential Equations]{Interconversion of the equations \ref{eq:integrodiff_CM} and \ref{eq:integrodiff_PD} for various integral kernels $K$ (shown in blue). \dave{First, we implement the spectral approach introduced in \cref{sec:B_implem}, giving the interconverted kernels $J_\mathrm{spec}$ (red). Second, we implement direct numerical solutions to the resolvent equations \cref{eq:resolvent_cm_detailed} and \cref{eq:resolvent_pd_detailed} arising from our theory, giving the kernels $J_\mathrm{num}$ (pink) overlapping closely with $J_\mathrm{spec}$.}}
    \label{fig:kernel_fig}
\end{figure}

\dave{In \cref{fig:kernel_fig}, we study how spectral interconversion through $\cB_\RR$ compares against direct numerical approaches to the resolvent equations \cref{eq:resolvent_cm_detailed} and \cref{eq:resolvent_pd_detailed} derived from our spectral theory; for a comparison against existing techniques (i.e., those that do \emph{not} use our analytical results) see \cref{fig:trajectory_fig,fig:volterra_conv,fig:fft_comparison,fig:discrete} below.}

\dave{In both the CM and PD settings, our spectral interconversion algorithm starts with full spectral information $\lambda$ of the kernel $K$ along with the coefficients $(c_0, c_1)$, and computes the kernel $J$ and coefficients $(\zeta_0, \zeta_1)$ in the following two steps:
\begin{enumerate}
\item Apply numerical interconversion map
$$(\mu(z), \zeta_0, \zeta_1) = \cB_\RR[\lambda(z), c_0, c_1]$$
\item Construct inverse kernel through trapezoidal quadrature
$$J(t) = \cL_b[\mu](t) \text{ or } J(t) = \cF[\mu](t)$$
\end{enumerate}}

\dave{For comparison, we implement direct numerical solutions of the resolvent equations \cref{eq:resolvent_cm_detailed} and \cref{eq:resolvent_pd_detailed}. For integral equations of the first and second kinds, we discretize the integrals according to the trapezoid rule, with $10^4$ timepoints between $0$ and $1$; such an approach is discussed in~\cite[Ch.~18.2]{teukolsky1992numerical}. We solve integro-differential equations by approximating integral terms with Gauss quadrature (with 20 nodes), as discussed in~\cite{ansmann2018efficiently}. }

\subsection{Spectral Interconversion from Time-Sampled Kernels}\label{subsec:invlaplace}
In practice, one may not know the spectrum of our integral kernel $K$ \emph{a priori}, but only the values of $K$ at discrete time points $t_1,...,t_n$. To solve this problem with our spectral theory, one must first estimate the \dave{measure $\lambda$ associated with $K$}.

\dave{We first treat the \ref{eq:integrodiff_CM} case, using Algorithm~\ref{alg:aaa_laplace} to compute the inverse Laplace transform. At the top of \cref{fig:sample_fig}, we show how this scheme works on an integral equation of the first kind ($c_0=c_1=0$) with integral kernel
\begin{equation}\label{eq:Kcm_example}
    K(t) = \frac{1}{(t + 1)^2} + e^{-t}, \quad \lambda(s) = \cL^{-1}[K](s) = \chi_{[0,\infty)}(s)se^{-s}\,ds + \delta(s - 1)\,ds,
\end{equation}
with $n = 5$ and $n = 10$ logarithmically spaced time samples and 1000 steps of Adam gradient optimization~\cite{kingma2014adam}. Once we have estimated $\widehat{\lambda}\approx\cL_b^{-1}[K]$, we can use the methods of \cref{sec:B_implem} to compute $\cB_\RR[\widehat{\lambda},c_0=0,c_1=0] = (\widehat{\mu},\zeta_0,\zeta_1)$ and thus recover the interconverted gCM kernel $J = \cL_b[\widehat{\mu}]$. We find that we are able to accurately reconstruct $J$ with $n = 10$ samples; $n = 5$ samples are sufficient to closely approximate $K$ itself, but too few to accurately recover $J$}.

\dave{We also show the \ref{eq:integrodiff_PD} case, using Algorithm~\ref{alg:dct_fft} to compute the inverse Fourier transform. At the bottom of \cref{fig:sample_fig}, we apply this scheme to an integro-differential equation ($c_0=c_1=1$) with integral kernel
\begin{equation}\label{eq:Kpd_example}
    K(t) = e^{-t^2} + \frac{1}{4}\cos(2t) + \frac{1}{4}\cos(5t),
\end{equation}
\begin{align*}
    \lambda(s) = \cF^{-1}[\lambda](s) = \frac{1}{2\sqrt{\pi}}e^{-\frac{s^2}{4}}\,ds + \frac{1}{8}\Big[\delta(s - 5) + \delta(s - 1) + \delta(s + 1) + \delta(s + 5)\Big]\,ds,
\end{align*}
with $n = 10$ and $n = 20$ equispaced points in time and 1000 steps of Adam gradient optimization. We again use the methods of \cref{sec:B_implem} to compute $\cB_\RR[\widehat{\lambda},c_0=1,c_1=1] = (\widehat{\mu},\zeta_0,\zeta_1)$ and recover $J=\cF[\widehat{\mu}]$. We see that the reconstruction of the $J$ is highly accurate when $n = 20$ samples of $K$ are given, but still remains reasonably accurate even with $n = 10$ samples}.

\begin{figure}
    \centering
    \includegraphics[width=\linewidth]{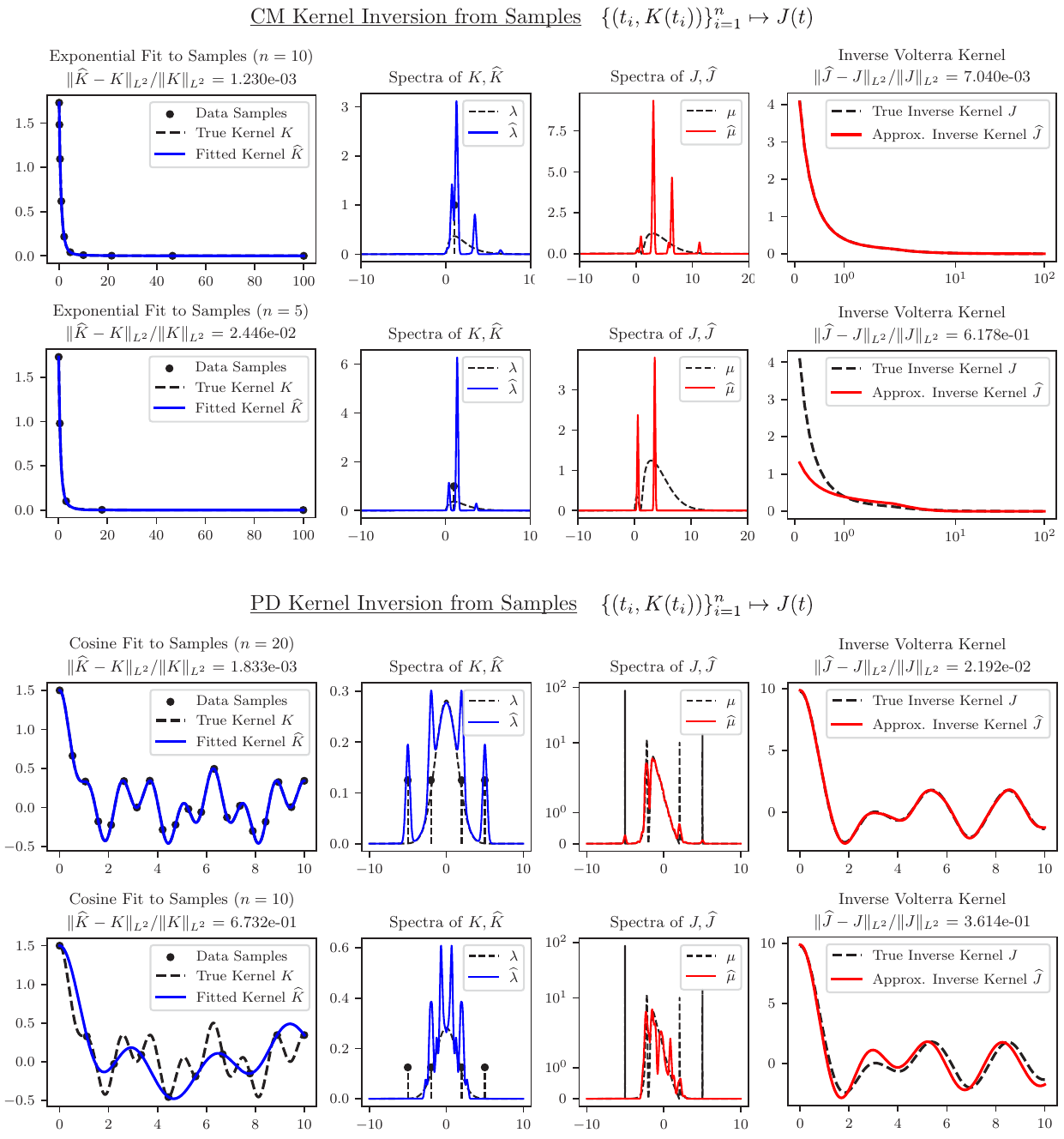}
    \caption[Interconversion from Time-Sampled Integral Kernels]{The first two rows above study the interconversion of a Volterra equation with $c_0 = c_1 = 0$ and the gCM Volterra kernel $K = \cL_b[\lambda]$ given by~\cref{eq:Kcm_example}. This interconversion is performed by first computing the approximate inverse Laplace transform of $K$ through AAA to obtain a discrete measure $\widehat{\lambda}$, which is then mapped to another discrete measure $\widehat{\mu}$ under $\cB_\RR$. Although $\widehat{\lambda}$ and $\widehat{\mu}$ in the middle two columns are exactly discrete, we plot them with kernel density estimation for clearer visualization. We see that the inversion accurately approximates the inverse kernel $J$ with only a moderate number of sample points. The second two rows study the interconversion of a Volterra equation with $c_0 = c_1 = 1$ and the gPD Volterra kernel $K = \cF[\lambda]$ given in~\eqref{eq:Kpd_example}. In this case, $\widehat{\lambda}$ is obtained by taking the inverse Fourier transform of $K$ computed through the DCT. Once again, the reconstruction of $J$ is accurate with only a moderate number of sample points.}
    \label{fig:sample_fig}
\end{figure}

\subsection{Deconvolution through Interconversion}\label{subsec:num_trajectories}
Now that we have established that our method can effectively interconvert either \ref{eq:integrodiff_CM} or \ref{eq:integrodiff_PD}, we test its ability to recover (or \textit{deconvolve}) the solution $x(t)$ from a \dave{noisy} input $y(t)$. 

We construct a random input $x(t)$ as follows. First, we generate a random walk $X_k = \frac{1}{\sqrt{N}}\sum_{i=1}^k\xi_i$, where $\xi_i\sim\cN(0,1)$ are i.i.d.~standard normal increments, normalized such that $\mathbb{E}[X_N^2]=1$. We define $x(t)$ by interpolating $X_k$ using a fifth-order spline, \dave{implemented in \texttt{SciPy}~\cite{2020SciPy-NMeth} as follows}:
\begin{equation}\label{eq:interp_rand_walk}
    x(t) = \mathtt{B\text{-}Spline}(\{(k\Delta\tau, X_k)\}_{k=1}^N).
\end{equation}
For the purposes of the present section, we generate such a trajectory $x(t)$ with timestep $\Delta \tau = 1$ and random walk steps $N = 10$.

We then study the following gCM Volterra equation
\begin{equation}\label{eq:cm_ex_invert}
    y(t) = 2x(t) - \int_0^tK(t-s)x(s)ds, \qquad K(t) = e^{-t} + e^{-2t},
\end{equation}
and the following gPD Volterra equation
\begin{equation}\label{eq:pd_ex_invert}
    y(t) = \int_0^tK(t-s)x(s)ds, \quad K(t) = \cos(t) + \cos(2t).
\end{equation}
For each equation, we numerically compute the convolution $K*x$ at 1000 time points to obtain a baseline value of $y(t)$. We then proceed to corrupt the resulting values of $y$ with $p\%$ Gaussian white noise to obtain $\widetilde{y}(t) = y(t) + \xi(t)$, where $\xi(t)$ is a Gaussian white noise process, scaled such that $\mathbb{E}[\|\xi\|_{L^2}]/\|y\|_{L^2} = \tfrac{p}{100}$.

\dave{In \cref{fig:trajectory_fig}, we apply the spectral approach above to recover $x(t)$ from noisy measurements $\widetilde{y}(t)$, and we denote this estimate by $\widehat{x}_\mathrm{spec}(t)$ (dark blue line). Specifically, we recover the interconverted kernel $J$ as before, but we now use the formulas of \cref{prop:main_CM} and \cref{prop:main_PD} to recover $\hat{x}_\text{spec}(t)$. This approach works remarkably well, even at high noise levels}. 

We compare against a baseline approach of inverting~\eqref{eq:cm_ex_invert} and~\eqref{eq:pd_ex_invert} directly, i.e., by discretizing these systems through a trapezoid rule at 1000 equispaced timepoints and solving the resulting matrix equation. This reconstruction, labeled $\widehat{x}_\mathrm{data}(t)$ (light purple line), shows high sensitivity to noise in both the gCM and gPD cases. \dave{As expected (and unlike the interconversion-based method), this approach is more noise-sensitive when applied to the first-kind equation \eqref{eq:pd_ex_invert} than to the second-kind equation~\eqref{eq:cm_ex_invert}}. 

\begin{figure}
    \centering
    \includegraphics[width=\linewidth]{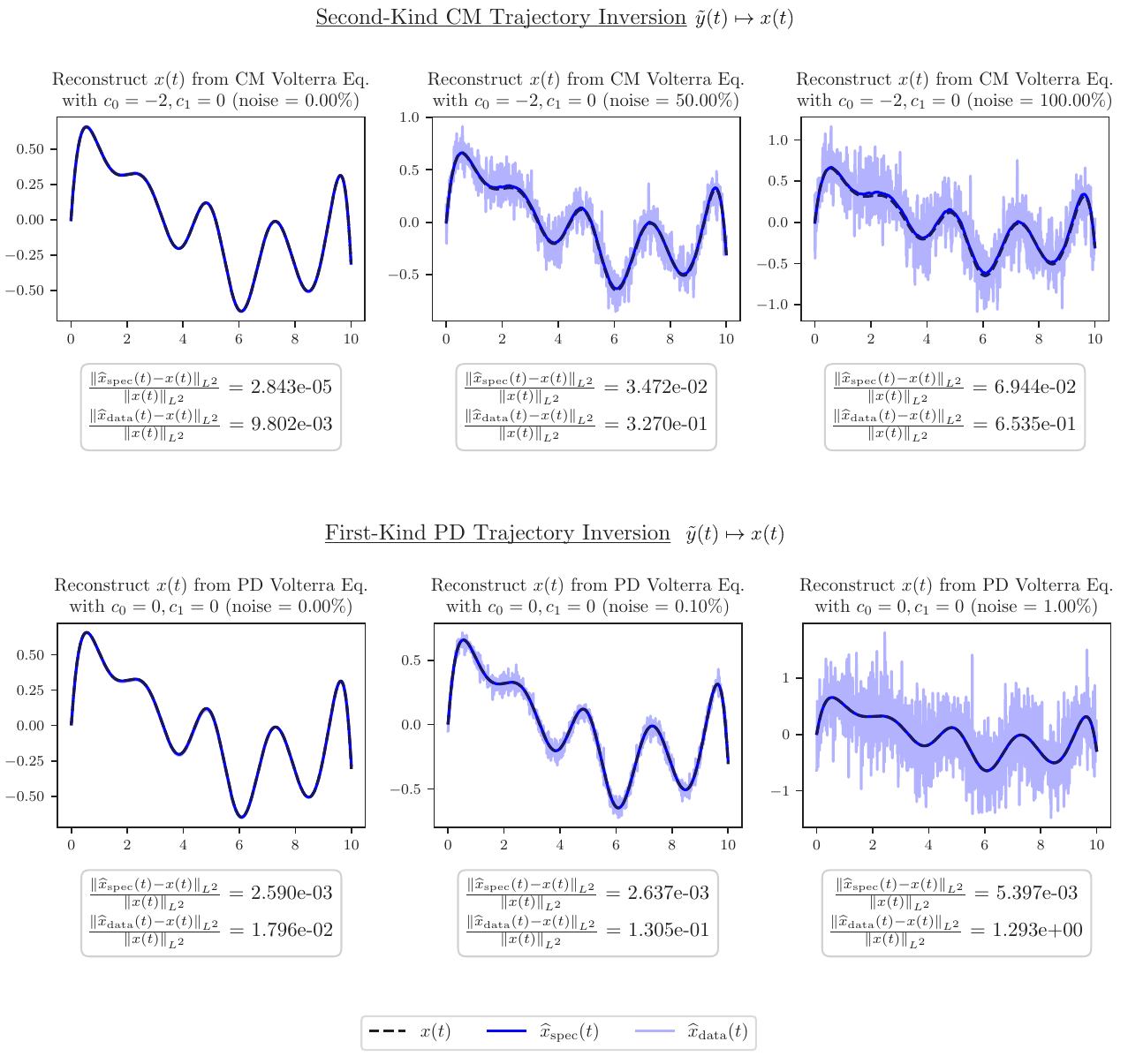}
    \caption[Comparison with Quadrature-Based Solution under Noise Corruption]{Reconstruction of the trajectory $x(t)$ given a noisy output $\widetilde{y}(t) = y(t) + \xi(t)$, in both gCM~\eqref{eq:cm_ex_invert} and gPD~\eqref{eq:pd_ex_invert} Volterra equations. With our spectral approach, we reconstruct $x(t)$ by first determining the interconverted Volterra kernel $J$ from $K$ through our spectral interconversion formulas, and then using $J$ to reconstruct $x(t)$ analytically. This reconstruction is shown as $\widehat{x}_\mathrm{spec}(t)$ (dark blue line), and we see that it is robust to significant noise corruption. An alternative approach for reconstructing $x(t)$ is to numerically solve the Volterra equations~\eqref{eq:cm_ex_invert} and~\eqref{eq:pd_ex_invert} through a trapezoid rule discretization. The resulting numerical problem is highly ill-conditioned, and as such, the reconstructed values $\widehat{x}_\mathrm{data}(t)$ (light purple line) are highly sensitive to noise.}
    \label{fig:trajectory_fig}
\end{figure}

\george{
In \cref{fig:volterra_conv}, we show how, even with no noise, solving a first-kind Volterra equation through spectral methods has superior convergence and time complexity as the length $N$ of the time series increases. We compare our method once again to the standard approach of inverting a large triangular linear system after trapezoid rule discretization of the first-kind equation.
}

\begin{figure}
    \centering
    \includegraphics[width=\linewidth]{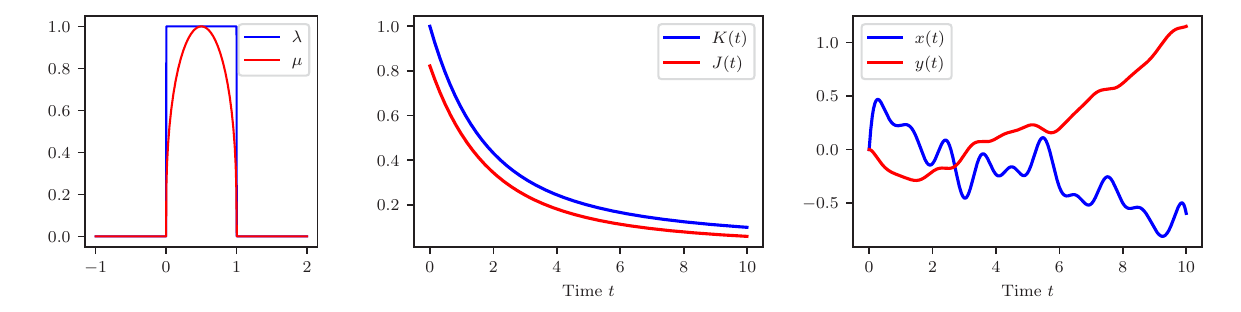}
    \begin{tabular}{ |p{3cm}||p{1.5cm}|p{1.5cm}|p{1.5cm}|p{1.5cm}|  }
     \hline
     \multicolumn{5}{|c|}{\textbf{Relative Error of Volterra Solution Methods} $\tfrac{\|\widehat{x} - x\|_{L^2}}{\|x\|_{L^2}}$} \\
     \hline
     Method & $N = 10$ & $N = 10^2$ & $N = 10^3$ & $N = 10^4$\\
     \hline
     Spectral (Ours) & 2.23e-01 & 2.73e-02 & 1.06e-03 & 5.72e-05\\
     \hline
     Linear Solver & 3.29e-01 & 6.30e-02 & 6.52e-03 & 6.70e-04\\
     \hline
    \end{tabular}

    \begin{tabular}{ |p{3cm}||p{1.5cm}|p{1.5cm}|p{1.5cm}|p{1.5cm}|  }
     \hline
     \multicolumn{5}{|c|}{\textbf{Runtime (ms) of Volterra Solution Methods}} \\
     \hline
     Method & $N = 10$ & $N = 10^2$ & $N = 10^3$ & $N = 10^4$\\
     \hline
     Spectral (Ours) & 0.093 & 0.077 & 0.278 & 0.759\\
     \hline
     Linear Solver & 0.024 & 0.029 & 0.695 & 55.536\\
     \hline
    \end{tabular}
    \caption[Comparison with Quadrature-Based Solution of First-Kind Equation]{\dave{Reconstruction of the trajectory $x(t)$ given clean measurements of $y(t) = (K * x)(t)$ on a gCM example with $c_0 = c_1 = 0$ and the kernel $K(t) = (1-e^{-t})/t$ shown above. Time series are discretized at $N = 10, \dots, 10^4$ equispaced time points. Our spectral approach converges faster as $N$ increases and also is more computationally efficient, compared to inversion of a linear triangular system via trapezoid rule discretization. Statistics are reported averaged over 1000 trials.}}
    \label{fig:volterra_conv}
\end{figure}

\subsection{Discrete-Time Volterra Equations}\label{sec:discrete_num}
\george{A fundamental problem of signals analysis is to deconvolve discrete time series that are filtered or smoothed by a one-sided kernel. As discussed in \cref{sec:signals}, prior approaches have predominantly focused on FFT or matrix inversion methods as the workhorse for numerical deconvolution.} \dave{In this section, we show how our interconversion theory allows us to achieve accuracy comparable with matrix inversion methods but efficiency competitive with FFT-based methods.}

We demonstrate our approach on the equation
\begin{equation}\label{eq:disc_x_to_y}
    y(n) = c_0x(n) + \sum_{j=0}^n K(n-j)x(j), \quad K(n) = \sum_{k=1}^N b_k\cos(n\theta_k)
\end{equation}
where $\theta_k \in [0, \pi]$ are distinct angles, $b_k\geq0$ are the corresponding weights, and $c_0 = -\tfrac{1}{2}\sum_{k=1}^N b_k$. The inverse Fourier transform of this kernel is
\begin{align*}
    d\lambda(\theta) = \sum_{k=1}^N\frac{b_k}{2}\Big(\delta(\theta - \theta_k) + \delta(\theta + \theta_k)\Big)\,d\theta,
\end{align*}
and by \cref{prop:main_PD}, we obtain $\mu\simeq (\mu, 0) = \cB[\lambda, 0]$. Because $\lambda$ is discrete, $\mu$ must take the form
\begin{align*}
    d\mu(\theta) = \sum_{k=1}^N\frac{\beta_k}{2}\Big(\delta(\theta - \gamma_k) + \delta(\theta + \gamma_k)\Big)\,d\theta,
\end{align*}
where $\gamma_k \in [0, \pi]$ interleave between the angles $\theta_k$ on the unit circle. In this case, we have
\begin{align*}
    J(t) = 4\cF[\mu](t) = 4\sum_{k=1}^N \beta_k\cos(\gamma_k t),\qquad \zeta_0 = - \tfrac{1}{2}J(0) = -2\sum_{k=1}^N \beta_k.
\end{align*}
Finally, the solution to \cref{eq:disc_x_to_y} is given by
\begin{equation}\label{eq:disc_y_to_x}
    x(n) = \zeta_0y(n) + \sum_{j=0}^n J(n-j)y(j).
\end{equation}

\begin{figure}
    \centering
    \includegraphics[width=\linewidth]{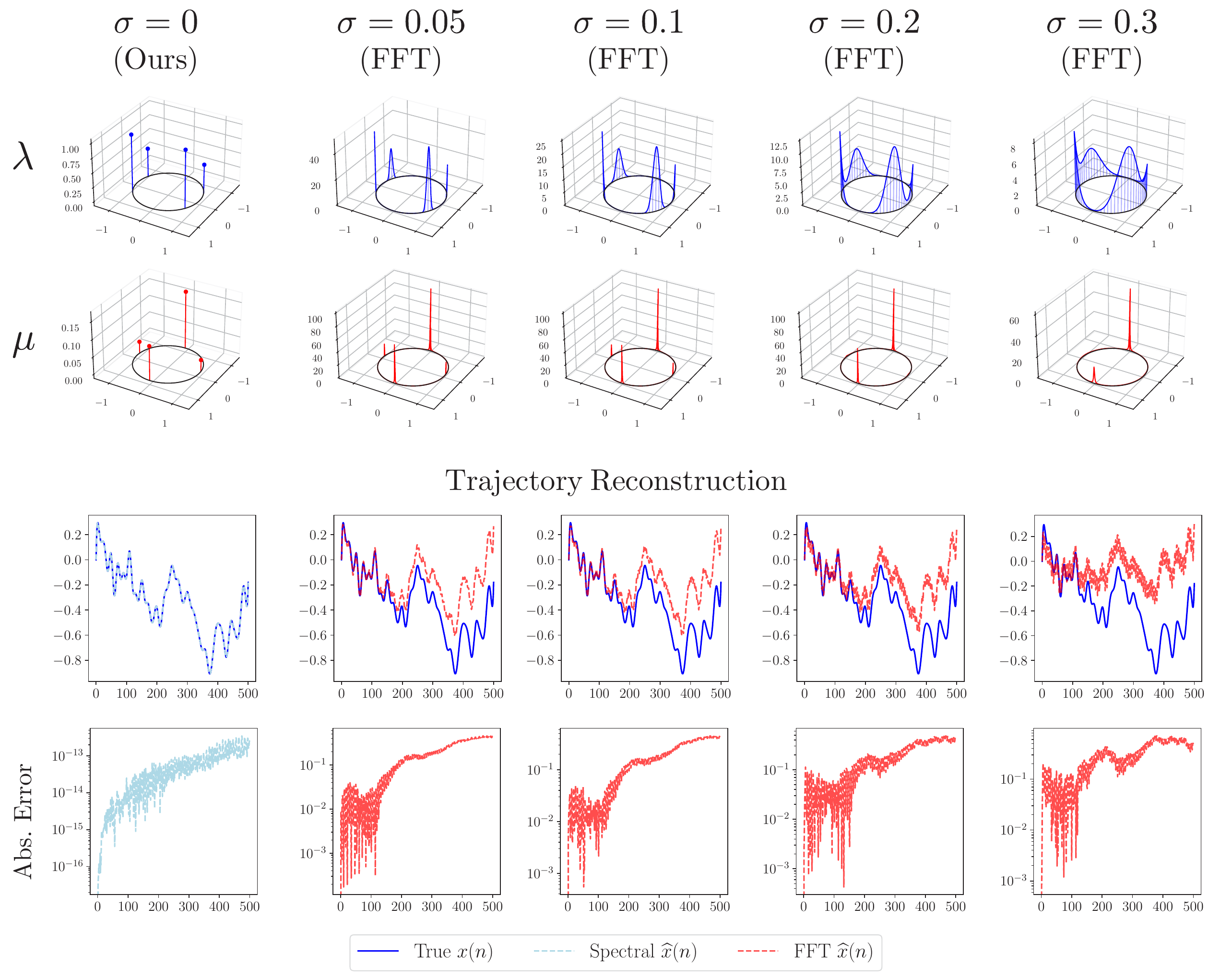}
    \caption[Comparison with FFT-Based Deconvolution of Difference Equations]{\dave{At the top of the first column, we show the discrete measure $\lambda\in\cM_+(S^1)$ associated with our discrete Volterra equation~\cref{eq:disc_x_to_y}. The interconverted measure $\mu$ (furnished by \Cref{prop:main_dPD}) is shown immediately below. We generate a random trajectory $x(n)$ on $n = 500$ points (blue curve) and convolve it against $c_0\delta_0 + \cF[\lambda]$ to produce an output $y(n)$. Our spectral approach is able to accurately reconstruct an approximation $\widehat{x}(n)\approx x(n)$ (blue dashed curve) using our interconversion formulas. In the second through fifth columns, we convolve $\lambda$ with a von Mises distribution of length scale $\sigma = 0.05$--$0.3$ and apply an FFT-based convolution. For all tested values of $\sigma$, the FFT gives order one error in its reconstruction $\widehat{x}(n)$ (red dashed curves).}}
    \label{fig:fft_comparison}
\end{figure}

In \cref{fig:fft_comparison}, we show a measure $\lambda$ (first column) with four atoms (i.e., $N = 2$) with $\theta_1, \theta_2 = 1, 2$ and $b_1, b_2 = 1, \tfrac{1}{2}$, respectively, and $c_0 = -\tfrac{1}{2}(b_1 + b_2) = -3/2$. We verify that our spectral map $\cB$ correctly inverts these dPD Volterra equations by testing it on a trajectory $x(n)$ defined as in~\eqref{eq:interp_rand_walk}, such that the random walk $X_k$ jumps every $\Delta n = 50$ discrete time intervals. We convolve $x(n)$ under~\eqref{eq:disc_x_to_y} to produce $y(n)$, and then deconvolve it under~\eqref{eq:disc_y_to_x} to reconstruct the trajectory $\widehat{x}(n)$. We see in \cref{fig:fft_comparison} (third row, first column) that this spectral reconstruction $\widehat{x}$ (blue dashed line) matches $x$ (blue line).

We compare this spectral approach to \dave{traditional} numerical solutions of \cref{eq:disc_x_to_y} in both the frequency and time domains. For this, recall from \cref{sec:signals} that the discrete deconvolution problem can be rephrased as solving a linear system. Forming the lower triangular matrix $\mathbf{T} \in \RR^{(n+1) \times (n+1)}$ with $T_{ij} = \mathbbm{1}_{\{i \geq j\}} c_0\delta(i-j) + K(i-j)$, we can solve
\begin{equation}
    \mathbf{y} = \mathbf{T}\mathbf{x}, \qquad \mathbf{x} = \begin{pmatrix}x(0)\\ \vdots\\ x(n)\end{pmatrix}, \ \mathbf{y} = \begin{pmatrix}y(0)\\ \vdots\\ y(n)\end{pmatrix}.
\end{equation}
The classical algorithm for this inversion uses forward substitution and requires $O(n^2)$ operations. However, the matrix $\mathbf{T}$ is \textit{Toeplitz} as well as triangular, so this scheme can be improved upon. Generic (i.e., non-triangular) Toeplitz matrices can be inverted in $O(n^2)$ operations using Levinson recursion~\cite{wiener1964wiener}, although relatively-involved \textit{superfast} methods exist that use the FFT to invert such matrices in $O(n\log n + np^2)$ operations, where $p$ depends on the entries of the Toeplitz matrix~\cite{chandrasekaran2008superfast}. Triangular Toeplitz matrices can likewise be inverted in $O(n \log n)$ time with $\sim 10$ applications of the FFT~\cite{commenges1984fast}.

\george{Approximate algorithms based on polynomial interpolation can drop the time complexity to the cost of only a few FFTs, but with stricter requirements on the regularity of the spectrum. The standard approach of this form~\cite{LIN2004511} is to take the Fourier transform of the first column of $\mathbf{T}$, which encodes $c_0\delta_0 + K$; compute the reciprocal of the Fourier coefficients (adding a small regularizing $\epsilon = 10^{-5}$); and evaluate the inverse FFT of the result to reconstruct $\zeta_0\delta_0 + J$, corresponding to the first column of $\mathbf{T}^{-1}$. We test this method on the example discussed above, where we produce a stochastic trajectory $x(n)$, convolve it against $c_0\delta_0 + K$ to produce $y(n)$, and use the FFT-based estimate of $\zeta_0\delta_0 + J$ to deconvolve and recover $\widehat{x}(n)$. The FFT method does not apply directly when $\lambda$ is an atomic measure, so we convolve it with a von Mises distribution as
$$\lambda(\theta) \mapsto \lambda(\theta) * \left[\frac{\exp(\cos(\theta)/\sigma^2)}{2\pi I_0(1/\sigma^2)}\right],$$
with varying length scales $\sigma$. In \cref{fig:fft_comparison} (second through fifth columns) we show how this improves the conditioning of deconvolution under the FFT, but its reconstruction $\widehat{x}(n)$ (red dashed line) still gives a poor estimate of $x(n)$ (blue line).} \dave{By contrast, our spectral approach yields near machine-precision for cases involving singular measures, without relying on regularization}.

\dave{In \cref{fig:discrete}, we compare the accuracy and efficiency of our spectral inversion to those of time-domain deconvolution, again using the example \cref{eq:disc_x_to_y}. We compare against two classical methods of inverting Toeplitz matrices: forward substitution for triangular matrices and Levinson recursion for Toeplitz matrices. }\george{Both approaches run in $O(n^2)$ time, where $n$ is the time series length. By contrast, we see that our spectral approach constructs $x$ in nearly linear time, suggesting that it is dominated by the computation of the FFT. All three approaches recover $x$ with comparable (near machine-precision) accuracy}.

\begin{figure}
    \centering
    \includegraphics[width=\linewidth]{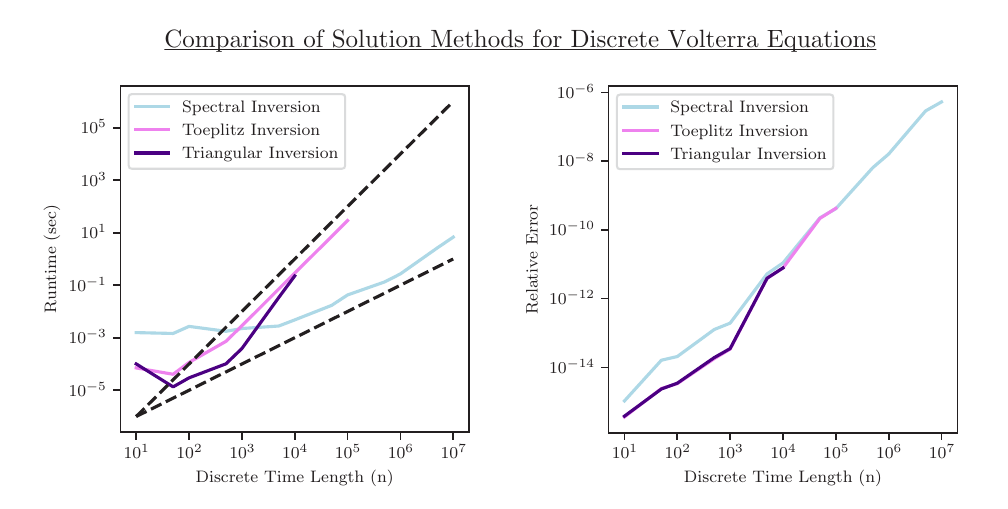}
    \caption[Comparison with Time-Domain Deconvolution of Difference Equations]{Comparison of three different methods for solving the discrete Volterra equation~\eqref{eq:disc_x_to_y}. The spectral method developed in this paper for inversion of discrete Volterra equations (light blue line) scales nearly linearly in $n$, suggesting that it is dominated by the two FFTs it performs. By contrast, Levinson recursion~\cite{wiener1964wiener} (pink line) and forward substitution (purple line) both scale quadratically with $n$. All methods have comparable relative (root squared) error in their reconstruction of $x$.}
    \label{fig:discrete}
\end{figure}

\subsection{Volterra Equations with Fractional Derivatives}\label{sec:fractional}
Next, we show how spectral interconversion allow us to solve Volterra equations with fractional derivatives. \george{As discussed in \cref{sec:history}, equations with fractional kernels are central in materials modeling, and are used to describe a variety of memory-dependent processes where long-term memory is present. Developing better approaches to solve these equations would enable better fractional models to simulate such processes.} We study the equation
\begin{equation}\label{eq:num_fracderiv}
    y(t) = \dot{x}(t) + D^{1/2}x(t) = \dot{x}(t) + \frac{d}{dt}\int_0^t K(t-\tau)x(\tau)\,d\tau,
\end{equation}
discussed in \cref{ex:fractional}, where $D^{1/2}$ is the Riemann--Liouville \dave{fractional} derivative~\cref{eq:frac_deriv}. Here, we have $K(t) = 1/\sqrt{\pi t}$, which can be represented as
\begin{equation}\label{eq:frac_lambda_laplace}
    K(t) = \int_{-\infty}^\infty \frac{e^{-ts}}{s}\lambda(s)ds, \qquad d\lambda(s) = \chi_{[0,\infty)}(s)\pi^{-1}\sqrt{s}\,ds.
\end{equation}
We note that $\lambda\notin\cM_+^{(1)}(\RR)$, so we require the machinery of \cref{sec:regularizedH} to solve this equation; recall from \cref{ex:fractional} that the solution takes the form
\[x(t) = \int_0^\infty E_{1/2}(-(t-\tau)^{1/2})y(\tau)\,d\tau,\]
where $E_{1/2}$ is the Mittag--Leffler kernel~\cite{haubold2011mittag}. In our notation, this corresponds to
\begin{align*}
    J(t) = \cL[\mu](t) = \pi^2E_\frac{1}{2}(-t^\frac{1}{2}), \qquad \mu(s) = \chi_{[0,\infty)}(s)\frac{\pi}{s^\frac{1}{2} + s^\frac{3}{2}}.
\end{align*}

In \cref{fig:fractional}, we compute the same result numerically, using the implementation of $\cB_\mathrm{reg}$ discussed in \cref{sec:B_implem}. We compare the result of our spectral interconversion against a direct implementation of the Mittag--Leffler kernel, using the \texttt{GenML} library in Python~\cite{qu2024genml}; we see that our spectral approach accurately captures both the kernel $J$ and its spectrum $\mu$ accurately, and that it is able to recover $x$ from a stochastic input $y$, generated using the same technique discussed in \cref{subsec:num_trajectories}.

\begin{figure}
    \centering
    \includegraphics[width=\linewidth]{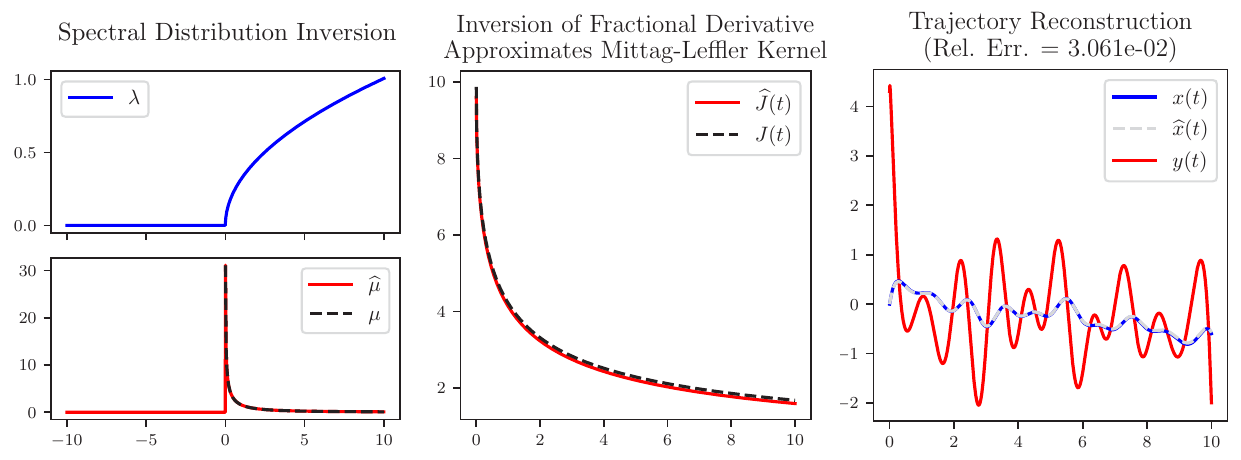}
    \caption[Interconversion of Fractional Differential Equations]{Interconversion of the fractional differential equation~\cref{eq:num_fracderiv} leads to a CM Volterra equation with a Mittag--Leffler kernel $J$, as \dave{furnished} by \cref{prop:rCM}. We see that our spectral approach accurately captures both $J$ and its spectrum $\mu=\cL^{-1}[J]$, and that it accurately recovers $x$ from a stochastic input $y$; the latter is generated using the technique discussed in \cref{subsec:num_trajectories}.}
    \label{fig:fractional}
\end{figure}

\subsection{Quantum Walks on Graphs}\label{sec:quantumwalk} \dave{Recall from \cref{sec:LTI} that the \ref{eq:integrodiff_PD} class can be seen to correspond to partially-observed quantum systems. Here, we investigate an example of practical interest: \emph{quantum walks}. Quantum walks are an analogue of the classical random walk and a basic element of many quantum algorithms---for instance, a quantum walk underlies Grover's search algorithm~\cite{Grover_2001}.}

Let $\cH$ be an $(n+1)$-dimensional complex Hilbert space, with a distinguished basis $|e\rangle,|0\rangle,|1\rangle,...,|n-1\rangle$. We consider the Hamiltonian
\[\hat{H} = \sum_{j=0}^{n-1}\Big(|j\rangle\langle j+1| + |j+1\rangle\langle j|\Big) + |0\rangle\langle e| + |e\rangle\langle 0 |,\]
writing $|n\rangle = |0\rangle$. Up to an affine transformation, $\hat{H}$ is the graph Laplacian of the $(n+1)$-vertex graph depicted in \cref{fig:quantum-search}; the numbered vertices $|0\rangle,...,|n-1\rangle$ form a cycle, and the remaining vertex $|e\rangle$ is attached only to $|0\rangle$. In this setting, we are interested in the following question: given a uniformly random initial state on the cycle, what is the probability $p=p(t)$ of measuring the particle in the state $|e\rangle$?

Let $\psi\in\cH$ be our time-evolving wavefunction. We write $\hat{P} = 1 - |e\rangle\langle e|$ for the projection onto the orthogonal complement of $|e\rangle$, and we decompose
\[\psi = \phi_0|e\rangle + \phi_1,\qquad \phi_0 = \langle e|\psi\rangle,\ \phi_1 = \hat{P}|\psi\rangle.\]
We write 
\[\hat{H}_1 = \hat{P}\hat{H}\hat{P} = \sum_{j=0}^{n-1}\Big(|j\rangle\langle j+1| + |j+1\rangle\langle j|\Big)\]
for the restriction of the Hamiltonian to the $n$-vertex cycle. As discussed in \cref{sec:LTI}, the value $\phi_0$ satisfies the integro-differential equation
\begin{equation}\label{eq:quantumID_1}
    \dot{\phi}_0(t) + \int_0^t\langle e|\hat{H}e^{-i\hat{H}_1(t-\tau)}\hat{H}|e\rangle\phi_0(\tau)\,d\tau = -i\langle e|\hat{H} e^{-i\hat{H}_1t}|\phi_1(t=0)\rangle.
\end{equation}
This equation can be simplified greatly; for one, it is clear that $\hat{H}|e\rangle = |0\rangle$, which allows us to restrict our analysis to the cycle $\cH\setminus\op{span}\{|e\rangle\}$. The restricted Hamiltonian $\hat{H}_1$ has eigenpairs
\[|E_k\rangle \doteq \frac{1}{\sqrt{n}}\sum_{j=0}^{n-1}e^{2\pi i jk/n}|j\rangle,\qquad E_k \doteq \langle E_k|\hat{H}|E_k\rangle = 2\cos(2\pi k/n),\]
for $k=0,...,n-1$. We find $|0\rangle = n^{-1/2}\sum |E_k\rangle$, and thus
\[\langle e|\hat{H}e^{-i\hat{H}_1t}\hat{H}|e\rangle = \langle 0| e^{-i\hat{H}_1t} |0\rangle = \frac{1}{n}\sum_{j=0}^{n-1} e^{-iE_k t}.\]
Moving now to the initial state, we write
\[|\psi(t=0)\rangle = \sum_{j=0}^{n-1}N_j|E_j\rangle,\]
where $(N_0,...,N_{n-1})\in\CC^n$ is uniformly distributed on the sphere $S^{2n-1}\subset\CC^n$. Since the basis $|E_j\rangle$ differs from $|j\rangle$ by a unitary transformation, this initial state also corresponds to a uniform distribution in position space. In any case, \cref{eq:quantumID_1} simplifies as
\begin{equation}\label{eq:phi0_walk}
    \dot{\phi}_0(t) + \frac{1}{n}\sum_{j=0}^{n-1}\int_0^t  e^{-2i\cos(2\pi j/n)(t-\tau)}\phi_0(\tau)\,d\tau = \frac{1}{i\sqrt{n}}\sum_{j=0}^{n-1} N_je^{-2i\cos(2\pi j/n)t}.
\end{equation}
Of course, the probability $p$ of the particle being measured at state $|e\rangle$ can be recovered as $p(t)=|\phi_0(t)|^2$. Even with stochastic forcing, such an equation can be solved exactly using \cref{prop:main_PD}. We show a solution in \cref{fig:quantum-search}, using 1000 independent initializations of the system. For instance, we see that, with $n=9$, we can maximize the 90th percentile curve of $p(t)$ by measuring at $t_c \approx 4.31$.

\begin{figure}
    \centering
    \begin{subfigure}{}
    \begin{tikzpicture}
    \begin{scope}[xshift=12cm]
    \GraphInit[vstyle=Art] 
    \SetGraphArtColor{black!50}{darkgray}
    \grCycle[prefix=a,RA=2/sin(60)]{9}
    \tikzset{AssignStyle/.append style = {below=6pt}}
    \SetGraphArtColor{red}{darkgray}
    \grEmptyCycle[prefix=b,RA=3]{1}
    \AssignVertexLabel[color = red,%
    size = \footnotesize]{b}{$\mathbf{|e\rangle}$}
    \end{scope}
    \Edges[color=darkgray](a0,b0)
    \end{tikzpicture}
    \end{subfigure}
    \begin{subfigure}{}
    \includegraphics[width=\linewidth/2]{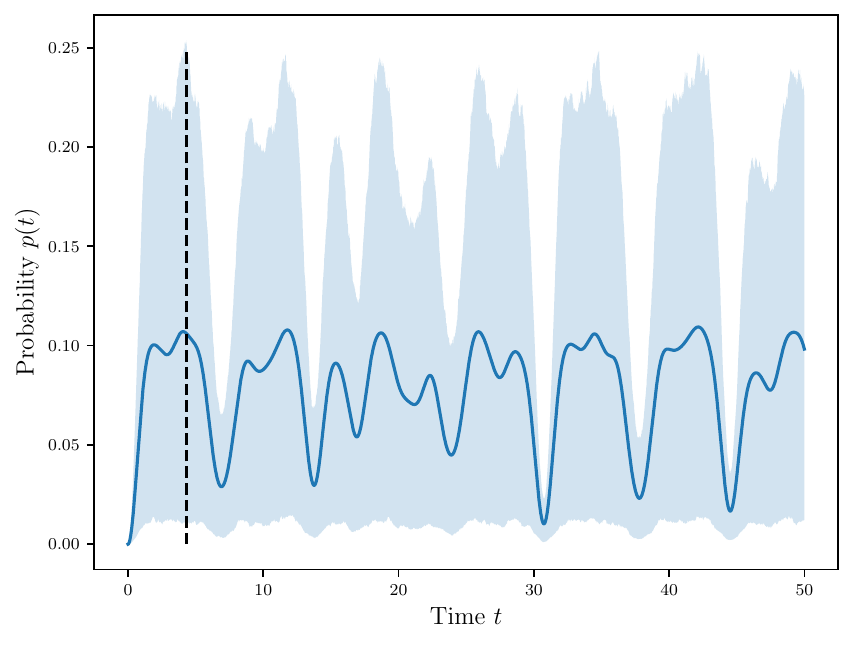}
    \end{subfigure}
    \caption[Marginal Probability of Quantum Walk Search]{Quantum walk on a graph with $n+1$ vertices (here, $n = 9$), with an initial state uniformly distributed on the $n$-point cycle. The component $\phi_0(t)$ of the wavefunction at $|e\rangle$ evolves according to \cref{eq:phi0_walk}, which is an integro-differential equation of the class \ref{eq:integrodiff_PD}. Using \cref{prop:main_PD}, then, we can explicitly recover the probability $p(t) = |\phi_0(t)|^2$ that the particle is measured at $|e\rangle$; maximizing this value is critical to quantum search algorithms. We carry this procedure out for 1000 independent initializations of the system, and report the mean value of $p(t)$ (dark blue line) and 10th/90th percentiles (light blue area) in the right-hand plot. The 90th percentile curve is maximized at time $t_c \approx 4.31$.}
    \label{fig:quantum-search}
\end{figure}

\begin{figure}
    \centering
    \includegraphics[width=0.75\linewidth]{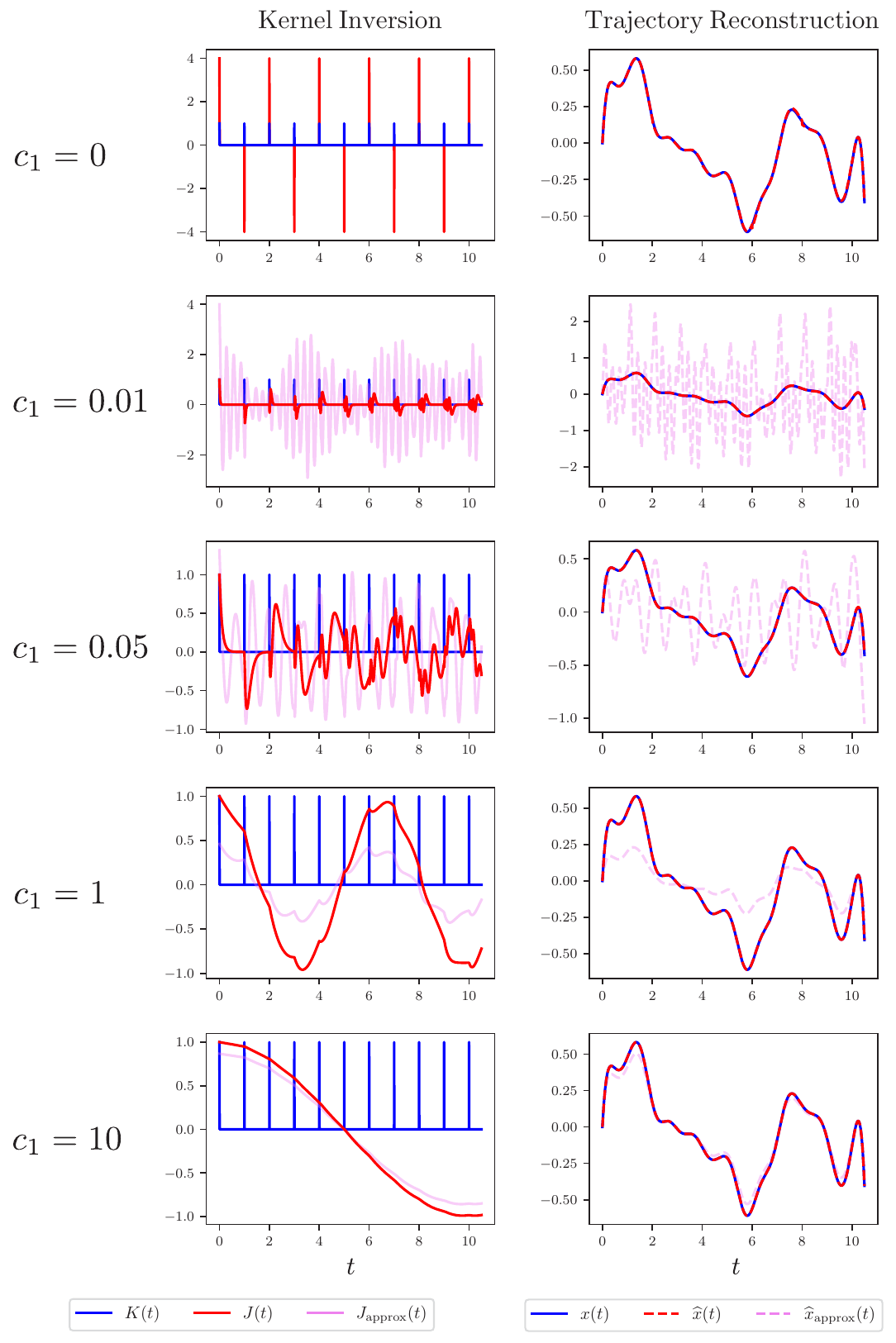}
    \caption[Delay Differential Equation with Infinitely Many Delays]{Asymptotic and numerical solutions to \cref{eq:diraccomb}, with $x_0=0$ and with various choices of $c_1>0$. The top left plot shows the Dirac comb $K(t)$ (blue line) along with its inverse kernel $J(t)$ for $c_1 = 0$ (red line). For the top right plot, we generate a stochastic trajectory $x(t)$, which we compare against the reconstructed solution $\widehat{x}\approx x$ given by our spectral approach. For different values of $c_1 > 0$, our spectral approach allows us to compute both $J(t)$ and $\widehat{x}(t)$ to a high degree of accuracy. We also construct an approximate kernel $J_\mathrm{approx}(t)$ (pink line) for each $c_1$ using the asymptotic estimates derived in \cref{sec:diraccomb}, and denote its reconstruction by $\widehat{x}_\mathrm{approx}$. As expected, these estimates converge as $c_1 \to \infty$. All kernels are rescaled by $c_1/\pi^2$ for clarity.}
    \label{fig:dirac_comb}
\end{figure}

\subsection{Delay Differential Equations with Infinitely Many Delays}\label{sec:diraccomb}
    Consider the equation
    \begin{equation}\label{eq:diraccomb}
        y(t) = c_1\dot{x}(t) + \frac{1}{2}x(t) + x(t-1) + \cdots + x(t-\lfloor t\rfloor),\qquad x(0) = x_0.
    \end{equation}
    As discussed in \cref{ex:easyvolterra}, equations of this form arise in approximating Volterra integro-differential equations with smooth integral kernels. On the other hand, \cref{eq:diraccomb} is itself in the class \ref{eq:integrodiff_rPD}, with
    \dave{\[K(t) = \sum_{k\in\ZZ}\delta(t-k)\,dt = \cF[\lambda](t),\qquad d\lambda(s) = \sum_{k\in\ZZ}\delta(s-2\pi k)\,ds,\]
    a kernel known as the \emph{Dirac comb}. The measure $\lambda$ satisfies}
    \[H_\mathrm{reg}[\lambda](s) = \frac{1}{2\pi}\cot(s/2).\]
    Now, it is easy to verify that $\lambda$ satisfies the criteria of \cref{thm:main_regularizedH_formula}. Indeed, $Z'$ clearly has no limit points away from $-1\in S^1$; but the density of $\psi[\lambda]$ at $-1\in S^1$ is nonzero, so $-1\notin N_0(\widetilde{\lambda})$ for any $c_1 \geq 0$.

    If $c_1=0$, the support of $\mu$ is exactly the zero set $2\pi(\ZZ+\frac{1}{2})$ of $H_\mathrm{reg}[\lambda]$, and the weight of each atom in $\mu$ is
    \[\beta = \pi^2\Bigg(\sum_{j\in\ZZ}\frac{1}{(2\pi)^2(j - 1/2)^2}\Bigg)^{-1} = 4\pi^2.\]
    This implies that $\zeta_1=0$ and that
    \[d\mu(s) = 4\pi^2\sum_{j\in\ZZ}\delta(s-2\pi j-\pi
    )\,ds,\qquad J(t) = \cF[\mu](t) = 4\pi^2\sum_{j\in\ZZ} (-1)^j\delta(s-j)\,ds,\]
    so we find
    \[x(t) = 2y(t) - 4y(t-1) + 4y(t-2) \pm\cdots \pm 4y(t-\lfloor t\rfloor).\]
    This solution is shown in \cref{fig:dirac_comb}. We can handle the integro-differential case similarly; if $c_1\neq 0$, the support of $\mu$ is the set
    \[Z=\{t\in\RR\;|\;\cot(t/2) = 2c_1t\},\]
    or, asymptotically (in the limit $c_1\to\infty$),
    \[\alpha_{0,\pm} = \pm \frac{1}{\sqrt{c_1}} + O(c_1^{-3/2}),\qquad \alpha_j = 2\pi j + \frac{1}{2\pi j c_1} + O(j^{-2}c_1^{-2}) \;\text{for}\;j\neq 0,\]
    with corresponding weights
    \[\beta_{0,\pm} = \frac{\pi^2}{2c_1+\pi^2/3} + O(c_1^{-3}),\qquad \beta_{j} = \frac{\pi^2}{4\pi^2j^2c_1^2+c_1} + O(j^{-4}c_1^{-4})\;\text{for}\;j\neq 0.\]
    We can test the accuracy of this approximation by evaluating
    \[(Q_\mathrm{reg}[\lambda_\mathrm{approx}](z) + \pi^{-1}c_1)Q_\mathrm{reg}[\mu_\mathrm{approx}](z) = 1 + \eps(c_1)\]
    at $z=+i$. With $c_1=2$, we find $|\eps|\approx 0.09$; with $c_1=4$, we find $|\eps|\approx 0.013$. We show the true and approximate inverse kernels $J = \cF[\mu]$ and $J_\mathrm{approx} = \cF[\mu_\mathrm{approx}]$ for several values of $c_1$ in \cref{fig:dirac_comb}. As expected, the asymptotic estimate converges as $c_1$ increases, giving a close estimate for $c_1=10$. This example highlights how our theory can yield significant analytic insight into the solutions of Volterra equations, even when they do not admit \dave{clean} analytic expressions.